\documentclass[11pt,leqno]{article}
\usepackage{amsmath,amsthm,amsfonts,amssymb}

\newcommand{\N}{{\mathbb N}}
\newcommand{\Z}{{\mathbb Z}}
\newcommand{\E}{{\mathbb E}}
\newcommand{\dt}{{\partial_t}}
\newcommand{\dxun}{{\partial_{x_1}}}
\newcommand{\dxde}{{\partial_{x_2}}}
\newcommand{\eps}{{\varepsilon}}
\newcommand{\tauetabar}{{(\underline{\tau},\underline{\eta})}}
\newcommand{\bchi}{\text{\boldmath{$\chi$}}}

\newcommand{\C}{{\mathbb C}}
\newcommand{\R}{{\mathbb R}}

\newcommand\cA{{\cal  A}}
\newcommand\cB{{\cal  B}}
\newcommand\cD{{\cal  D}}
\newcommand\cU{{\cal  U}}
\newcommand\cH{{\cal  H}}
\newcommand\cV{{\cal  V}}

\newcommand\cC{{\cal  C}}
\newcommand\cI{{\cal I}}
\newcommand\cJ{{\cal J}}
\newcommand\cR{{\cal  R}}
\newcommand\cG{{\cal  G}}

\newcommand\cL{{\cal  L}}

\newcommand\cF{{\cal  F}}

\newcommand\cO{{\cal O}}

\newcommand\cM{{\mathcal M}}

\newcommand\cS{{\mathcal S}}

\newcommand{\bB}{\mathbb{B}}
\newcommand{\bF}{\mathbb{F}}
\newcommand{\bR}{\mathbb{R}}

\newcommand{\bL}{\mathbb{L}}

\newcommand{\bE}{\mathbb{E}}

\newcommand{\bA}{\mathbb{A}}

\newcommand{\bN}{\mathbb{N}}
\newcommand{\bT}{\mathbb{T}}

\newcommand{\bD}{\mathbb{D}}
\newcommand{\bC}{\mathbb{C}}
\newcommand{\bZ}{\mathbb{Z}}

\newcommand{\tv}{\tilde v}

\newcommand{\tw}{\tilde{w}}

\newcommand{\uv}{{\underline v}}

\newcommand{\utau}{\underline \tau}

\newcommand{\ueta}{\underline \eta}
\newcommand{\uzeta}{\underline \zeta}

\newcommand{\uomega}{{\underline \omega}}

\newcommand\adots{\mathinner{\mkern2mu\raise1pt\hbox{.}
\mkern3mu\raise4pt\hbox{.}\mkern1mu\raise7pt\hbox{.}}}

\newcommand{\re}{{\rm Re }\, }

\newcommand{\bfS}{\textbf{S}}

\newtheorem{theo}{Theorem}[section]
\newtheorem{prop}[theo]{Proposition}
\newtheorem{cor}[theo]{Corollary}
\newtheorem{lem}[theo]{Lemma}
\newtheorem{defn}[theo]{Definition}

\newtheorem{rem}[theo]{Remark}

\newtheorem{nota}[theo]{Notations}
\newtheorem{assumption}[theo]{Assumption}
\newtheorem{definition}[theo]{Definition}

\numberwithin{equation}{section}
 \title{Semilinear geometric optics with boundary amplification}

\author{\sc \small
Jean-Francois Coulombel\thanks{CNRS and Universit\'e de Nantes, Laboratoire de math\'ematiques Jean Leray
(UMR CNRS 6629), 2 rue de la Houssini\`ere, BP 92208, 44322 Nantes Cedex 3, France;
{\tt jean-francois.coulombel@univ-nantes.fr}; Research of J.-F. C. was supported by the French Agence Nationale
de la Recherche, contract ANR-08-JCJC-0132-01.},
Olivier Gu\`es\thanks{Universit\'e de Provence, Laboratoire d'Analyse, Topologie et Probabilit\'es (UMR CNRS 6632),
Technop\^ole Ch\^ateau-Gombert, 39 rue F. Joliot Curie, 13453 Marseille Cedex 13, France; {\tt gues@cmi.univ-mrs.fr};
Research of O.G. was partially supported by the French Agence Nationale de la Recherche, contract
ANR-08-JCJC-0132-01.},
Mark Williams\thanks{University of North Carolina; Mathematics Department, CB 3250, Phillips Hall, Chapel Hill,
NC 27599. USA; {\tt williams@email.unc.edu}; Research of M.W. was partially supported by NSF grants number
DMS-0701201 and DMS-1001616.}}

\begin{document}

\maketitle
\begin{abstract}
We study weakly stable semilinear hyperbolic boundary value problems with highly oscillatory data.  Here weak
stability means that exponentially growing modes are absent, but the so-called uniform Lopatinskii condition fails
at some boundary frequency $\beta$ in the hyperbolic region. As a consequence of this degeneracy there is an
amplification phenomenon: outgoing waves of amplitude $O(\eps^2)$ and wavelength $\eps$ give rise to reflected
waves of amplitude $O(\eps)$, so the overall solution has amplitude $O(\eps)$.  Moreover, the reflecting waves
emanate from a radiating wave that propagates in the boundary along a characteristic of the Lopatinskii determinant.

An approximate solution that displays the qualitative behavior just described is constructed by solving suitable
profile equations that exhibit a loss of derivatives, so we solve the profile equations by a Nash-Moser iteration.
The exact solution is constructed by solving an associated singular problem involving singular derivatives of the
form $\partial_{x'}+\beta\frac{\partial_{\theta_0}}{\eps}$, $x'$ being the tangential variables with respect to the
boundary. Tame estimates for the linearization of that problem are proved using a  first-order calculus of singular
pseudodifferential operators constructed in the companion article \cite{CGW2}. These estimates exhibit
a loss of one singular derivative and force us to construct the exact solution by a separate Nash-Moser iteration.
The same estimates are used in the error analysis, which shows that the exact and approximate solutions are
close in $L^\infty$ on a fixed time interval independent of the (small) wavelength $\eps$. The approach using
singular systems allows us to avoid constructing high order expansions and making small divisor assumptions.
\end{abstract}

\tableofcontents

\section{Introduction and main results}
\label{intro}

\emph{\quad}In this paper we study weakly stable semilinear hyperbolic boundary value problems with oscillatory
data. The problems are weakly stable in the sense that exponentially growing modes are absent, but the uniform
Lopatinskii condition fails at a boundary frequency $\beta$ in the hyperbolic region $\cH$\footnote{See Definition
\ref{def1} and Assumption \ref{assumption3} below for precise statements.}. As a consequence of this degeneracy
in the boundary conditions, there is an amplification phenomenon: boundary data of wavelength $\eps$ and
amplitude $O(\eps^2)$ in the problem \eqref{0} below gives rise to a response of amplitude $O(\eps)$. In the
meantime, resonance may occur between distinct oscillations. In the situation studied below, a resonant quadratic
interaction between two incoming  waves of amplitude $O(\eps)$ may produce an outgoing wave of amplitude
$O(\eps^2)$. When reflected and amplified on the boundary, this oscillation gives rise to incoming waves of
amplitude $O(\eps)$. Hence the $O(\eps)$ amplitude regime appears as the natural weakly nonlinear regime.

Let us now introduce some notation. On $\overline {\mathbb{R}}^{d+1}_+ = \{x=(x',x_d)=(t,y,x_d)=(t,x''):x_d\geq 0\}$,
consider the $N\times N$ semilinear hyperbolic boundary problem  for $v=v_\eps(x)$, where $\eps>0$\footnote{We
usually suppress the subscript $\eps$.}:
\begin{align}\label{0}
\begin{split}
&(a)\;L_0(\partial)v+f_0(v)=0,\\
&(b)\;\phi(v)=\eps^2 \, G \left( x',\frac{x'\cdot\beta}{\eps} \right) \text{ on }x_d=0,\\
&(c)\;v=0 \text{ and }G=0\text{ in }t<0,
\end{split}
\end{align}
where $L_0(\partial)=\partial_t+\sum^{d}_{j=1}B_j \, \partial_j$, the matrix $B_d$ is invertible, and both $f_0(v)$ and
$\phi(v)$ vanish at $v=0$.  The function $G(x',\theta_0)$ is assumed to be periodic in $\theta_0$, and the frequency
$\beta\in\bR^{d}\setminus \{0\}$ is taken to be a boundary frequency at which the so-called uniform Lopatinskii
condition fails. A consequence of this failure is that the choice of the factor $\eps^2$ in \eqref{0}(b) corresponds to
the weakly nonlinear regime for this problem.  The leading profile is nonlinearly coupled to the next order profile in
the nonlinear system \eqref{b8s},\eqref{6} derived below. We also refer to Appendix \ref{exeuler} for a detailed
specific example which illustrates the nonlinear feature of the leading profile equation.

Before proceeding we write the problem in an equivalent form that is better adapted to the boundary.
After multiplying \eqref{0}(a) by $(B_d)^{-1}$ we obtain
\begin{align}\label{1}
\begin{split}
&L(\partial)v+f(v)=0,\\
&\phi(v)=\eps^2 \, G \left( x',\frac{x'\cdot\beta}{\eps} \right) \text{ on }x_d=0,\\
&v=0 \text{ and }G=0\text{ in }t<0,
\end{split}
\end{align}
where we have set
\begin{equation*}%\label{L}
L(\partial)=\partial_d+\sum^{d-1}_{j=0}A_j \, \partial_j \quad \text{ with } A_j :=B_d^{-1}B_j\text{ for }j=0,\dots,d-1.
 \end{equation*}
Setting $v=\eps u$  and writing  $f(v)=D(v)v$, $\phi(v)=\psi(v)v$, we get the problem for $u=u_\eps(x)$
\begin{align}\label{2}
\begin{split}
&(a)\;L(\partial)u+D(\eps u)u=0,\\
&(b)\;\psi(\eps u)u=\eps \, G \left( x',\frac{x'\cdot\beta}{\eps} \right)\text{ on }x_d=0,\\
&(c)\;u=0 \text{ in }t<0.
\end{split}
\end{align}
For the problem \eqref{2} we pose the two basic questions of rigorous nonlinear geometric optics:\\

(1)\; Does an exact solution $u_\eps$ of \eqref{2} exist for $\eps\in (0,1]$ on a fixed time interval $[0,T_0]$
independent of $\eps$?

(2)\; Suppose the answer to the first question is yes. If we let $u^{app}_\eps$ denote an approximate
solution on $[0,T_0]$ constructed by the methods of nonlinear geometric optics (that is, solving eikonal
equations for phases and suitable transport equations for profiles), how well does $u^{app}_\eps$
approximate $u_\eps$ for $\eps$ small? For example, is it true that\footnote{Let us observe that by the
amplification phenomenon, we expect the solution $v$ to \eqref{0} to have amplitude $O(\eps)$, so the
solution $u$ to \eqref{2} should have amplitude $O(1)$. Hence the limit \eqref{2z} deals with the difference
between two $O(1)$ quantities.}
\begin{equation}\label{2z}
\lim_{\eps\to 0}|u_\eps-u^{app}_\eps|_{L^\infty}\to 0 \, ?
\end{equation}

The amplification phenomenon was studied in a formal way for several different \emph{quasilinear} problems
in \cite{AM,MA,MR}.  In \cite{MR} amplification was studied in connection with Mach stem formation in reacting
shock fronts, while \cite{AM} explored a connection to the formation of instabilities in compressible vortex
sheets. Both papers derived equations for profiles using an ansatz that exhibited amplification; however,
neither of the two questions posed above were addressed. The first rigorous amplification results were proved
in \cite{CG} for \emph{linear} problems. That article provided positive answers to the above questions (question
$(1)$ is trivial for linear problems) by making use of approximate solutions of high order, and showed in particular
that the limit \eqref{2z} holds.

In this paper we give positive  answers to the above questions for the \emph{semilinear} system \eqref{2}.  As is
typical in nonlinear geometric optics problems involving several phases, difficulties with small divisors rule out the
construction of high order approximate solutions\footnote{Such difficulties are sometimes avoided by assuming
that small divisors do not occur, see e.g. \cite{JMR1}, but we do not want to make this assumption.}. Instead of
constructing the exact solution $u_\eps$ as a small perturbation of a high-order approximate solution, we
construct $u_\eps$  in the form
\begin{equation*}
u_\eps(x)=U_\eps(x,\theta_0)|_{\theta_0=\frac{\beta\cdot x'}{\eps}},
\end{equation*}
where $U_\eps(x,\theta_0)$ is an exact solution of the singular system \eqref{15}. The singular system is solved
using symmetrization and diagonalization arguments from \cite{W1}, modified and supplemented with methods
from \cite{C1} for deriving linear estimates for weakly stable hyperbolic boundary problems. In deriving the
basic estimate \eqref{aprioriL2} for the singular linear problem, a loss of derivatives\footnote{In fact, the basic 
$L^2$ estimate for the singular system \eqref{15} exhibits loss of a single ``singular derivative" $\partial_{x'}
+\frac{\beta\partial_{\theta_0}}{\eps}$, which is optimal according to the analysis in \cite{CG}.} forces us to use
a new tool, namely, a substantial refinement, given in the companion paper \cite{CGW2}, of the calculus of 
singular pseudodifferential operators constructed in \cite{W1}. In the new version of the calculus, residual operators 
have better smoothing properties than previously realized and can therefore be considered as remainders in
our problem. The loss of derivatives in the linear estimate presents a serious difficulty in the application to our
semilinear problem. Picard iteration appears to be out of the question, so in section \ref{nashex} we use a
Nash-Moser iteration scheme adapted to the scale of spaces \eqref{17} to construct the solution $U_\eps
(x,\theta_0)$ to the semilinear singular problem.

If the problem \eqref{2} satisfied the uniform Lopatinskii condition, then because of the factor $\eps$ in the boundary
data $\eps \, G$, the equations for the leading profile, $\cV^0$ in \eqref{a10}, would be linear; and in fact $\cV^0$
would vanish. The weakly nonlinear regime would correspond to a source term $G$ (and not $\eps \, G$) in
\eqref{2}, see \cite{W3,W2}. Under our weak stability assumption, it turns out that $\cV^0$ is nonlinearly coupled
to the second-order profile $\cV^1$ in the profile equations \eqref{b8s},\eqref{6}. To solve these equations we first
isolate a ``key subsystem" \eqref{12} that decouples from the full system. The basic $L^2$ estimate for the linearization
of the key subsystem still exhibits a loss of one derivative, and we are again forced to use Nash-Moser iteration
in order to solve this subsystem. Once the key subsystem is solved, the solution of the full profile system
\eqref{b8s},\eqref{6} follows easily. It appears in our analysis that the leading order amplitude equation shares the
weak well-posedness of the original nonlinear problem but we have not checked whether the loss of derivative for
the amplitude equation is optimal (we conjecture it is so).

The error analysis used to answer question (2) above is based on the estimate for the singular system \eqref{15},
see Proposition \ref{i5z} below, and is discussed in more detail in section \ref{errori}.

This paper can be read independently of \cite{CGW2}; for the reader's convenience, we have gathered all the
necessary material on the singular calculus in Appendix \ref{calc}. Before giving a fuller discussion, we first provide
some definitions, notation, and a precise statement of assumptions.

\subsection{Assumptions}
\label{assumptions}

We make the following hyperbolicity assumption on the system \eqref{0}:

\begin{assumption}
\label{assumption1}
There exist an integer $q \ge 1$, some real functions $\lambda_1,\dots,\lambda_q$ that are analytic on $\R^d
\setminus \{ 0 \}$ and homogeneous of degree $1$, and there exist some positive integers $\nu_1,\dots,\nu_q$
such that:
\begin{equation*}
\forall \, \xi=(\xi_1,\dots,\xi_d) \in \R^d \setminus \{ 0 \} \, ,\quad
\det \Big[ \tau \, I+\sum_{j=1}^d \xi_j \, B_j \Big] =\prod_{k=1}^q \big( \tau+\lambda_k(\xi) \big)^{\nu_k} \, .
\end{equation*}
Moreover the eigenvalues $\lambda_1(\xi),\dots,\lambda_q(\xi)$ are semi-simple (their algebraic multiplicity
equals their geometric multiplicity) and satisfy $\lambda_1(\xi)<\dots<\lambda_q(\xi)$ for all $\xi \in \R^d
\setminus \{ 0\}$.
\end{assumption}

\noindent For simplicity, we restrict our analysis to noncharacteristic boundaries and therefore make the
following:

\begin{assumption}
\label{assumption2}
The matrix $B_d$ is invertible and the matrix $B:=\psi(0)$ has maximal rank, its rank $p$ being equal to the
number of positive eigenvalues of $B_d$ (counted with their multiplicity). Moreover, the integer $p$ satisfies
$1 \le p \le N-1$.
\end{assumption}

In the normal modes analysis for  \eqref{2}, one first performs a Laplace transform in the time variable $t$ and
a Fourier transform in the tangential space variables $y$.  We let $\tau-i\, \gamma \in \C$ and $\eta \in \R^{d-1}$
denote the dual variables of $t$ and $y$.  We introduce the symbol
\begin{equation*}%\label{defA}
{\mathcal A}(\zeta):= -i \, B_d^{-1} \left( (\tau-i\gamma) \, I +\sum_{j=1}^{d-1} \eta_j \, B_j \right)
\, ,\quad \zeta:=(\tau-i\gamma,\eta) \in \C \times \R^{d-1} \, .
\end{equation*}
For future use, we also define the following sets of frequencies:
\begin{align*}
& \Xi := \Big\{ (\tau-i\gamma,\eta) \in \C \times \R^{d-1} \setminus (0,0) : \gamma \ge 0 \Big\} \, ,
& \Sigma := \Big\{ \zeta \in \Xi : \tau^2 +\gamma^2 +|\eta|^2 =1 \Big\} \, ,\\
& \Xi_0 := \Big\{ (\tau,\eta) \in \R \times \R^{d-1} \setminus (0,0) \Big\} = \Xi \cap \{ \gamma = 0 \} \, ,
& \Sigma_0 := \Sigma \cap \Xi_0 \, .
\end{align*}

Two key objects in our analysis are the hyperbolic region and the glancing set that are defined as follows:

\begin{definition}
\label{def1}
\begin{itemize}
 \item The hyperbolic region ${\mathcal H}$ is the set of all $(\tau,\eta) \in \Xi_0$ such that the matrix
       ${\mathcal A}(\tau,\eta)$ is diagonalizable with purely imaginary eigenvalues.

 \item Let ${\bf G}$ denote the set of all $(\tau,\xi) \in \R \times \R^d$ such that $\xi \neq 0$ and there exists
       an integer $k \in \{1,\dots,q\}$ satisfying:
\begin{equation*}
\tau + \lambda_k(\xi) = \dfrac{\partial \lambda_k}{\partial \xi_d} (\xi) = 0 \, .
\end{equation*}
If $\pi ({\bf G})$ denotes the projection of ${\bf G}$ on the $d$ first coordinates (in other words $\pi (\tau,\xi)
=(\tau,\xi_1,\dots,\xi_{d-1})$ for all $(\tau,\xi)$), the glancing set ${\mathcal G}$ is ${\mathcal G} :=
\pi ({\bf G}) \subset \Xi_0$.
\end{itemize}
\end{definition}

\noindent We recall the following result that is due to Kreiss \cite{K} in the strictly hyperbolic case (when all integers
$\nu_j$ in Assumption \ref{assumption1} equal $1$) and to M\'etivier \cite{Met} in our more general framework:

\begin{prop}[\cite{K,Met}]
\label{thm1}
Let Assumptions \ref{assumption1} and \ref{assumption2} be satisfied. Then for all $\zeta \in \Xi \setminus
\Xi_0$, the matrix ${\mathcal A}(\zeta)$ has no purely imaginary eigenvalue and its stable subspace $\E^s
(\zeta)$ has dimension $p$. Furthermore, $\E^s$ defines an analytic vector bundle over $\Xi \setminus \Xi_0$
that can be extended as a continuous vector bundle over $\Xi$.
\end{prop}

\noindent For all $(\tau,\eta) \in \Xi_0$, we let $\E^s(\tau,\eta)$ denote the continuous extension of $\E^s$ to the
point $(\tau,\eta)$.  The analysis in \cite{Met} shows that away from the glancing set ${\mathcal G} \subset \Xi_0$,
$\E^s(\zeta)$ depends analytically on $\zeta$, and the hyperbolic region ${\mathcal H}$ does not contain any
glancing point.

To treat the case when the boundary operator in \eqref{2}(b) is independent of $u$, meaning $\psi(\eps u)
\equiv \psi(0) =:B$, we make the following \emph{weak stability assumption} on the problem $(L(\partial),B)$.

\begin{assumption}
\label{assumption3}
\begin{itemize}
 \item For all $\zeta \in \Xi \setminus \Xi_0$, $\text{\rm Ker} B \cap \E^s (\zeta) = \{ 0\}$.

 \item The set $\Upsilon_0 := \{ \zeta \in \Sigma_0 : \text{\rm Ker} B \cap \E^s (\zeta) \neq \{ 0\} \}$ is
       nonempty and included in the hyperbolic region ${\mathcal H}$. %Moreover, the intersection
          %$\text{\rm Ker} B \cap \E^s (\zeta)$ has dimension one for $\zeta \in \Upsilon$.

 \item For all $\underline{\zeta} \in \Upsilon_0$, there exists a neighborhood ${\mathcal V}$ of $\underline{\zeta}$
       in $\Sigma$, a real valued ${\mathcal C}^\infty$ function $\sigma$ defined on ${\mathcal V}$, a basis
       $E_1(\zeta),\dots,E_p(\zeta)$ of $\E^s(\zeta)$ that is of class ${\mathcal C}^\infty$ with respect to
       $\zeta \in {\mathcal V}$, and a matrix $P(\zeta) \in \text{\rm GL}_p (\C)$ that is of class
       ${\mathcal C}^\infty$ with respect to $\zeta \in {\mathcal V}$, such that
\begin{equation*}%\label{z1}
\forall \, \zeta \in {\mathcal V} \, ,\quad B \, \begin{pmatrix}
E_1(\zeta) & \dots & E_p(\zeta) \end{pmatrix}
= P(\zeta) \, \text{\rm diag } \big( \gamma +i\, \sigma (\zeta),1,\dots,1 \big) \, .
\end{equation*}
% \item For each point $\uzeta\in\Upsilon$ there exists a neighborhood $\cV$ of $\uzeta$ in $\Sigma$
%          and a constant $c>0$ such that
%\begin{equation}
%\label{z1}
%\forall \zeta\in\cV,\;\forall Z\in\bE^s(\zeta),\;|BZ|\geq c\, \gamma \, |Z| \, .
%\end{equation}
\end{itemize}
\end{assumption}

\noindent For comparison and later reference we recall the following definition.

\begin{defn}[\cite{K}]
\label{ustable}
As before let $p$ be the number of positive eigenvalues of $B_d$. The problem $(L(\partial),B)$ is said to
be \emph{uniformly stable} or to satisfy the \emph{uniform Lopatinskii condition} if
\begin{equation*}%\label{us}
B \, : \, \E^s (\zeta) \longrightarrow \bC^p
\end{equation*}
is an isomorphism for all $\zeta\in \Sigma$.
\end{defn}

\begin{rem}
\label{careful}
\textup{Observe that if $(L(\partial),B)$ satisfies the uniform Lopatinskii condition, continuity implies that this
condition still holds for $(L(\partial),B+\dot\psi)$, where $\dot\psi$ is any sufficiently small perturbation of $B$.
Hence the uniform Lopatinskii condition is a convenient framework for nonlinear perturbation. The analogous
statement may not be true when $(L(\partial),B)$ is only weakly stable. Remarkably, weak stability persists
under perturbation in the so-called WR class exhibited in \cite{BRSZ}, and Assumption \ref{assumption3}
above is a convenient equivalent definition of the WR class (see \cite[Appendix B]{CG}). In order to handle 
general nonlinear boundary conditions as in \eqref{2} we shall strengthen Assumption \ref{assumption3} in 
Assumption \ref{nonlinbc} below.}
\end{rem}

\textbf{Boundary and interior phases.}
We consider a planar real phase $\phi_0$ defined on the boundary:
\begin{equation}
\label{phasebord}
\quad \phi_0(t,y) :=
\underline{\tau} \, t +\underline{\eta} \cdot y \, ,\quad \tauetabar \in \Xi_0 \, .
\end{equation}
As follows from earlier works, see e.g. \cite{MA}, oscillations on the boundary associated with the phase
$\phi_0$ give rise to oscillations in the interior associated with some planar phases $\phi_m$. These phases
are characteristic for the hyperbolic operator $L_0(\partial)$ and their trace on the boundary $\{ x_d=0\}$
equals $\phi_0$. For now we make the following:

\begin{assumption}
\label{a8}
The phase $\phi_0$ defined by \eqref{phasebord} satisfies $\tauetabar \in \Upsilon_0$. In particular
$\tauetabar \in {\mathcal H}$.
\end{assumption}

\noindent Thanks to Assumption \ref{a8}, we know that the matrix ${\mathcal A} \tauetabar$ is
diagonalizable with purely imaginary eigenvalues. These eigenvalues are denoted $i\, \underline{\omega}_1,
\dots,i\, \underline{\omega}_M$, where the $\underline{\omega}_m$'s are real and pairwise distinct. The
$\underline{\omega}_m$'s are the roots (and all the roots are real) of the dispersion relation:
\begin{equation*}
\det \Big[ \underline{\tau} \, I+\sum_{j=1}^{d-1} \underline{\eta}_j \, B_j +\omega \, B_d \Big] = 0 \, .
\end{equation*}
To each root $\underline{\omega}_m$ there corresponds a unique integer $k_m \in \{ 1,\dots,q\}$ such that
$\underline{\tau} + \lambda_{k_m} (\underline{\eta},\underline{\omega}_m)=0$. We can then define the following
real\footnote{If $\tauetabar$ does not belong to the hyperbolic region ${\mathcal H}$, some of the phases
$\varphi_m$ may be complex, see e.g. \cite{W3,W2,Le,Mar}. Moreover, glancing phases introduce a new
scale $\sqrt{\eps}$ as well as boundary layers. } phases and their associated group velocities:
\begin{equation}
\label{phases}
\forall \, m =1,\dots,M \, ,\quad \phi_m (x):= \phi_0(t,y)+\underline{\omega}_m \, x_d \, ,\quad
{\bf v}_m := \nabla \lambda_{k_m} (\underline{\eta},\underline{\omega}_m) \, .
\end{equation}
Let us observe that each group velocity ${\bf v}_m$ is either incoming or outgoing with respect to the space
domain $\R^d_+$: the last coordinate of ${\bf v}_m$ is nonzero. This property holds because $\tauetabar$
does not belong to the glancing set ${\mathcal G}$. We can therefore adopt the following classification:

\begin{definition}
\label{def2}
The phase $\phi_m$ is incoming if the group velocity ${\bf v}_m$ is incoming (that is, when $\partial_{\xi_d}
\lambda_{k_m} (\underline{\eta},\underline{\omega}_m)>0$), and it is outgoing if the group velocity ${\bf v}_m$
is outgoing ($\partial_{\xi_d} \lambda_{k_m} (\underline{\eta},\underline{\omega}_m) <0$).
\end{definition}

\noindent In all that follows, we let ${\mathcal I}$ denote the set of indices $m \in \{ 1,\dots,M\}$ such
that $\phi_m$ is an incoming phase, and ${\mathcal{O}}$ denote the set of indices $m \in \{ 1,\dots,M\}$ such
that $\phi_m$ is an outgoing phase.   If $p\geq 1$, then $\cI$ is nonempty, while if $p\leq N-1$, $\cO$ is
nonempty (see Lemma \ref{lem1} below).  We  will use the notation:
\begin{align*}%\label{a9a}
&L_0(\tau,\xi) := \tau \, I +\sum_{j=1}^d \xi_j \, B_j \, ,\; \;
L(\beta,\uomega_m):=\uomega_m \, I +\sum_{k=0}^{d-1} \beta_k \, A_k,\\
&\;\beta:=(\utau,\ueta), \;\;x'=(t,y),\;\; \phi_0(x')=\beta\cdot x'.
\end{align*}
For each phase $\phi_m$, $d\phi_m$ denotes the differential of the function $\phi_m$ with respect to its
argument $x=(t,y,x_d)$. It follows from Assumption \ref{assumption1} that the eigenspace of ${\mathcal A}
(\beta)$ associated with the eigenvalue $i\, \underline{\omega}_m$ coincides with the kernel of $L_0(d\phi_m)$
and has dimension $\nu_{k_m}$. The following well-known lemma, whose proof is recalled in \cite{CG}, gives a
useful decomposition of $\E^s$ in the hyperbolic region.

\begin{lem}
\label{lem1}
The stable subspace $\E^s(\beta)$ admits the decomposition:
\begin{equation}
\label{decomposition1}
\E^s (\beta) = \oplus_{m \in {\mathcal I}} \, \ker L_0({ d} \phi_m) \, ,
\end{equation}
and each vector space in the decomposition \eqref{decomposition1} admits a basis of real vectors.
\end{lem}

To formulate our last assumption we observe first that for every point $\uzeta\in\cH$ there is a neighborhood
$\cV$ of $\uzeta$ in $\Sigma$ and a $C^\infty$ conjugator $Q_0(\zeta)$ defined on $\cV$ such that
\begin{align}\label{conjugator}
Q_0(\zeta)\cA(\zeta)Q_0^{-1}(\zeta)=\begin{pmatrix}
i\omega_1(\zeta) I_{n_1} & \; & 0\\
\; & \ddots & \;\\
0 & \; & i\omega_J(\zeta) I_{n_J} \end{pmatrix} =:-\bD_1(\zeta),
\end{align}
where the $\omega_j$ are real when $\gamma=0$ and there is a constant $c>0$ such that either
\begin{equation*}%\label{conj1}
\re (i\omega_j) \leq -c\, \gamma \text{ or } \re (i\omega_j) \geq c\, \gamma \text{ for all } \zeta \in \cV.
\end{equation*}
In view of Lemma \ref{lem1} we can choose the first $p$ columns of $Q^{-1}_0(\zeta)$ to be a basis of
$\bE^s(\zeta)$, and we write
\begin{equation*}%\label{conj2}
Q_0^{-1}(\zeta)=[Q_{in}(\zeta) \;Q_{out}(\zeta)].
\end{equation*}
Choose $J'$ so that the first $J'$ blocks of $-\bD_1$ lie in the first $p$ columns, and the remaining blocks
in the remaining $N-p$ columns. Thus, $\re (i\omega_j) \le -c\, \gamma$ for $1\leq j\leq J'$.

Observing that the linearization of the boundary condition in \eqref{2} is
\begin{equation*}
\dot u \longmapsto \psi(\eps u) \dot u +[\partial_v \psi(\eps u) \, \dot u] \, \eps u \, ,
\end{equation*}
we define the operator
\begin{align}\label{ca2}
\cB(v_1,v_2) \, \dot u := \psi(v_1) \dot u +[\partial_v \psi(v_1) \, \dot u] \, v_2 \, ,
\end{align}
which appears in Assumption \ref{nonlinbc} below. For later use we also define
\begin{align}\label{caa}
\cD(v_1,v_2) \, \dot u := D(v_1) \dot u +[\partial_v D(v_1) \, \dot u] \, v_2 \, ,
\end{align}
as well as
\begin{align}\label{caaa}
\cB(v_1) :=\cB(v_1,v_1) \, , \; \; \cD(v_1) := \cD(v_1,v_1) \, .
\end{align}

We now state the weak stability assumption that we make when considering the general case of nonlinear
boundary conditions in \eqref{2}.

\begin{assumption}
\label{nonlinbc}
\begin{itemize}
 \item There exists a neighborhood $\cO$ of $(0,0)\in\bR^{2N}$ such that for all $(v_1,v_2)\in\cO$ and all
          $\zeta \in \Xi \setminus \Xi_0$, $\ker \;\cB(v_1,v_2) \cap \E^s (\zeta) = \{ 0\}$. For each $(v_1,v_2) \in
          \cO$ the set $\Upsilon(v_1,v_2):=\{\zeta\in\Sigma_0:\ker \; \cB(v_1,v_2) \cap \E^s (\zeta) \neq \{ 0\} \}$
          is nonempty and included in the hyperbolic region ${\mathcal H}$. Moreover, if we set $\Upsilon :=
          \cup_{(v_1,v_2)\in\cO}\Upsilon(v_1,v_2)$, then $\overline{\Upsilon}\subset \cH$ (closure in $\Sigma_0$).

 \item For every $\uzeta\in\overline{\Upsilon}$ there exists a neighborhood $\cV$ of $\uzeta$ in $\Sigma$
          and a $C^\infty$ function $\sigma(v_1,v_2,\zeta)$ on $\cO\times\cV$ such that for all $(v_1,v_2,\zeta)
          \in \cO\times \cV$ we have $\text{\rm Ker}\; \cB(v_1,v_2) \cap \E^s (\zeta) \neq \{ 0\}$ if and only if
          $\zeta\in\Sigma_0$ and $\sigma(v_1,v_2,\zeta)=0$.

          \qquad Moreover, there exist matrices $P_i(v_1,v_2,\zeta)\in\text{\rm GL}_p (\C)$, $i=1,2,$ of class
          $C^\infty$ on $\cO\times\cV$ such that $\forall \, (v_1,v_2,\zeta) \in \cO\times {\mathcal V}$
\begin{equation}\label{ca3a}
P_1(v_1,v_2,\zeta)\cB(v_1,v_2)Q_{in}(\zeta)P_2(v_1,v_2,\zeta) =\text{\rm diag }
\big( \gamma +i\, \sigma (v_1,v_2,\zeta),1,\dots,1 \big) \, .
\end{equation}
 \end{itemize}
\end{assumption}

For nonlinear boundary conditions, the phase $\phi_0$ in \eqref{phasebord} is assumed to satisfy
$\tauetabar \in \Upsilon(0,0)$, or in other words the intersection $\ker B \cap \E^s \tauetabar$ is not
reduced to $\{ 0\}$ (the set $\Upsilon_0$ in Assumption \ref{assumption3} is a short notation for
$\Upsilon (0,0)$). The phases $\phi_m$ are still defined by \eqref{phases} and thus only depend
on $L(\partial)$ and $B$ and not on the nonlinear perturbations $f_0$ and $\psi (\eps \, u) -\psi(0)$
added in \eqref{2}.

\begin{rem}\label{ca3}
1) \textup{The properties stated in Assumption \ref{nonlinbc} are just a convenient description of the requirements
for belonging to the WR class of \cite{BRSZ}. Like the uniform Lopatinskii condition, Assumption \ref{nonlinbc} can
in practice be verified by hand via a ``constant-coefficient" computation. More precisely, for $(v_1,v_2)$ near
$(0,0)\in\bR^{2N}$ and $\zeta\in\Sigma$, one can define (see, e.g., \cite[chapter 4]{BS}) a Lopatinskii determinant
$\Delta(v_1,v_2,\zeta)$ that is $C^\infty$ in $(v_1,v_2)$, analytic in $\zeta=(\tau-i\gamma,\eta)$ on $\Sigma
\setminus \cG$, and satisfies
\begin{equation*}
\Delta(v_1,v_2,\zeta)=0\text{ if and only if }\text{\rm Ker}\; \cB(v_1,v_2) \cap \E^s (\zeta) \neq \{ 0\}.
\end{equation*}
In particular, $\Delta(v_1,v_2,\cdot)$ is real-analytic on $\cH$.}

\textup{Following \cite{BRSZ}, see also \cite[chapter 8]{BS}, we claim that  Assumption \ref{nonlinbc} holds provided
\begin{equation}\label{ca4}
\emptyset\neq\{\zeta\in\Sigma: \Delta(0,0,\zeta)=0\}\subset \cH\text{ and }
\Delta(0,0,\uzeta)=0 \Rightarrow \partial_\tau \Delta(0,0,\uzeta)\neq 0,
\end{equation}
and it thus only involves a weak stability property for the linearized problem at $(v_1,v_2)=(0,0)$. Indeed, the implicit
function theorem then implies that for $(v_1,v_2)$ near zero and $(\tau,\eta)$ near $\uzeta$, the set $\{(\tau,\eta)\in
\Sigma_0: \Delta(v_1,v_2,\tau,\eta)=0\}$ is a real-analytic hypersurface in $\cH$.  On the other hand, an application
of the implicit function theorem to $\Delta(v_1,v_2,z,\eta)$, for $(z,\eta)\in\Sigma$ shows that the real dimension of
the manifold $\{(z,\eta)\in\Sigma:\Delta(v_1,v_2,z,\eta)=0\}$ must be the same, that is, $d-2$. The two zero sets must
then coincide; there are no zeros in $\Sigma\setminus\Sigma_0$. The function $\sigma$ and the neighborhoods
$\cO$ and $\cV$ arise in a factorization of $\Delta$ given by the Weierstrass Preparation Theorem. The construction
of the conjugating matrices $P_i$, $i=1,2$ follows from a construction in \cite[p. 268-270]{MS}.}

\textup{Instead of assuming \eqref{ca4}, we have stated Assumption \ref{nonlinbc} in a form that is more directly
applicable to the proof of Proposition \ref{i5z} and to the error analysis of Theorem \ref{main2}.}

2) \textup{To prove the basic estimate for the linearized singular system, Proposition \ref{i5z}, and to construct
the exact solution $U_\eps$ to the singular system \eqref{15} below, it is enough to require that the analogue of
Assumption \ref{nonlinbc} holds when $\cB(v_1,v_2)$ is replaced by $\cB(v_1):=\cB(v_1,v_1)$. However, for the error
analysis of section \ref{error} in the case of nonlinear boundary conditions, we need  Assumption \ref{nonlinbc}
as stated.}
\end{rem}

The next Lemma, proved in \cite{CG}, gives a useful decomposition of $\bC^N$ and introduces projectors
needed later for formulating and solving the profile equations.

\begin{lem}
\label{lem2}
The space $\bC^N$ admits the decomposition:
\begin{equation}
\label{decomposition2}
\C^N = \oplus_{m=1}^M \, \ker L_0({ d} \phi_m)
\end{equation}
and each vector space in \eqref{decomposition2} admits a basis of real vectors. If we let $P_1,\dots,P_M$ denote
the projectors associated with the decomposition \eqref{decomposition2}, then there holds $\text{\rm Im } B_d^{-1}
L_0({ d} \phi_m) = \ker P_m$ for all $m=1,\dots,M$.
\end{lem}

\subsection{Main results}

For each $m\in\{1,\dots,M\}$ we let
\begin{equation*}%\label{a9}
r_{m,k}, \;k=1,\dots,\nu_{k_m}
\end{equation*}
denote a basis of $\ker L_0(d\phi_m)$ consisting of real vectors. In section \ref{error} we shall construct a
``corrected" approximate solution $u^c_\eps$ of \eqref{2} of the form
\begin{equation}
\label{a10}
u^c_\eps(x)=\cV^0 \left( x,\frac{\phi}{\eps} \right) +\eps \, \cV^1 \left( x,\frac{\phi}{\eps} \right)
+\eps^2 \, \cU^2_p \left( x,\frac{\phi_0}{\eps},\frac{x_d}{\eps} \right),
\end{equation}
where $\phi :=(\phi_1,\dots,\phi_M)$ denotes the collection of all phases,
\begin{align}
\label{a11}
\begin{split}
&\cV^0 \left( x,\frac{\phi}{\eps} \right) =\sum_{m\in\cI} \, \sum^{\nu_{k_m}}_{k=1}
\sigma_{m,k}\left(x,\frac{\phi_m}{\eps}\right) \, r_{m,k} \, ,\\
&\cV^1 \left( x,\frac{\phi}{\eps} \right) =\underline{\cV}^1(x) +\sum^M_{m=1} \, \sum^{\nu_{k_m}}_{k=1}
\tau_{m,k}\left(x,\frac{\phi_m}{\eps}\right) \, r_{m,k} +\cR\cV^0 \, ,
\end{split}
\end{align}
and the $\sigma_{m,k}(x,\theta_m)$ and $\tau_{m,k}(x,\theta_m)$ are scalar $C^1$ functions periodic in $\theta_m$
with mean $0$ which describe the propagation of oscillations with phase $\phi_m$ and group velocity ${\bf v}_m$.
Here $\cR$ denotes the operator
\begin{equation*}%\label{a12}
\cR\cV^0=-R \, [L(\partial_x)\cV^0 +D(0)\cV^0]
\end{equation*}
for $R$ defined as in \eqref{R}. The last corrector $\eps^2 \, \cU^2_p(x,\theta_0,\xi_d)$ in \eqref{a10} is a trigonometric
polynomial constructed in the error analysis of section \ref{error}.

The following theorem, our main result, is an immediate corollary of the more precise Theorem \ref{main2}.
Here we let $\Omega_T :=\{(x,\theta_0)=(t,y,x_d,\theta_0)\in\bR^{d+1}\times \bT^1: x_d\geq 0, t <T\}$, and
$b\Omega_T :=\{(t,y,\theta_0) \in \bR^d \times \bT^1: t <T\}$. The spaces $E^s$ are defined in \eqref{17}
below.

\begin{theo}
\label{main}
We make Assumptions \ref{assumption1}, \ref{assumption2}, \ref{assumption3}, and \ref{a8} when the boundary
condition in \eqref{2} is linear ($\psi(\eps u) \equiv \psi(0)$); in the general case we substitute Assumption \ref{nonlinbc}
for Assumption \ref{assumption3}.
Fix $T>0$, set $M_0:=3\, d+5$, and let
\begin{align}\notag
a:=[(d+1)/2]+M_0+3 \text{ and }\tilde{a}=2a-[(d+1)/2].
\end{align}
Consider the semilinear boundary problem \eqref{2}, where $G(t,y,\theta_0)\in H^{\tilde a}(b\Omega_T)$.
There exists $\eps_0>0$ such that if $\langle G \rangle_{H^{a+2}(b\Omega_T)}$ is small enough, there exists a unique
function $U_\eps(x,\theta_0)\in E^{a-1}(\Omega_T)$ satisfying the singular system \eqref{15} on $\Omega_T$ such that
\begin{equation*}
u_\eps(x):=U_\eps\left(x,\frac{x'\cdot\beta}{\eps}\right)
\end{equation*}
is an exact solution of \eqref{2} on $(-\infty,T] \times \overline{\bR}^d_+$ for $0<\eps \leq \eps_0$. In addition there exists
a profile $\cV^0(x,\theta)$ as in \eqref{a11}, whose components $\sigma_{m,k}$ lie in $H^{a-1}(\Omega_T)$, such that
the approximate solution defined by
\begin{equation*}
u^{app}_\eps :=\cV^0 \left(x,\frac{\phi}{\eps}\right)
\end{equation*}
satisfies
\begin{equation*}
\lim_{\eps\to 0}|u_\eps-u^{app}_\eps|_{L^\infty}=0\text{ on }(-\infty,T]\times \overline{\bR}^d_+.
\end{equation*}
\end{theo}

Observe that although the boundary data in the problem \eqref{2} is of size $O(\eps)$, the approximate solution
$u^{app}_\eps$ is of size $O(1)$, exhibiting an amplification due to the weak stability at frequency $\beta$. The
main information provided by Theorem \ref{main} is that this amplification does not rule out existence of a smooth
solution on a fixed time interval, that is it does not trigger a violent instability, at least in this weakly nonlinear regime.
As far as we know, the derivation of the leading order amplitude equation \eqref{12} is also new in the general
framework that we consider. We hope that the analysis developed in this article will be useful in justifying
quasilinear amplification phenomena such as the Mach stems or kink modes formation of \cite{AM,MA}.

\begin{rem}\label{13a}
\textup{a)\;In order to avoid some technicalities we have stated our main result for a problem \eqref{2} where all
data is $0$ in $t<0$. This result easily implies a similar result in which outgoing waves defined in $t<0$ of amplitude
$O(\eps)$ and wavelength $\eps$ give rise to reflected waves of amplitude $O(1)$. In either formulation, analysis
of the profile equations (see Remark \ref{radiate}) shows that the waves of amplitude $O(1)$ emanate from a
radiating wave that propagates in the boundary along a characteristic of the Lopatinskii determinant.}

\textup{b)\;We have decided to fix $T>0$ at the start and choose data small enough so that a solution to the nonlinear
problem exists up to time $T$. One can also (as discussed in Remark \ref{d28}) fix the data  in the problem ($G$ in
\eqref{2}) at the start, and then choose $T$ small enough so that a solution to the nonlinear problem exists up to time
$T$.}
\end{rem}

In the remainder of this introduction, we discuss the construction of exact solutions, the construction of the 
approximate solution $\cV^0$, and the error analysis. Complete proofs are given in Sections \ref{exact}, 
\ref{profile}, \ref{error}, and \ref{iteration}.

\subsection{Exact solutions and singular systems}

\emph{\quad}The theory of weakly stable hyperbolic initial boundary value problems fails to provide a solution of the
system \eqref{2} that exists on a fixed time interval independent of $\eps$\footnote{This would be true even for problems
$(L(\partial),B)$ that are uniformly stable in the sense of Definition \ref{ustable}.}. In order to obtain such an exact solution
to the system \eqref{2} we adopt the strategy of studying an associated singular problem first used by \cite{JMR} for an
initial value problem in free space.  We look for a solution of the form
\begin{align}\label{14a}
u_\eps(x)=U_\eps(x,\theta_0)|_{\theta_0=\frac{\phi_0(x')}{\eps}},
\end{align}
where $U_\eps(x,\theta_0)$ is periodic in $\theta_0$ and satisfies the singular system derived by substituting \eqref{14a}
into the problem \eqref{2}.  Recalling that $L(\partial)=\partial_d +\sum^{d-1}_{j=0}A_j\partial_j$  we obtain:
\begin{align}\label{15}
\begin{split}
&\partial_d U_\eps+\sum^{d-1}_{j=0}A_j\left(\partial_j +\frac{\beta_j \partial_{\theta_0}}{\eps}\right) U_\eps
+D(\eps U_\eps)\, U_\eps=:\\
&\qquad\qquad\partial_d U_\eps+\mathbb{A}\left(\partial_{x'}+\frac{\beta \partial_{\theta_0}}{\eps}\right)U_\eps
+D(\eps U_\eps)\, U_\eps=0\\
&\psi(\eps U_\eps) \, U_\eps|_{x_d=0} =\eps \, G(x',\theta_0),\\
&U_\eps =0 \text{ in } t<0.
\end{split}
\end{align}
The special difficulties presented by such singular problems when there is a boundary are described in detail in the
introductions to \cite{W1} and \cite{CGW1}. In particular we mention: (a) symmetry assumptions on the matrices $B_j$
appearing in the problem \eqref{0} equivalent to \eqref{2} are generally of no help in obtaining an $L^2$ estimate for
\eqref{15} (boundary conditions satisfying Assumption \ref{assumption3} can not be maximally dissipative, see \cite{CG});
(b) one cannot control $L^\infty$ norms just by estimating tangential derivatives $\partial^\alpha_{(x',\theta_0)}U_\eps$
because \eqref{15} is not a hyperbolic problem in the $x_d$ direction\footnote{For initial value problems in free space,
one \emph{can} control $L^\infty$ norms just by estimating enough derivatives tangent to time slices $t=c$.}; moreover,
even if one has estimates of tangential derivatives  uniform with respect to $\eps$, because of the factors $1/\eps$ in
\eqref{15} one cannot just use the equation to control $\partial_d U_\eps$ and thereby control $L^\infty$ norms.

In \cite{W1} a class of singular pseudodifferential operators, acting on functions $U(x',\theta_0)$ periodic in $\theta_0$
and having the form
\begin{equation*}%\label{16}
p_D U(x',\theta_0)=\frac{1}{(2\, \pi)^{d+1}} \, \sum_{k\in\bZ} \int_{\R^d} {\rm e}^{ix' \cdot \xi' +i\theta_0 k} \,
p\left(\eps V(x',\theta_0),\xi'+\frac{k\beta}{\eps},\gamma\right) \, \widehat{U}(\xi',k) \, {\rm d}\xi',\;\;\gamma\geq 1,
\end{equation*}
was introduced to deal with these difficulties. Observe that the differential operator $\mathbb{A}$ appearing in \eqref{15}
can be expressed in this form. Kreiss-type symmetrizers $r_s(D_{x',\theta_0})$ in the singular calculus were constructed
in \cite{W1} for (quasilinear systems similar to) \eqref{15} under the assumption that $(L(\partial),\psi(0))$ is uniformly
stable in the sense of Definition \ref{ustable}. With these one can prove $L^2(x_d,H^s(x',\theta_0))$ estimates uniform
in $\eps$ for  \eqref{15}, even when $\eps \, G$ is replaced by $G$ in the boundary condition. To progress further and
control $L^\infty$ norms, the boundary frequency $\beta$ is restricted to lie the complement of the glancing set. With
this extra assumption the singular calculus was used in \cite{W1} to block-diagonalize the operator
$\mathbb{A}\left(\eps U_\eps,\partial_{x'}+\frac{\beta \partial_{\theta_0}}{\eps}\right)$ microlocally near the $\beta$
direction and thereby prove estimates uniform with respect to $\eps$ in the spaces
\begin{align}\label{17}
E^s :=C(x_d,H^s(x',\theta_0))\cap L^2(x_d,H^{s+1}(x',\theta_0)).
\end{align}
These spaces are Banach algebras and are contained in $L^\infty$ for $s>\frac{d+1}{2}$. For large enough $s$, as
determined by the requirements of the calculus, existence of solutions to \eqref{15} in $E^s$ on a time interval $[0,T]$
independent of $\eps\in (0,\eps_0]$ follows by Picard iteration in the uniformly stable case.

The singular calculus of \cite{W1} was used again in \cite{CGW1} to rigorously justify leading order geometric optics
expansions for the quasilinear analogue of \eqref{2} in the uniformly stable case (with $\beta\in \cH$ and the forcing 
term $G$ in place of $\eps G$ in the boundary condition). Under the assumptions made in the present paper, in particular 
assuming weak stability as in Assumption \ref{assumption3} or Assumption \ref{nonlinbc}, we face the additional difficulty 
that the basic $L^2$ estimate for the problem $(L(\partial),B)$  exhibits a loss of  derivatives. A consequence of this is 
that the singular calculus of \cite{W1} is no longer adequate for estimating solutions of \eqref{15}. The main reason is 
that remainders in the calculus of \cite{W1} are just bounded operators on $L^2$, while for energy estimates with a loss 
of derivative remainders should be smoothing operators. We therefore need to use an improved version of the calculus 
constructed in \cite{CGW2} in which residual operators are shown to have better smoothing properties than previously 
thought. With the improved calculus we are able in section \ref{tameex} to estimate solutions of \eqref{15} in $E^s$ 
spaces \eqref{17}, but of course there is a loss of one singular derivative in the estimates. This loss forces us in section 
\ref{nashex} to use Nash-Moser iteration on the scale of $E^s$ spaces to obtain an exact solution of the singular 
system \eqref{15} on a fixed time interval independent of $\eps$. Observe that one singular derivative costs a factor 
$1/\eps$ and this is another reason why the scaling $\eps G$ in \eqref{15} is crucial.

\begin{rem}\label{diff}
\textup{The main idea for proving the estimate for the linearized singular problem, Proposition \ref{i5z}, is to adapt
the techniques of \cite{C1} to the singular pseudodifferential framework. There is however one major obstacle
along the way. While the error term in the composition of two zero order operators (or in the composition of an
operator of order of order $-1$ (on the left) with an operator of order $1$, a $(-1,1)$ composition) is smoothing of order
one in the sense of \eqref{n17}, the same is unfortunately not true of the error term in $(1,-1)$ compositions (there are
counter-examples for that). The properties of the $(1,-1)$ error terms that arise in our proof are described in \eqref{ib3a};
see also Proposition \ref{n31a}. To deal with these errors we use further microlocal cutoffs $\chi^e$ in the extended
calculus\footnote{These are cutoffs in $(\xi',m)$-space which depend on $(\eps,\gamma)$; they are described in section
\ref{extended}.} and partition the solution as in \eqref{ia2}.  In particular we are led to estimate terms like those involving
norms of $\nabla_{x'}\Lambda_D^{-1} \dot U^\gamma_{2,in}$ in Proposition \ref{hard} and  $\nabla_{x'}\Lambda_D^{-1}
\dot U^\gamma_3$ in Proposition \ref{kreisspiece}.}
\end{rem}

\subsection{Derivation of the leading profile equations}

\qquad We now derive  the profile equations for the semilinear problem \eqref{2}. We work with profiles $\cV^j(x,\theta)$
periodic in $\theta=(\theta_1,\dots,\theta_M)$, where $\theta_j$ is a placeholder for $\phi_j/\eps$. Looking for an
approximate solution of \eqref{2} of the form $u^a=(\cV^0+\eps\cV^1+\eps^2\cV^2)|_{\theta=\phi/\eps}$, where
$\phi=(\phi_1,\dots,\phi_M)$,  we get interior equations
\begin{align}\label{3}
\begin{split}
&(a)\;\cL(\partial_{\theta})\cV^0=0,\\
&(b)\;\cL(\partial_{\theta})\cV^1+L(\partial)\cV^0+D(0)\cV^0=0,\\
&(c)\;\cL(\partial_{\theta})\cV^2+L(\partial)\cV^1+D(0)\cV^1+(\partial_vD(0) \cV^0)\cV^0=0,
\end{split}
\end{align}
by plugging $u^a$ into \eqref{2}(a) and setting the coefficients of, respectively, $\eps^{-1}$, $\eps^0$, and $\eps$
equal to zero. The operator $\cL(\partial_\theta)$ is defined by
\begin{equation}\label{3a}
\cL(\partial_\theta):=\sum^M_{j=1}L(d\phi_j)\partial_{\theta_j}.
\end{equation}

With $B:=\psi(0)$, the boundary equations, obtained by plugging $u^a$ into \eqref{2}(b) and setting the coefficients
of $\eps^0$ and $\eps$ equal to zero, are
\begin{align}\label{4}
\begin{split}
&B\cV^0(x',0,\theta_0,\dots,\theta_0)=0,\\
&B\cV^1+(\partial_v \psi(0) \cV^0) \, \cV^0 =G(x',\theta_0),
\end{split}
\end{align}
where $\theta_0$ is a placeholder for $\phi_0/\eps$.  We will see that as a consequence of the weak stability at
frequency $\beta$, the problem for the leading profile $\cV^0$ is nonlinear and nonlocal. (See Appendix \ref{exeuler}
for a concrete example.) Thus, the scaling in \eqref{1} \emph{is} the weakly nonlinear scaling when the uniform
Lopatinskii condition fails at a hyperbolic frequency $\beta$. To analyze these equations we proceed to define
appropriate function spaces and a pair of auxiliary operators $E$ and $R$.

Functions $\cV(x,\theta)\in L^2(\overline{\bR}^{d+1}_+\times\mathbb{T}^M)$ have Fourier series
\begin{align}\label{b1}
\cV(x,\theta)=\sum_{\alpha\in\bZ^M}V_\alpha(x)e^{i\alpha\cdot\theta}.
\end{align}
Since only quadratic interactions appear in \eqref{3} and we anticipate that $\cV^0$ will have the form in \eqref{a11},
for $k=1,2$ we let
\begin{equation*}%\label{b4b}
\bZ^{M;k}=\{\alpha\in\bZ^M: \text{ at most $k$ components of }\alpha \text{ are nonzero}\},
\end{equation*}
and we consider the subspace $H^{s;k}(\overline{\bR}^{d+1}_+\times\mathbb{T}^M)\subset
H^s(\overline{\bR}^{d+1}_+\times\mathbb{T}^M)$ defined by
\begin{align}\label{b4a}
H^{s;k}(\overline{\bR}^{d+1}_+\times\mathbb{T}^M)=\left\{\cV(x,\theta) \in
H^s(\overline{\bR}^{d+1}_+\times\mathbb{T}^M) : \cV(x,\theta)
=\sum_{\alpha\in\bZ^{M;k}}V_\alpha(x)e^{i\alpha\cdot\theta} \right\}.
\end{align}
Thus, multiplication defines a continous map
\begin{align}\label{b4c}
H^{s;1}(\overline{\bR}^{d+1}_+\times\mathbb{T}^M) \times H^{s;1}(\overline{\bR}^{d+1}_+\times\mathbb{T}^M)
\to H^{s;2}(\overline{\bR}^{d+1}_+\times\mathbb{T}^M)
\end{align}
for $s>(d+1+2)/2$.

\begin{defn}\label{b2}
Setting $\phi :=(\phi_1,\dots,\phi_M)$, we say $\alpha\in\bZ^{M;2}$ is a \emph{characteristic mode} and write
$\alpha\in\cC$ if $\det L(d(\alpha \cdot \phi))=0$.  Otherwise we call $\alpha$  a \emph{noncharacteristic} mode.
We decompose
\begin{equation*}%\label{b3}
\cC= \cup_{m=1}^M\cC_m, \text{ where } \cC_m
:= \{ \alpha\in\bZ^{M;2}:\alpha \cdot \phi =n_\alpha\phi_m\text{ for some }n_\alpha\in\bZ \}.
\end{equation*}
\end{defn}

Observe that for $\alpha\in\cC_m$, the integer $n_\alpha$ is necessarily equal to $\sum_{k=1}^M\alpha_k$.
Since $\phi_i$ and $\phi_j$ are linearly independent for $i\neq j$, any $\alpha\in\bZ^{M;2}\setminus 0$ belongs
to at most one of the sets $\cC_m$ and $n_\alpha \neq 0$ if $\alpha \neq 0$.

Elements $\alpha\in\cC_m$ with two nonzero components correspond to \emph{resonances}. Resonances are
generated in products like $\sigma_{p,k}(x,\frac{\phi_p}{\eps})\sigma_{r,k'}(x,\frac{\phi_r}{\eps})$, which arise
from the quadratic term in \eqref{3}(c), whenever there exists a relation of the form
\begin{equation*}%\label{b3z}
n_m\phi_m=n_p\phi_p+n_r\phi_r, \text{ where }m\in\{1,\dots,M\}\setminus \{p,r\}\text{ and }n_m,n_p,n_r\in\bZ.
\end{equation*}
We then refer to $(\phi_m$, $\phi_p$, $\phi_r)$ as a triple of resonant phases. This relation implies, for example,
that $\phi_p$ oscillations interact with $\phi_r$ oscillations to produce $\phi_m$ oscillations.

\begin{defn}\label{b2z}
We define the continuous projector\footnote{The continuity of $E$ is shown in \cite{CGW1}, Remark 2.5.}
$E:H^{s;2}(\overline{\bR}^{d+1}_+\times\mathbb{T}^M)\to H^{s;1}(\overline{\bR}^{d+1}_+\times\mathbb{T}^M)$,
$s \ge 0$, by
\begin{align}\label{b4e}
E =E_0 +\sum^M_{m=1}E_m, \text{ where } E_0 \cV :=V_0 \text{ and } 
E_{m}\cV :=\sum_{\alpha\in\cC_m\setminus 0} P_m V_\alpha(x)e^{in_\alpha\theta_m},
\end{align}
for $P_m$ as in Lemma \ref{lem2}.
\end{defn}

For $\cL(\partial_\theta)$ as in \eqref{3a} we have that for $\cV^0\in H^{s;2}(\overline{\bR}^{d+1}_+\times
\mathbb{T}^M)$,
\begin{align}\label{b7}
E\cV^0=\cV^0\text{ if and only if }\cV^0\in H^{s;1}(\overline{\bR}^{d+1}_+\times\mathbb{T}^M)
\text{ and }\cL(\partial_\theta)\cV^0=0,
\end{align}
and \eqref{b7} in turn is equivalent to the property that $\cV^0$ has an expansion of the form
\begin{align}\label{b8}
\cV^0=\uv(x)+\sum^M_{m=1} \, \sum^{\nu_{k_m}}_{k=1}\sigma_{m,k}\left(x,\theta_m\right)r_{m,k} \, ,
\end{align}
for some real-valued functions $\sigma_{m,k}$. Moreover, since for any $m$,
\begin{align}\label{b8z}
L(d\phi_m)=\uomega_m \, I +\sum_{j=0}^{d-1} \beta_j \, A_j =\sum_{k\neq m}(\uomega_m-\uomega_k) P_k,
\end{align}
we have for $\cV\in H^{s;2}(\overline{\bR}^{d+1}_+\times\mathbb{T}^M)$:
\begin{align}\label{b8y}
E\cL(\partial_\theta)\cV=\cL(\partial_\theta)E\cV=0.
\end{align}

We also need to introduce a partial inverse $R$ for $\cL(\partial_\theta)$.  We begin by defining
\begin{equation*}%\label{b8w}
R_m=\sum_{k\neq m}\frac{1}{\uomega_m-\uomega_k}P_k,
\end{equation*}
which in view of \eqref{b8z} satisfies
\begin{align}\label{b8x}
L(d\phi_m)R_m=R_mL(d\phi_m)=I-P_m.
\end{align}
The operator $R$ is defined formally at first on functions $\cV(x,\theta)=\sum_{\alpha\in\bZ^{M;2}} V_\alpha(x)
e^{i\alpha \cdot \theta}$ of $H^{s;2}(\overline{\bR}^{d+1}_+\times\mathbb{T}^M)$ by
\begin{align}\label{R}
R\cV :=\sum_{\alpha\in\bZ^{M;2}} R(\alpha)V_\alpha(x)e^{i\alpha \cdot \theta}
\end{align}
where
\begin{align}\label{b8u}
R(\alpha) :=\begin{cases}
\frac{1}{i n_\alpha}R_m\text{ if }\alpha\in\cC_m\setminus \{ 0\}, \\
0 \text{ if }\alpha=0, \\
\cL(i\alpha)^{-1}\text{ if }\alpha \notin \cC, \end{cases}
\text{ and } \cL(i\alpha) :=i\sum^M_{j=1}L(d\phi_j)\alpha_j=iL(d(\alpha\cdot\phi)).
\end{align}

\begin{rem}
\textup{The operator $R$ is well-defined on functions $\cV\in H^{s;2}(\overline{\bR}^{d+1}_+\times\mathbb{T}^M)$
whose spectrum contains only finitely many noncharacteristic modes, and then $R\cV$ lies in the same space.
Otherwise, there can be a problem with small divisors; the possibility of there being infinitely many noncharacteristic
modes $\alpha$ for which $\det L(d(\alpha\cdot\phi))$ is close to zero can prevent convergence of \eqref{R} in
$H^{t;2}(\overline{\bR}^{d+1}_+\times\mathbb{T}^M)$  for any $t$.}
\end{rem}

It follows readily from \eqref{b8x} that for $\cF\in H^{s;1}(\overline{\bR}^{d+1}_+\times\mathbb{T}^M)$, $s>0$,
\begin{align}\label{b8t}
\cL(\partial_{\theta})R\cF=R\cL(\partial_{\theta})\cF=(I-E)\cF.
\end{align}
Such $\cF$ have no noncharacteristic modes. Along with \eqref{b8y} equation \eqref{b8t} implies

\begin{prop}\label{b8r}
Suppose $\cF\in H^{s;1}(\overline{\bR}^{d+1}_+\times\mathbb{T}^M)$, $s\ge 0$. Then the equation
$\cL(\partial_\theta) \cV=\cF$ has a solution $\cV\in H^{s;1}(\overline{\bR}^{d+1}_+\times\mathbb{T}^M)$
if and only if $E\cF=0$.
\end{prop}

By applying the operators $E$ and $R$ to the equations \eqref{3} and using \eqref{b7}, \eqref{b8y}, and \eqref{b8t},
we obtain:
\begin{align}\label{b8s}
\begin{split}
&(a)\;E\cV^0=\cV^0,\\
&(b)\;E(L(\partial)\cV^0+D(0)\cV^0)=0,\\
&(c)\;B\cV^0=0\text{ on }x_d=0,\theta=(\theta_0,\dots,\theta_0),\\
&(d)\;\cV^0=0\text{ in } t<0.
\end{split}
\end{align}
and
\begin{align}\label{6}
\begin{split}
&(a)\;(I-E)\cV^1+R(L(\partial)\cV^0+D(0)\cV^0)=0,\\
&(b)\;E\left(L(\partial)\cV^1+D(0)\cV^1+(\partial_vD(0) \, \cV^0)\cV^0 \right)=0,\\
&(c)\;B\cV^1+(\partial_v \psi(0)\cV^0) \cV^0=G \text{ on }x_d=0, \theta=(\theta_0,\dots,\theta_0),\\
&(d)\;\cV^1=0\text{ in }t<0.
\end{split}
\end{align}

\begin{rem}\label{rema}
\textup{(a) Since $E\cV^0=\cV^0$ the function $L(\partial)\cV^0+D(0)\cV^0$ in \eqref{6}(a) has \emph{no}
noncharacteristic modes so the action of $R$ on this function is well-defined.}

\textup{(b) It is easy to check that functions $\cV^0$, $\cV^1$ belonging to $H^{s;1}(\overline{\bR}^{d+1}_+
\times\mathbb{T}^M)$, $s>\frac{d+3}{2}$, and  satisfying \eqref{b8s} and \eqref{6}(a) also satisfy \eqref{3}(a),(b)
and \eqref{4}. Equation \ref{6}(b) and Proposition \ref{b8r} suggest that we might obtain a solution of \eqref{3}(c)
by taking
\begin{equation*}%\label{b8q}
(I-E)\cV^2=-R \, \Big( L(\partial)\cV^1+D(0)\cV^1+(\partial_vD(0) \cV^0) \cV^0 \Big).
\end{equation*}
There are two problems with this.  First, the quadratic term $(\partial_vD(0)\cV^0) \cV^0$ generally has \emph{infinitely}
many noncharacteristic modes, so one should expect a problem with small divisors. Second, the statement \eqref{b8t}
and Proposition \ref{b8r} are both \emph{not} true when $\cF\in H^{s;2}(\overline{\bR}^{d+1}_+\times\mathbb{T}^M)$,
even if $\cF$ has finitely many noncharacteristic modes\footnote{This is because of the fact that for any $k\in \bZ
\setminus \{0 \}$, there can be many $\alpha\in (\cC_m\setminus 0)\cap\bZ^{M;2}$ such that $n_\alpha=k$. See the
proof of Proposition \ref{solve}.}. These difficulties affect the error analysis and are discussed further in section
\ref{errori}.}
\end{rem}

\emph{\quad}To determine the equations satisfied by the individual profiles $\uv(x)$, $\sigma_{m,k}(x,\theta_m)$ in
the expansion \eqref{b8} of $\cV^0$,  we first refine the decomposition of the projector $E$ in \eqref{b4e}. For each
$m\in\{1,\dots,M\}$ we let
\begin{equation*}%\label{b16}
\ell_{m,k}, \;k=1,\dots,\nu_{k_m}
\end{equation*}
denote a basis of real vectors for the left eigenspace of the real matrix
\begin{align}\label{b16a}
i\cA(\beta) =\utau \, A_0+\sum_{j=1}^{d-1} \ueta_j \, A_j
\end{align}
associated to the eigenvalue $-\uomega_m$, chosen to satisfy
\begin{equation*}%\label{b17}
\ell_{m,k} \cdot r_{m',k'}=\begin{cases}1, &\text{ if }m=m'\text{ and }k=k',\\
0, &\text{ otherwise.}
\end{cases}
\end{equation*}
For $v\in\bC^N$ set
\begin{equation*}%\label{b18}
P_{m,k}v := (\ell_{m,k}\cdot v)r_{m,k} \text{ (no complex conjugation here).}
\end{equation*}
We can now write
\begin{equation*}%\label{b19}
E =E_0 +\sum_{m=1}^M \sum_{k=1}^{\nu_{k_m} }E_{m,k},
\end{equation*}
where $E_{m,k} :=P_{m,k}E_m.$  When the multiplicity $k=1$ we write $E_m$ instead of $E_{m,1}$ and do similarly
for $\ell_{m,k}$, $r_{m,k}$ and so on.

The following lemma, which is a slight variation on a well-known result \cite{L}, is included for the sake of completeness:

\begin{lem}\label{b21}
Suppose $E\cV^0=\cV^0$ and that $\cV^0$ has the expansion \eqref{b8}. Then
\begin{equation*}%\label{b22}
E_{m,k}(L(\partial)\cV^0)=(X_{\phi_m}\sigma_{m,k}) \, r_{m,k}
\end{equation*}
where $X_{\phi_m}$ is the characteristic vector field associated to $\phi_m$\footnote{The vector field $X_{\phi_m}$
is a constant multiple of the vector field $\partial_t + \mathbf{v}_m\cdot \nabla_{x''}$ computed by Lax for the Cauchy
problem, where $\mathbf{v}_m$ is the group velocity defined in Definition \ref{def2}.}:
\begin{equation*}%\label{b23}
X_{\phi_m} :=\partial_d +\sum^{d-1}_{j=0}-\partial_{\xi_j}\omega_m(\beta) \, \partial_j \, .
\end{equation*}
\end{lem}

\begin{proof}
For  $\xi'\in\cH$ near $\beta$, let $-\omega_m(\xi')$   be the eigenvalues $i\cA(\xi')$ \eqref{b16a} and $P_m(\xi')$ the
corresponding projectors (these objects depend smoothly on $\xi'$ near $\beta$ thanks to the analysis of \cite{Met}).
Differentiate the equation
\begin{equation*}%\label{b24}
\left(\omega_m(\xi')I+\sum_{j=0}^{d-1}{A}_j\xi_j\right)P_m(\xi')=0
\end{equation*}
with respect to $\xi_j$, evaluate at $\beta$,  apply $P_m$ on the left to obtain
\begin{equation*}%\label{b24a}
P_m A_j P_m=-\partial_{\xi_j}\omega_m(\beta)P_m,
\end{equation*}
from which the lemma readily follows.
\end{proof}

By Assumption \ref{assumption3} we know that the vector space $\ker B \, \cap \, \E^s(\beta)$ is one-dimensional;
moreover, it admits a real basis because $B$ has real coefficients and $\E^s(\beta)$ has a real basis. This vector
space is therefore spanned by a vector $e \in \R^N \setminus \{ 0\}$ that we can decompose in a unique way by
using Lemma \ref{lem1}:
\begin{equation}
\label{defe}
\ker B \, \cap \, \E^s(\beta) = \text{\rm Span } \{e\} \, ,\quad
e = \sum_{m \in {\mathcal I}} e_m \, ,\quad P_m \, e_m =e_m \, .
\end{equation}
Each vector $e_m$ in \eqref{defe} has real components. We also know that the vector space $B \, \E^s(\beta)$
is $(p-1)$-dimensional. We can therefore write it as the kernel of a real linear form:
\begin{equation}
\label{defb}
B \, \E^s(\beta) = \left\{ X \in \C^p \, ,\quad b \cdot X = 0 \right\} \, ,
\end{equation}
for a suitable vector $b \in \R^p \setminus \{ 0\}$.

Any function $\cV(x,\theta)\in H^{s;2}(\overline{\bR}^{d+1}_+\times\mathbb{T}^M)$ can be decomposed
\begin{equation*}%\label{decomp}
\cV =\underline{\cV} +\cV_{inc} +\cV_{out} +\cV_{nonch}=\underline{\cV}+\cV^*,
\end{equation*}
where the terms correspond respectively to the parts of the Fourier series \eqref{b1} with $\alpha=0$, $\alpha$
incoming, $\alpha$ outgoing, and $\alpha$ noncharacteristic\footnote{Here we say $\alpha$ is incoming if
$\alpha\in\cC_m\setminus 0$ for an index $m$ such that $\phi_m$ is an incoming phase.}.

\begin{prop}\label{7a}
Suppose $\cV^0\in H^{s;2}(\overline{\bR}^{d+1}_+\times\mathbb{T}^M)$, $s\geq 1$, is a solution of \eqref{b8s}. Then
\begin{align*}%\label{7}
%\begin{split}
%&(a)\;
&\underline\cV^0=0, \quad \cV^0_{out}=0, \quad \cV^0_{nonch}=0, \quad \text{ and so }\cV^0=\cV^0_{inc}=E\cV^0_{inc},\\
%&(b)\;
&\cV^0(x',0,\theta_0,\dots,\theta_0)=a(x',\theta_0) \, e \text{ for some unknown periodic function }a \text{ with mean }0.
%\end{split}
\end{align*}
\end{prop}

\begin{proof}
Since $E\cV^0=\cV^0$, we have $\cV^0_{nonch}=0$. Applying $E_0$ to the problem \eqref{b8s}, we find that the
mean value $\underline\cV^0$ satisfies the weakly stable boundary problem
\begin{align*}%\label{7b}
%\begin{split}
& L(\partial) \underline\cV^0 +D(0) \underline\cV^0=0\\
& B\underline\cV^0=0\text{ on }x_d=0\\
& \underline\cV^0=0\text{ in }t<0.
%\end{split}
\end{align*}
By the well-posedness result of \cite{C} we have $\underline\cV^0=0$.

Lemma \ref{b21} implies that outgoing profiles $\sigma_{m,k}$, $m\in\cO$, in the expansion \eqref{b8} of $\cV^0$
satisfy problems of the form
\begin{align*}%\label{7c}
%\begin{split}
& X_{\phi_m}\sigma_{m,k} +\sum_{k'=1}^{\nu_{k_m}} (\ell_{m,k}\cdot D(0) r_{m,k'}) \sigma_{m,k'}=0,\\
& \sigma_{m,k}=0\text{ in }t<0,
%\end{split}
\end{align*}
where $X_{\phi_m}$ is an outgoing vector field. Thus, $\sigma_{m,k}=0$ for all $k=1,\dots,\nu_{k_m}$.

Part (b) follows immediately from the boundary condition in \eqref{b8s} and \eqref{defe}.
\end{proof}

Since $\cV^0=\cV^0_{inc}$,  we obtain from \eqref{6}(a)
\begin{equation*}%\label{8}
(I-E)\cV^1=(I-E)\cV^1_{inc}=-R(L(\partial)\cV^0+D(0)\cV^0),
\end{equation*}
so
\begin{equation*}%\label{9}
\cV^1=\underline{\cV}^1+\cV^1_{inc}+\cV^1_{out} \in H^{s;1},\text{ where }E\cV^1_{out}=\cV^1_{out}.
\end{equation*}
Next decompose the boundary condition \eqref{6}(c):
\begin{multline}\label{10}
BE\cV^1_{inc} =G^*-[(\partial_v \psi(0)\cV^0) \cV^0)]^*-B\cV^1_{out}-B(I-E)\cV^1_{inc} \\
=G^*-[(\partial_v \psi(0)\cV^0)\cV^0)]^*-B\cV^1_{out}+BR(L(\partial)\cV^0+D(0)\cV^0) \, .
\end{multline}

\begin{rem}\label{res}
\textup{(a)  If $\cV^1_{out}|_{x_d=0,\theta_j=\theta_0}$ were known, one could write down a transport equation
for $a(x',\theta_0)$ which is determined by the solvability condition for \eqref{10} implied by \eqref{defb}:
\begin{align}\label{10a}
b\cdot\left(G^*-[(\partial_v \psi(0)\cV^0)\cV^0)]^*-B\cV^1_{out}+BR(L(\partial)\cV^0+D(0)\cV^0)\right)=0.
\end{align}
However, the presence of the term $E \left( (\partial_vD(0)\cV^0)\cV^0\right)$ in \eqref{6}(b) implies that two
incoming modes in $\cV^0_{inc}$ (which is still unknown) can resonate to produce an outgoing mode that will
affect $\cV^1_{out}$. Thus, we do not know $\cV^1_{out}|_{x_d=0,\theta_j=\theta_0}$, and we see that the
nonlinear boundary equation \eqref{10a} is coupled to the nonlinear interior equation \eqref{6}}.

\textup{(b) If the phases are such that an outgoing mode can never be produced by a product of two incoming
modes, then $\cV^1_{out}$ can be determined from \eqref{6} to be $0$, and one can proceed as in \cite{CG}
to solve for $a$ without having to use Nash-Moser iteration.}
\end{rem}

The key subsystem to focus on now is (recalling $\cV^0=E\cV^0=\cV^0_{inc}$ and writing with obvious notation
$E=E_0+E_{inc}+E_{out}$):
\begin{align}\label{12}
\begin{split}
&(a)\, E_{inc}(L(\partial)\cV^0_{inc}+D(0)\cV^0_{inc})=0,\\
&(b)\, E_{out}\left(L(\partial)\cV^1_{out} +D(0)\cV^1_{out} +(\partial_vD(0)\cV^0_{inc})\cV^0_{inc}\right)=0,\\
&(c)\, b \cdot \left( G^*-[(\partial_v \psi(0)\cV^0)\cV^0)]^*-B\cV^1_{out}+BR(L(\partial)\cV^0_{inc}+D(0)\cV^0_{inc})
\right)=0,\\
&(d)\, \cV^0_{inc}(x',0,\theta_0,\dots,\theta_0)=a(x',\theta_0)e,
\end{split}
\end{align}
where $\cV^0_{inc}$ and $\cV^1_{out}$ both vanish in $t<0$.

A formula for $\cV^0_{inc}$ in terms of $a(x',\theta_0)$ can be determined by solving transport equations using
\eqref{12}(a), and that formula can be plugged into \eqref{12}(b) to get $\cV^1_{out}$ in terms of $a$. Thus, the
subsystem \eqref{12} can be expressed as a very complicated nonlinear, nonlocal equation for the single unknown
$a$. This is done in Appendix \ref{exeuler} for  a strictly hyperbolic example with only one resonance. However, that 
is not the way we solve \eqref{12}; instead we solve the subsystem in its given form by iteration. Picard iteration does 
not work; there is a loss of derivatives from one iterate to the next (because of $R$), so we use a Nash-Moser scheme. 
An essential point is to take advantage of the smoothing property of the interaction integrals that pick out resonances 
in $E_{out}((\partial_vD(0)\cV^0_{inc}) \cV^0_{inc})$\footnote{Interaction integrals are similar to convolution integrals.}; 
that property allows us to get tame estimates in section \ref{profile}.

An important tool in solving the subsystem \eqref{12} is the following result from \cite{CG}, which will allow us to
write the boundary equation \eqref{12}(c) as a transport equation for $a(x',\theta_0)$.

\begin{prop}[\cite{CG}, Proposition 3.5]\label{xlop}
Let the vectors $b$ and $e_m$ be as in \eqref{defb}, \eqref{defe}, and let $\sigma(\zeta)$ be the function appearing
in Assumption \ref{assumption3}.  Then there exists a nonzero real number $\kappa$ such that
\begin{align*}%\label{12a}
& R_m P_m=0 \text{ for all }m\in\{1,\dots, M\}\\
& b\cdot B \sum_{m\in\cI}R_m A_0 e_m=\kappa \partial_\tau\sigma(\utau,\ueta) \text{ and }
\partial_\tau\sigma(\utau,\ueta)=1\, ,\\
& b\cdot B\sum_{m\in \cI}R_mA_j e_m=\kappa \partial_{\eta_j}\sigma(\utau,\ueta),\;j=1,\dots,d-1.
\end{align*}
and thus
\begin{equation*}%\label{12b}
b\cdot B\sum_{m\in\cI}R_m L(\partial)e_m=\kappa\left(\partial_\tau\sigma(\utau,\ueta)\partial_t
+\sum^{d-1}_{j=1}\partial_{\eta_j}\sigma(\utau,\ueta)\partial_{x_j}\right) =: X_{Lop}.
\end{equation*}
\end{prop}

\noindent Taking note of the factor $\frac{1}{in_\alpha}$ in the definition \eqref{b8u} of $R$, we obtain the
immediate corollary:

\begin{cor}\label{12c}
The boundary term $b\cdot B \, R \, L(\partial)\cV^0_{inc}$ in \eqref{12} may be written
\begin{equation*}%\label{12d}
b\cdot B \, R \, L(\partial)\cV^0_{inc}=X_{Lop}\cA,
\end{equation*}
where $\cA(x',\theta_0)$ is the unique function with mean $0$ in $\theta_0$ such that $\partial_{\theta_0}\cA=a$.
\end{cor}

\begin{rem}\label{radiate}
\textup{Proposition \ref{xlop} shows that propagation in the boundary, which is described by $a(x',\theta_0)$, is
governed by the ($x$-projection of the) Hamiltonian vector field associated to the Lopatinskii determinant. Since
$\cV^0(x',0,\theta_0,\dots,\theta_0)=a(x',\theta_0)$, this shows that waves of amplitude $O(1)$ emanate from the
radiating boundary wave defined by $a$.}
\end{rem}

After  \eqref{12} is solved $\cV^0$ is known, so $\underline{\cV}^1$, $\cV^1_{out}$,  and $(I-E)\cV^1_{inc}$ can
now be determined by returning to the full system \eqref{6}.  The trace of $E\cV^1_{inc}$ is not yet determined;
one should make a \emph{choice} of $E\cV^1_{inc}|_{x_d=0,\theta_j=\theta_0}$ such that \eqref{10} holds, and
then solve for $E\cV^1_{inc}$ using \eqref{6}(b). A precise description of the regularity of $\cV^0$ and $\cV^1$
is given in Theorem \ref{mainprof}. The last piece of the corrected approximate solution, $\eps^2 \, \cU^2_p$ in
\eqref{a10} is discussed next.

\subsection{Error analysis}
\label{errori}

\emph{\quad}Given a periodic function $f(x,\theta)$, where $\theta=(\theta_1,\dots,\theta_M)$, let us denote
\begin{equation*}%\label{20}
f(x,\theta)|_{\theta\rightarrow(\theta_0,\xi_d)}:=f(x,\theta_0+\uomega_1\xi_d,\dots,\theta_0+\uomega_M\xi_d);
\end{equation*}
so we have
\begin{equation*}%\label{21}
f(x,\theta)|_{\theta\rightarrow(\frac{\phi_0}{\eps},\frac{x_d}{\eps})}=f \left(x,\frac{\phi}{\eps} \right).
\end{equation*}
Taking the profiles $\cV^0$, $\cV^1$ constructed in Theorem \ref{mainprof}, if we define
\begin{equation*}%\label{22}
\cU^b_\eps(x,\theta_0):=\left(\cV^0(x,\theta)+\eps\cV^1(x,\theta)\right)|_{\theta\rightarrow(\theta_0,\frac{x_d}{\eps})},
\end{equation*}
we find that $\cU^b_\eps$ satisfies the singular system
\begin{align}\label{23}
\begin{split}
&(a)\;\bL_{\eps}(\cU^b_\eps) :=
\partial_d \cU^b_\eps+\mathbb{A}\left(\partial_{x'}+\frac{\beta \partial_{\theta_0}}{\eps}\right)\cU^b_\eps
+D(\eps \cU^b_\eps)Ê\, \cU^b_\eps=O(\eps),\\
&(b)\;\psi(\eps \, \cU^b_\eps) \, \cU^b_\eps=\eps \, G(x',\theta_0)+O(\eps^2)\text{ on }x_d=0,\\
&(c)\;\cU^b_\eps=0\text{ in }t<0,
\end{split}
\end{align}
where the error terms refer to norms in $E^s$ and $H^t$ spaces whose orders are made precise in section \ref{error}.
For example, \eqref{23} follows directly from the profile equations \eqref{3}(a),(b), together with the identity
\begin{align}\label{24}
\bL_{\eps} \left(f(x,\theta)|_{\theta\rightarrow(\theta_0,\frac{x_d}{\eps})}\right)
=\dfrac{1}{\eps} \left(\cL(\partial_{\theta})f(x,\theta)\right)|_{\theta\rightarrow(\theta_0,\frac{x_d}{\eps})}
+\left(L(\partial)f(x,\theta)\right)|_{\theta\rightarrow(\theta_0,\frac{x_d}{\eps})}
+(D(\eps f)f)|_{\theta\rightarrow(\theta_0,\frac{x_d}{\eps})}.
\end{align}

Since our basic estimate for the linearized singular system exhibits a loss of one singular derivative (basically,
we lose a $1/\eps$ factor), the accuracy in \eqref{23}(a) is not good enough to conclude that
$|U_\eps-\cU^b_\eps|_{L^\infty (x,\theta_0)}$ is small (the error terms are only $O(\eps)$). To improve the
accuracy we construct an additional corrector $\cU^2_p(x,\theta_0,\xi_d)$ and replace $\cU^b_\eps$ by
\begin{align}\label{25}
\cU_\eps(x,\theta_0) :=\left(\cV^0(x,\theta)+\eps\cV^1(x,\theta)\right)|_{\theta\rightarrow(\theta_0,\frac{x_d}{\eps})}
+\eps^2 \, \cU^2_p(x,\theta_0,\frac{x_d}{\eps}).
\end{align}
In constructing $\cU^2_p$ we deal with the first (small divisor) problem described in Remark \ref{rema}(b)
by approximating $\cV^0$ and $\cV^1$ by trigonometric polynomials $\cV^0_p$ and $\cV^1_p$ to within
an accuracy $\delta>0$ in appropriate Sobolev norms, and seek $\cU^2_p$ in the form of a trigonometric
polynomial\footnote{Trigonometric polynomial approximations were already used to deal with small divisor
problems in the error analysis of \cite{JMR}.}. To deal with the second (solvability) problem, we use the
following Proposition, which allows us to use the profile equation \eqref{6}(b) as a solvability condition, in
spite of the failure of Proposition \ref{b8r} when  $\cF\in H^{s;2}(\overline{\bR}^{d+1}_+\times\mathbb{T}^M)$.
We define
\begin{equation*}%\label{26}
\cL_0(\partial_{\theta_0},\partial_{\xi_d}):=L(d\phi_0)\partial_{\theta_0}+\partial_{\xi_d}.
\end{equation*}

\begin{prop}\label{solve}
Suppose $F(x,\theta)\in H^{s;2}(\overline{\bR}^{d+1}_+\times\mathbb{T}^M)$ has a Fourier series which is a finite
sum and that $EF=0$. Then there exists a solution of the equation
\begin{align}\label{27}
\cL_0(\partial_{\theta_0},\partial_{\xi_d})\cU(x,\theta_0,\xi_d)=F(x,\theta)|_{\theta\rightarrow(\theta_0,\xi_d)}
\end{align}
in the form of a trigonometric polynomial  in $(\theta_0,\xi_d)$ of the form
\begin{align}\label{28}
\cU(x,\theta_0,\xi_d)
=\sum_{(\kappa_0,\kappa_d)\in\cJ}U_{\kappa_0,\kappa_d}(x) e^{i\kappa_0\theta_0+i\kappa_d\xi_d},
\end{align}
where $\cJ$ is a finite subset of $\bZ\times\bR$ and the coefficients $U_{\kappa_0,\kappa_d}$ lie in
$H^s(\overline{\bR}^{d+1}_+).$
\end{prop}

The proof is given in section \ref{error}. Observe that $\cU$ is periodic in $\theta_0$ but almost periodic in
$(\theta_0,\xi_d)$. Proposition \ref{solve} is applied to solve the equation
\begin{equation*}%\label{29}
\cL_0(\partial_{\theta_0},\partial_{\xi_d})\cU^2_p =\left[-(I-E)\left(L(\partial)\cV^1_p+D(0)\cV^1_p
+(\partial_vD(0)\cV^0_p)\cV^0_p\right)\right]|_{\theta\rightarrow(\theta_0,\xi_d)}
\end{equation*}
With this choice of $\cU^2_p$ we show in section \ref{error} that the new approximate solution
$\cU_\eps(x,\theta_0)$ in \eqref{25} satisfies instead of \eqref{23} the singular system
\begin{align}\label{30}
\begin{split}
&(a)\;\bL_\eps(\cU_\eps)=O\left(\eps(K\delta+C(\delta)\eps)\right),\\
&(b)\;\psi(\eps\cU_\eps) \, \cU_\eps-\eps G(x',\theta_0)=O(\eps^2C(\delta))\text{ on }x_d=0,\\
&(c)\;\cU_\eps=0\text{ in }t<0,
\end{split}
\end{align}
where the errors in \eqref{30}(a),(b) are measured in appropriate norms. Now one can apply our basic estimate
\eqref{i16} for the linearized singular problem to conclude that the difference between exact
and approximate solutions of the semilinear singular system \eqref{15} satisfies for some constants  $C(\delta)$
and $K$:
\begin{equation*}%\label{31}
|U_\eps(x,\theta_0)-\cU_\eps(x,\theta_0)|_{E^s}\leq K\, \delta +C(\delta) \, \eps, \text{ for some } s>\dfrac{d+1}{2}.
\end{equation*}
This estimate clearly implies the conclusion of Theorem \ref{main} by choosing first $\delta>0$ small enough
and then letting $\eps$ tend to zero.

\subsection{Remarks on quasilinear problems}

In this article, we are able to rigorously justify a weakly nonlinear regime with amplification for \emph{semilinear}
hyperbolic initial boundary value problems. Our assumptions only deal with the principal part of the operators,
meaning that we only assume a weak stability property for the problem $(L(\partial),B)$
obtained by linearizing at the origin and dropping the zero order term in the hyperbolic system. The weak stability
is of WR type in the terminology of \cite{BRSZ}. Despite the weak regime that we consider ($O(\eps^2)$ source
term at the boundary and $O(\eps)$ solution), the leading profile equation displays some nonlinear features. We
emphasize that the regime that we consider here is exactly one power of $\eps$ weaker than the weakly nonlinear
regime for the semilinear Cauchy problem or for semilinear uniformly stable boundary value problems. As in \cite{CG},
this power of $\eps$ corresponds exacty to the loss of one derivative in the energy estimates.

We believe that the techniques developed here can be extended to give a rigorous justification of weakly nonlinear
geometric optics with amplification for \emph{quasilinear} hyperbolic initial boundary value problems of the form
\begin{align*}
\begin{split}
&\dt v +\sum_{j=1}^d B_j(v) \, \partial_j v +f_0(v)=0,\\
&\phi(v)=\eps^3 \, G \left( x',\frac{x'\cdot\beta}{\eps} \right) \text{ on }x_d=0,\\
&v=0 \text{ and }G=0\text{ in }t<0.
\end{split}
\end{align*}
The corresponding solution $v_\eps$ would be of amplitude $O(\eps^2)$. In particular the arguments used in
Section \ref{exact} to obtain uniform estimates with a loss of one singular derivative for the singular initial boundary
value problem might be extended to the corresponding singular quasilinear problem. There are however several new 
obstacles along the way, one of which is to extend  the singular pseudodifferential calculus of
\cite{CGW2} in order to obtain a two-terms expansion of $(1,0)$ and $(0,1)$ compositions. The weaker scaling
($\eps^2$ in place of $\eps$) should be sufficient to obtain the appropriate results. Let us observe that for
$O(\eps^2)$ solutions, the principal part of the hyperbolic operator has coefficients that are uniformly bounded in
$W^{2,\infty}$, which is precisely the regularity needed in \cite{C1,C} to obtain a priori estimates and well-posedness.
The leading profile equation obtained in this quasilinear framework is very similar to the one we
have derived here, and we thus believe that a weak well-posedness result using Nash-Moser iteration should prove
the existence of the leading profile. For all the above reasons, we thus believe that the $\eps^3$ source term on the
boundary is the relevant "weakly nonlinear regime with amplification" in the quasilinear case, and we postpone the
verification of the many technical details to a future work.

\newpage
\section{Exact oscillatory solutions on a fixed time interval}
\label{exact}

\subsection{The basic  estimate for the linearized singular system}
\label{est}

\emph{\quad}The goal of this section is to prove Proposition \ref{i5z} below and its time-localized version, that is,
Proposition \ref{i14}. These propositions provide the a priori estimates for the linearized singular system that form
the basis for the Nash-Moser iteration of section \ref{nashex} and the error analysis of section \ref{error}.

We begin by gathering some of the notation for spaces and norms that is needed below.

\begin{nota}\label{spaces}
Here  we take $s\in \bN=\{0,1,2,\dots\}$.

(a)\;Let $\Omega:=\overline{\mathbb{R}}^{d+1}_+\times\mathbb{T}^1$, $\Omega_T:=\Omega\cap\{-\infty<t<T\}$,
$b\Omega:=\mathbb{R}^d\times\mathbb{T}^1$, $b\Omega_T:=b\Omega\cap \{-\infty<t<T\}$,  and set $\omega_T
:=\overline{\mathbb{R}}^{d+1}_+\cap\{-\infty<t<T\}$.

(b)\;Let $H^s\equiv H^s(b\Omega)$, the standard Sobolev space with norm $\langle V(x',\theta_0)\rangle_s$.
For $\gamma\geq 1$ we set $H^s_\gamma:=e^{\gamma t} \, H^s$ and $\langle V\rangle_{s,\gamma} := \langle
e^{-\gamma t} \, V \rangle_s$.

(c)\;$L^2H^s\equiv L^2(\overline{\mathbb{R}}_+,H^s(b\Omega))$ with norm $|U(x,\theta_0)|_{L^2H^s} \equiv
|U|_{0,s}$ given by
\begin{equation*}%\label{i1}
|U|_{0,s}^2=\int^\infty_0|U(x',x_d,\theta_0)|_{H^s(b\Omega)}^2dx_d.
\end{equation*}
The corresponding norm on $L^2H^s_\gamma$ is denoted $|V|_{0,s,\gamma}$.

(d)\;$CH^s\equiv C(\overline{\mathbb{R}}_+,H^s(b\Omega))$ denotes the space of continuous bounded
functions of $x_d$ with values in $H^s(b\Omega)$, with norm $|U(x,\theta_0)|_{CH^s} =|U|_{\infty,s} :=
\sup_{x_d\geq 0} |U(.,x_d,.)|_{H^s(b\Omega_T)}$ (note that $CH^s\subset L^\infty H^s$). The corresponding
norm on $CH^s_\gamma$ is denoted $|V|_{\infty,s,\gamma}$.
%(e)\; Similarly, $H^s_T\equiv H^s(b\Omega_T)$ with norm $\langle V\rangle_{s,T}$ and $L^2H^s_T\equiv 
%L^2(\overline{\mathbb{R}}_+,H^s_T)$, $CH^s_T\equiv C(\overline{\mathbb{R}}_+,H^s_T)$ have norms 
%$|U|_{0,s,T}$, $|U|_{\infty,s,T}$ %respectively.

%(f)\;$L^2H^s_\gamma\equiv L^2(x_d,H^s_\gamma)$ and $CH^s_\gamma\equiv C(x_d,H^s_\gamma)$ 
%have norms $|U(x,\theta)|_{0,k,\gamma}$, %$|U|_{\infty,k,\gamma}$ respectively.

%(g)\;$\mathbb{H}^k_T=\{U\in CH^k_T\cap L^2H^{k+1}_T:U|_{x_N=0}\in H^{k+1}_T\}$ with the norm
%\[
%\|U\|_{k,T}=|U|_{\infty,k,T}+|U|_{0,k+1,T}+\sqrt{T}\langle U\rangle_{k+1,T}.
%\]

%(f)\;$L^\infty W^{1,\infty}\equiv L^\infty(x_d,W^{1,\infty}(b\Omega))$ with norm $|U|_{L^\infty W^{1,\infty}} 
%\equiv |U|^*$. We also write $|U|_{L^\infty(\Omega)}=|U|_*$, $\langle V\rangle_{L^\infty(b\Omega)}= 
%\langle V\rangle_*$, $\langle V\rangle_{W^{1,\infty}(b\Omega)}\equiv\langle V\rangle^*$, $|U|_{L^\infty(\Omega_T)}
%=|U|_*$, etc..

%(i)\; For $k,l\in\mathbb{N}$ denote by $h:(\mathbb{R}_+)^l\to\mathbb{R}_+$ or $h_k:(\mathbb{R}_+)^l\to\mathbb{R}_+$ 
%an increasing %function of each of its arguments independent of $\eps$, $\gamma$.  $h$ and constants $C$ may 
%change from line to line or even %from term to term in the text.  $C(K)$ denotes a constant that depends on $K$.

%(j)\;For any function $U$, $U^\gamma\equiv e^{-\gamma x_0}U$.

(e) Let $M_0:=3d+5$ and define $C^{0,M_0} :=C(\overline{\mathbb{R}}_+,C^{M_0}(b\Omega))$ as the space of
continuous bounded functions of $x_d$ with values in $C^{M_0}(b\Omega)$, with norm $|U(x,\theta_0)|_{C^{0,M_0}}
:= |U|_{L^\infty W^{M_0,\infty}}$. Here $L^\infty W^{M_0,\infty}$ denotes the space $L^\infty(\overline{\bR}_+;
W^{M_0,\infty}(b\Omega))$.\footnote{The size of $M_0$ is determined by the requirements of the singular
calculus described in Appendix \ref{calc}.}

(f)The corresponding spaces on $\Omega_T$ are denoted $L^2H^s_T$, $L^2H^s_{\gamma,T}$, $CH^s_T$,
$CH^s_{\gamma,T}$ and $C^{0,M_0}_T$ with norms $|U|_{0,s,T}$, $|U|_{0,s,\gamma,T}$, $|U|_{\infty,s,T}$,
$|U|_{\infty,s,\gamma,T}$, and $|U|_{C^{0,M_0}_T}$ respectively. On $b\Omega_T$ we use the spaces
$H^s_T$ and $H^s_{\gamma,T}$ with norms $\langle U\rangle_{s,T}$ and $\langle U\rangle_{s,\gamma,T}$.

(g) All constants appearing in the estimates below are independent of $\eps$, $\gamma$, and $T$ unless such
dependence is explicitly noted.
\end{nota}

The linearization of the singular problem \eqref{15} at $U(x,\theta_0)$ has the form
\begin{align}\label{i3}
\begin{split}
&(a)\, \partial_d \dot U_\eps +\mathbb{A} \left( \partial_{x'}+\dfrac{\beta \partial_{\theta_0}}{\eps} \right)
\dot U_\eps +\cD (\eps U) \, \dot U_\eps=f(x,\theta_0) \quad \text{ on }\Omega \, ,\\
&(b)\, \cB (\eps U) \, \dot U_\eps|_{x_d=0} =g(x',\theta_0) \, ,\\
&(c)\, \dot U_\eps=0 \text{ in } t<0,
\end{split}
\end{align}
where the matrices $\cB(\eps U)$, $\cD(\eps U)$ are defined in \eqref{caaa}\footnote{Here and below we often
suppress the subscript $\eps$ on $\dot U$.}. Instead of \eqref{i3}, consider the equivalent problem satisfied by
$\dot U^\gamma :=e^{-\gamma t}\dot U$:
\begin{align}\label{i5}
\begin{split}
&\partial_d \dot U^\gamma +\mathbb{A} \left( (\partial_{t}+\gamma,\partial_{x''})
+\dfrac{\beta \, \partial_{\theta_0}}{\eps} \right) \dot U^\gamma +\cD (\eps U) \, \dot U^\gamma
=f^\gamma(x,\theta_0) \, ,\\
&\cB (\eps U) \, \dot U^\gamma|_{x_d=0} =g^\gamma(x',\theta_0) \, ,\\
&\dot U^\gamma=0 \text{ in } t<0 \, .
\end{split}
\end{align}
Below we let $\Lambda_D$ denote the singular Fourier multiplier (see \eqref{n5}) associated to the symbol
\begin{equation}\label{i5a}
\Lambda(X,\gamma) :=\left( \gamma^2+\left|\xi'+\frac{k\, \beta}{\eps}\right|^2 \right)^{1/2},\;
X := \xi'+\dfrac{k\, \beta}{\eps}.
\end{equation}
The basic  estimate for the linearized singular problem \eqref{i5} is given in the next Proposition. Observe that the
estimate \eqref{aprioriL2} exhibits a loss of one ``singular derivative" $\Lambda_D$. This is quite a high price to pay,
which counts as a factor $1/\eps$. In view of \cite[Theorem 4.1]{CG}, there is strong evidence that the loss below is
optimal.

\begin{prop}[Main $L^2$ linear estimate]\label{i5z}
We make the structural assumptions of Theorem \ref{main} and recall $M_0 =3d+5$. Fix $K>0$ and suppose
$|\eps \, \partial_d U|_{C^{0,M_0-1}} +|U|_{C^{0,M_0}} \leq K$ for $\eps \in (0,1]$. There exist positive constants
$\eps_0(K)>0$, $C(K)>0$ and $\gamma_0(K) \ge 1$ such that sufficiently smooth solutions $\dot U$ of the
linearized singular problem \eqref{i3} satisfy:\footnote{Note that the norms $|u|_{0,1}$ and $|\Lambda_D u|_{0,0}$
are not equivalent.}
\begin{align}\label{aprioriL2}
|\dot U^\gamma|_{0,0} +\dfrac{\langle \dot U^\gamma\rangle_0}{\sqrt{\gamma}}
\leq C(K) \left( \dfrac{|\Lambda_D f^\gamma|_{0,0}+|\nabla_{x'}f^\gamma|_{0,0}}{\gamma^2}
+\dfrac{\langle\Lambda_D g^\gamma \rangle_0 +\langle \nabla_{x'}g^\gamma \rangle_0}{\gamma^{3/2}} \right)
\end{align}
for $\gamma \geq \gamma_0(K)$, \;$0<\eps\leq \eps_0(K)$.

The same estimate holds if $\cB(\eps U)$ in \eqref{i3} is replaced by $\cB(\eps U,\eps\cU)$ and $\cD(\eps U)$ is
replaced by $\cD(\eps U,\eps\cU)$ as long as $|\eps \partial_d(U,\cU)|_{C^{0,M_0-1}} +|U,\cU|_{C^{0,M_0}} \leq
K$ for $\eps\in (0,1]$.
\end{prop}

\begin{cor}[Main $H^1_{tan}$ linear estimate]\label{estimH1}
Under the same assumptions as in Proposition \ref{i5z}, smooth enough solutions $\dot U$ of the linearized singular
problem \eqref{i3} satisfy:
\begin{align}\label{i6}
|\dot U^\gamma|_{\infty,0}+|\dot U^\gamma|_{0,1}+\frac{\langle \dot U^\gamma\rangle_1}{\sqrt{\gamma}}
\leq C(K) \left( \frac{|\Lambda_D f^\gamma|_{0,1}+|\nabla_{x'}f^\gamma|_{0,1}}{\gamma^2}
+\frac{\langle\Lambda_D g^\gamma \rangle_1 +\langle \nabla_{x'}g^\gamma \rangle_1}{\gamma^{3/2}} \right)
\end{align}
for $\gamma \geq \gamma_0(K)$, \;$0<\eps\leq \eps_0(K)$.
\end{cor}

\textbf{ Short guide to the proof.} The proof of Proposition \ref{i5z} is completed using the next three propositions,
each of which has the same hypotheses as Proposition \ref{i5z}. In the first step of the proof of Proposition \ref{i5z},
we choose a partition of unity defined by frequency cutoffs $\chi_i(\zeta)$, $i=1,\dots,N_1+N_2$, such that for
$i=1,\dots,N_1$ the function $\chi_i$ is supported near a point of the ``bad" set $\overline{\Upsilon}$, while
for $i>N_1$ the function $\chi_i$ is supported away from $\overline{\Upsilon}$. The estimates of $\chi_{i,D} \,
\dot U^\gamma$ for $i>N_1$ are done in Proposition \ref{kreisspiece}. For such indices, Kreiss symmetrizers
in the singular calculus are used to estimate $ \chi_{i,D} \dot U^\gamma$ without loss.

To handle the piece $\chi_{i,D} \dot U^\gamma$, $i\leq N_1$, where the loss occurs, we also choose an auxiliary
cutoff $\chi^e$ in the extended calculus with the properties \eqref{n31} and use the decomposition
\begin{align}\label{ia2}
\chi_{i,D}\dot U^\gamma =\chi^e_D \, \chi_{i,D} \dot U^\gamma +(1-\chi^e_D)\, \chi_{i,D} \dot U^\gamma \, .
\end{align}
The first term on the right in \eqref{ia2} is estimated in the Proposition \ref{ia3} while the second term in Proposition
\ref{hard}. Proposition \ref{hard} is the most difficult part in the analysis because in this frequency domain, some of
the commutators are poorly controlled.

\begin{proof}[Proof of Proposition \ref{i5z}]
\textbf{I) Partition of unity.} The compactness of $\overline\Upsilon$ (see Assumption \ref{nonlinbc}) and $\Sigma$
allows us to choose a finite open covering of $\Sigma$, $\cC=\{\cV_i\}_{i=1,\dots,N_1+N_2}$ such that $\{ \cV_i
\}_{i=1,\dots,N_1}$ covers $\overline{\Upsilon}$ and such that $\cup_{N_1+1}^{N_1+N_2}\cV_i$ is disjoint from
a neighborhood of $\overline{\Upsilon}$. Since $\overline{\Upsilon}\subset \cH$ we can arrange so that for each
$i\in\{1,\dots,N_1\}$ there is a conjugator $Q_{0,i}(\zeta)$\footnote{Recall the notation $\zeta=(\tau-i\gamma,\eta)$.
Sometimes we will also write $\zeta=(\xi',\gamma)$ to match the notation of \cite{CGW2}.} and diagonal matrix
$\bD_{1,i}(\zeta)$ satisfying \eqref{conjugator} in $\cV_i$. Moreover, we can choose a neighborhood $\cO$ of
$(0,0)\in\bR^{2N}$ such that for each $i \leq N_1$ there are functions $\sigma_i$, $P_{i,1}$, and $P_{i,2}$ on
$\cO \times \cV_i$ with the properties described in Assumption \ref{nonlinbc}. For these symbols, we shall use
the substitution $(v_1,v_2) \rightarrow \eps U(x,\theta_0)$ to prescribe the space dependence.

We let $\chi_i(\zeta),i=1,\dots,N_1+N_2$ be a smooth partition of unity subordinate to $\cC$, and extend the
$\chi_i$ to all $\zeta$ as functions homogeneous of degree zero. We smoothly extend each $Q_{0,i}$ (as a
matrix with bounded inverse) first to $\Sigma$, and then to all $\zeta$ as a function homogenous of degree zero.
We take similar extensions in $\zeta$ of $P_{i,1}$, $P_{i,2}$, $\bD_{1,i}$ and $\sigma_i$, but with homogeneity
of degree $1$ in the cases of $\bD_{1,i}$ and $\sigma_i$. As with $Q_{0,i}$, the extensions of $P_{i,1}$ and
$P_{i,2}$ are taken to have bounded inverses.\footnote{Taking such extensions will reduce the number of
cutoff functions we need later.} Of course, for a given $i\leq N_1$ the property \eqref{ca3a} is satisfied only
for $\zeta/|\zeta|\in\cV_i$.

\textbf{II) First estimate near the bad set.} The first estimate deals with a piece of $\dot U^\gamma$ that is
microlocalized near the bad set $\Upsilon$ and in the frequency domain $|(\xi',\gamma)| \ll |k \, \beta|/\eps$.

\begin{prop}\label{ia3}
Fix $i$ such that $1\leq i\leq N_1$ and let $\dot U^\gamma_1:=\chi^e_D \, \chi_{i,D}\dot U^\gamma$ where
$\chi_i$ is defined in step \textbf{I)} above and $\chi^e_D$ is a cut-off in the extended calculus satisfying
\eqref{n31}. Then we have the a priori estimate
\begin{align}\label{ia4}
|\dot U^\gamma_1|_{0,0} +\dfrac{|\dot U^\gamma_1|_{\infty,0}}{\sqrt{\gamma}} \leq C(K) \left(
\frac{|\Lambda_D f^\gamma|_{0,0}}{\gamma^2} +\frac{\langle\Lambda_D g^\gamma\rangle_{0}}{\gamma^{3/2}}
+\frac{|\dot U^\gamma|_{0,0}}{\gamma^2} +\frac{\langle\dot U^\gamma\rangle_0}{\gamma^{3/2}}\right) \, ,
\end{align}
for $\gamma \ge \gamma_0(K)$.
\end{prop}

\begin{proof}[Proof of Proposition \ref{ia3}]
The loss of derivatives in the estimate prevents us from treating the zero order term $\cD (\eps U)\, \dot{U}^\gamma$
as a forcing term, as we would in a uniformly stable problem. Thus, we need to use an argument that simultaneously
diagonalizes $\bA$ and the lower order term $\cD (\eps U)$.

We now set $\chi^e \chi_i =\bchi$, $v:=\bchi_D\dot U^\gamma$ and look for an estimate $v$. We let
$\bA(X,\gamma) =-\cA (X,\gamma)$ denote the singular symbol such that
\begin{equation*}
\bA_D =\mathbb{A} \left( (\partial_t +\gamma,\partial_{x''}) +\frac{\beta \partial_{\theta_0}}{\eps} \right).
\end{equation*}
Dropping superscripts $\gamma$,  we see from \eqref{i5} that $v$ satisfies
\begin{align}\label{ii2}
\begin{split}
&\partial_d v +\bA_D v +\cD(\eps U)\, v =\bchi_D f +[\cD(\eps U),\bchi_D] \dot U =\bchi_D f+r_{-1,D}\dot U \, ,\\
&\cB(\eps U)\, v|_{x_d=0} =\bchi_D g +[\cB(\eps U),\bchi_D] \dot U |_{x_d=0} =\bchi_D g +r_{-1,D}\dot U |_{x_d=0}.
\end{split}
\end{align}
Here and below $r_{-1,D}$ denotes a singular operator of order $-1$ (which can change from one occurrence
to the next) computed using the singular calculus.  Similarly, $r_{0,D}$ will denote an operator of order $0$. In
spite of the loss of the factor $\Lambda_D$ in the estimate \eqref{aprioriL2}, we are able to treat $r_{-1,D}\dot U$
as a forcing term (see, for example, \eqref{ii12} below). A term like $r_{0,D}\dot U/\gamma$ would  be too large
to absorb.

\textbf{1. Simultaneous diagonalization.} This diagonalization argument is similar to the one in \cite{C1}. Let
$Q_0(\zeta) :=Q_{0,i}(\zeta)$ and $\bD_1(\zeta):=\bD_{1,i}(\zeta)$ be the matrices as in \eqref{conjugator}
such that $Q_0(\zeta) \bA(\zeta) Q_0^{-1}(\zeta)=\bD_1(\zeta)$ in the conical extension of $\cV_i$. We define
\begin{align}\label{ii3}
w := Q_D v \, ,
\end{align}
where $Q=Q_0(X,\gamma)+Q_{-1}(\eps U,X,\gamma)$. Here the matrix $Q_{-1}(\eps U,\zeta)$ is a symbol of
order $-1$ defined for all $\zeta$, but chosen so that on the conical extension of $\cV_i$, the matrix
\begin{align}\label{ii4}
\bD_0(\eps U,\zeta) := [Q_{-1}Q_0^{-1},\bD_1] +Q_0 \, \cD(\eps U) \, Q_0^{-1}
\end{align}
is block diagonal, necessarily of order $0$, with blocks of the same dimensions $n_1,\dots,n_J$ as those of $\bD_1$.
Since the eigenvalues associated to the blocks of $\bD_1$ are mutually distinct, a direct computation shows that
$Q_{-1}Q_0^{-1}$, and thus $Q_{-1}$, can be chosen so that the commutator cancels the off-diagonal blocks of
$Q_0 \, \cD(\eps U) \, Q_0^{-1}$. (The diagonal blocks of the commutator are all zero blocks and can therefore not
cancel those of $Q_0 \, \cD(\eps U) \, Q_0^{-1}$.) Since $Q_0 \, \bA =\bD_1 \, Q_0$ on $\cV_i$, \eqref{ii4} implies
the relation
\begin{align}\label{ii5}
Q \, \bA +Q_0 \, \cD =\bD_1 \, Q +[Q_{-1}Q_0^{-1},\bD_1] \, Q_0 +Q_0 \, \cD =\bD_1 \, Q +\bD_0 \, Q_0.
\end{align}

\begin{rem}\label{iia4}[Entries of the symbol $Q_{-1}$]
\textup{Observe that the scalar entries of the matrix $Q_{-1,D}$ can be chosen to have the form
\begin{align}\label{iia5}
(Q_{-1,D})_{i,j}=c(\eps U) \, a_{-1,D},
\end{align}
where $a_{-1}(\zeta)$ is of order $-1$ and independent of $(x,\theta)$, thus giving rise to a Fourier multiplier.}
\end{rem}

Noting that $x$-dependence is absent in $\bA$ and $Q_0$ and using the commutation property \eqref{ii5}, we have
\begin{align}\label{ii6}
\begin{split}
\partial_d w &= Q_D \, \partial_d v +(\partial_dQ_{-1})_D v
= -Q_D \, (\bA +\cD(\eps U))_D v +Q_D\, \bchi_D f +r_{-1,D} \dot U \\
&= -(Q\, \bA +Q_0 \, \cD(\eps U))_D v +Q_D\, \bchi_D f +r_{-1,D} \dot U \\
&= -(\bD_1Q+\bD_0Q_0)_Dv+Q_D\, \bchi_D f+r_{-1,D}\dot U \\
&= -(\bD_1+\bD_0)_Dw+Q_D\, \bchi_D f+r_{-1,D}\dot U \, .
\end{split}
\end{align}
Here we have used the support property of $\chi^e$  to conclude (see Proposition \ref{n31a}(a) and (b), as well
as Remark \ref{n31c})
\begin{align}\label{ii7}
\begin{split}
&(a) \, (\partial_d Q_{-1})_D v =r_{-1,D} \dot U \quad \text{ and }\\
&(b) \, (\bD_1 \, Q_{-1})_D v =\bD_{1,D} \, (Q_{-1})_D v +r_{-1,D} \dot U \, .
\end{split}
\end{align}
The fact that the statements in \eqref{ii7} are no longer true when $\chi^e$ is replaced by $(1-\chi^e)$ in the
definition of $v$ is the main reason for the difference between Propositions \ref{ia3} and \ref{hard} (the latter
being much more technical). Here we are able to treat all commutators as forcing terms since they are order
$-1$ operators.

\textbf{2. Outgoing modes.} Recall that $-\bD_1$ and $-\bD_0$ are block diagonal:
\begin{equation*}%\label{ii8}
-\bD_1(\zeta) =\begin{pmatrix}
i\omega_1(\zeta) I_{n_1} & & 0\\
 & \ddots & \\
0 & & i\omega_J(\zeta) I_{n_J} \end{pmatrix} \, ,\; \;
-\bD_0(\eps U,\zeta) =\begin{pmatrix}
C_1 & & 0\\
 & \ddots & \\
0 & & C_J \end{pmatrix},
\end{equation*}
so the system \eqref{ii6} satisfied by $w=(w_1,\dots,w_J)$ can be written as a collection of $J$ decoupled
transport equations
\begin{align}\label{ii9}
\partial_d w_j =(i\omega_j)_Dw_j+C_{j,D}w_j+ r_{0,D} f+r_{-1,D}\dot U
\end{align}
with $\re (i\omega_j) \leq -c\, \gamma$ for $1\leq j\leq J'$, and $\re (i\omega_j) \geq c\, \gamma$ for
$J'+1\leq j\leq J$ ($c>0$ denotes a constant).

Following the strategy of \cite{C1}, we now give two preliminary estimates of the outgoing modes $w_j$,
$j\geq J'+1$. Taking the real part of the $L^2(\Omega)$ inner product of \eqref{ii9} with $\Lambda_D^2 w_j$,
we obtain
\begin{align*}%\label{ii10}
-\dfrac{\langle \Lambda_Dw_j(0)\rangle^2_0}{2} &= \re (\Lambda_D \, (i\omega_j)_Dw_j,\Lambda_D w_j)_{L^2(\Omega)}
+\re (\Lambda_D \, C_{j,D}w_j,\Lambda_D w_j)_{L^2(\Omega)}\\
&+\re (\Lambda_D \, Q_D\, \bchi_D f,\Lambda_D w_j)_{L^2(\Omega)}
+\re (\Lambda_D \, r_{-1,D}\dot U,\Lambda_D w_j)_{L^2(\Omega)} \, .
\end{align*}
Since $\re (i\omega_j) \geq c \, \gamma$, we get after absorbing some terms on the left
\begin{align}\label{ii11}
\gamma \, |\Lambda_D w_j|_{0,0}^2 +\langle \Lambda_D w_j(0)\rangle^2_0 \leq
\dfrac{C}{\gamma} \, \left( |\Lambda_D f|_{0,0}^2 +|\dot U|^2_{0,0} \right) \, .
\end{align}
Here, for example, we have used Young's inequality and estimated
\begin{align}\label{ii12}
|\re (\Lambda_D \, r_{-1,D}\dot U ,\Lambda_D w_j)_{L^2(\Omega)}| \leq \dfrac{C_\delta}{\gamma} \,
|\dot U|_{0,0}^2 +\delta \, \gamma \, |\Lambda_D w_j|^2_{0,0} \, ,
\end{align}
with $\delta>0$ small enough so as to absorb the term $|\Lambda_D w_j|_{0,0}^2$ from right to left. Similarly, taking
the real part of the $L^2$ inner product of \eqref{ii9} with $w_j$ on $[x_d,\infty)\times b\Omega$ instead of $\Omega$,
we obtain for all $x_d\geq 0$:
\begin{align}\label{ii13}
\gamma \, | w_j|_{0,0}^2 +\langle  w_j(x_d)\rangle^2_0 \leq \dfrac{C}{\gamma} \, \left(
|f|_{0,0}^2 +\dfrac{1}{\gamma^2} \, |\dot U|^2_{0,0} \right) \, .
\end{align}

\textbf{3. Incoming modes.} Estimating $w_j$ for $j\leq J'$ in a similar way, but now using $\re (i\omega_j) \leq
-c\, \gamma$ and pairing the corresponding transport equation with $w_j$, we obtain
\begin{equation*}%\label{ii14}
\gamma \, |w_j|_{0,0}^2 \leq C\, \langle w_j(0) \rangle^2_0 +\dfrac{C}{\gamma} \, \left( |f|_{0,0}^2
+|r_{-1,D}\dot U|^2_{0,0}\right) \leq C \, \langle w_j(0) \rangle^2_0
+\dfrac{C}{\gamma} \, \left( |f|_{0,0}^2 +\dfrac{1}{\gamma^2} \, |\dot U|^2_{0,0} \right) \, .
\end{equation*}
Taking the real part of the $L^2$ inner product of \eqref{ii9} with $w_j$ on $[0,x_d]\times b\Omega$ rather than
on all $\Omega$, we also obtain
\begin{align}\label{ii15}
\gamma \, |w_j|_{0,0}^2 +\langle w_j(x_d) \rangle^2_0 \leq C \, \langle w_j(0) \rangle^2_0
+\dfrac{C}{\gamma} \, \left( |f|_{0,0}^2 +\dfrac{1}{\gamma^2} \, |\dot U|^2_{0,0} \right) \, .
\end{align}

\textbf{4. Boundary estimate.} We observe that $v$ can be expressed in terms of $w$ as
\begin{align}\label{ii17}
v =(Q_0^{-1})_D w +r_{-1,D}\dot U \, .
\end{align}
Recalling the boundary condition in \eqref{ii2} and using the decomposition $Q_0^{-1}(\zeta) =[Q_{in}(\zeta)
\; Q_{out}(\zeta)]$, we let accordingly $w=(w_{in},w_{out})$ and rewrite the boundary condition in \eqref{ii2} as
\begin{align}\label{ii18}
\cB(\eps U) \, Q_{in,D} w_{in}|_{x_d=0} =-\cB(\eps U) \, Q_{out,D} w_{out}|_{x_d=0} +\bchi_D g +r_{-1,D}
\dot U|_{x_d=0} \, .
\end{align}
By \eqref{ca3a} we have on $\cV_i$
\begin{equation*}%\label{ii19}
\cB(\eps U) \, Q_{in} =P_1^{-1} \, \left( P_1 \, \cB(\eps U) \, Q_{in} \, P_2 \right) \, P_2^{-1}
=P^{-1}_1 \, \begin{pmatrix}
\Lambda^{-1} \, (\gamma+i\sigma) & \\
 & I \end{pmatrix} \, P^{-1}_2,
\end{equation*}
so using the rules of singular calculus, we get
\begin{equation*}%\label{ii20}
\Lambda_D \, \cB(\eps U) \, Q_{in,D} w_{in}|_{x_d=0} =(P^{-1}_1)_D \, \begin{pmatrix}
\gamma +i\sigma_D & \\
 & \Lambda_D \, I \end{pmatrix} \, (P^{-1}_2)_D w_{in}|_{x_d=0} +r_{0,D} w_{in}|_{x_d=0} .
\end{equation*}
With \eqref{ii18}, this implies
\begin{align}\label{ii21}
\langle (P^{-1}_1)_D \begin{pmatrix}
\gamma +i\sigma_D & \\
 & \Lambda_D \, I \end{pmatrix} (P^{-1}_2)_D w_{in}|_{x_d=0} \rangle_0 \leq C \left(
\langle \Lambda_D w_{out}|_{x_d=0} \rangle_0 +\langle \Lambda_D g \rangle_0 +\langle \dot U|_{x_d=0} \rangle_0
\right).
\end{align}
We have $P_{1,D}(P^{-1}_1)_D=I+r_{-1,D}$ so up to choosing $\gamma$ large (and absorbing the $r_{-1,D}$ term),
the estimate \eqref{ii21} implies
\begin{align}\label{ii22}
\langle \begin{pmatrix}
\gamma +i\sigma_D & \\
 & \Lambda_D \, I \end{pmatrix} (P^{-1}_2)_D w_{in}|_{x_d=0} \rangle_0 \leq C \left(
\langle \Lambda_D w_{out}|_{x_d=0} \rangle_0 +\langle \Lambda_D g \rangle_0
+\langle \dot U|_{x_d=0} \rangle_0 \right).
\end{align}
Letting $\begin{pmatrix}w_1\\w'\end{pmatrix}:=(P^{-1}_2)_Dw_{in}|_{x_d=0}$ we find, using the fact that $\sigma$ is
real and choosing again $\gamma$ large enough,
\begin{equation*}%\label{ii23}
\langle \begin{pmatrix}
\gamma +i\sigma_D & \\
 & \Lambda_D \, I \end{pmatrix} \begin{pmatrix}
w_1\\
w' \end{pmatrix} \rangle_0^2 \geq \dfrac{1}{C} \, \big( \gamma^2 \, \langle w_1\rangle^2_0
+\langle \Lambda_D w' \rangle_0^2 \big) \geq \dfrac{\gamma^2}{C} \, \langle w_1,w'\rangle^2_0.
\end{equation*}
Thus, from \eqref{ii22} we may conclude
\begin{align}\label{ii24}
\gamma \, \langle w_{in}|_{x_d=0} \rangle_0 \leq C \left( \langle \Lambda_D w_{out}|_{x_d=0} \rangle_0
+\langle \Lambda_D g \rangle_0 +\langle \dot U|_{x_d=0} \rangle_0 \right).
\end{align}

\textbf{5. Conclusion.} Combining the estimates \eqref{ii11}, \eqref{ii13}, \eqref{ii15}, and \eqref{ii24} we obtain
\begin{multline*}%\label{ii25}
\gamma^3 \, |w_{in}|^2_{0,0} +\gamma^2 \, |w_{in}|^2_{\infty,0} +\gamma \, |\Lambda_D w_{out}|^2_{0,0}
+\langle \Lambda_D w_{out}|_{x_d=0} \rangle^2_0 +\gamma^2 \, |w_{out}|^2_{\infty,0} \\
\le C \left( \dfrac{|\Lambda_D f|_{0,0}^2}{\gamma} +\langle \Lambda_D g\rangle_0^2
+\dfrac{|\dot U|^2_{0,0}}{\gamma} +\langle \dot U|_{x_d=0} \rangle_0^2 \right).
\end{multline*}
Recalling the relation \eqref{ii17} between $v$ and $w =Q_Dv$, we obtain the estimate \eqref{ia4} for $\gamma$
large enough. Though we have not emphasized the regularity of the symbols, we claim that the regularity assumption
on $U$ is sufficient to apply the rules of singular symbolic calculus at each step of the above calculations, see
Appendix \ref{calc}. We feel free to skip some of the details.
\end{proof}

\textbf{III) Second estimate near the bad set.} Next we estimate the second term on the right in \eqref{ia2}. In
the statement below $Q_D=Q_{0,D}+Q_{-1,D}$ is the same operator as that constructed above in the proof
of Proposition \ref{ia3}. Since $(Q_{0,D})^{-1}Q_{-1,D}$ has norm less than one as an operator on $L^2$ for
$\gamma$ large, we can define $(Q_D)^{-1}$ as an operator on $L^2$ using a Neumann series.

\begin{prop}\label{hard}
Fix $i$ such that $1\leq i\leq N_1$, let $\dot U^\gamma_2 := (1-\chi^e_D) \chi_{i,D} \dot U^\gamma$ and write
\begin{align}\label{ia8}
\dot U^\gamma_2 =\dot U^\gamma_{2,in} +\dot U^\gamma_{2,out},
\end{align}
where\footnote{Here $w_{in} \in \C^p$ and $w_{out} \in \C^{N-p}$ are defined by the same diagonalization procedure
as in the previous proof.}
\begin{align}\label{ia8a}
\dot U^\gamma_{2,in} :=(Q_D)^{-1}(w_{in},0) \quad \text{ and } \quad \dot U^\gamma_{2,out} := (Q_D)^{-1}(0,w_{out}).
\end{align}
Then we have
\begin{multline}\label{ia9}
|\dot U^\gamma_{2,in}|_{0,0} +\dfrac{|\dot U^\gamma_{2,in}|_{\infty,0}}{\sqrt{\gamma}}
+\dfrac{|(\nabla_{x'},\gamma) \, \dot U^\gamma_{2,out}|_{0,0}}{\gamma}
+\dfrac{| (\nabla_{x'},\gamma) \, \dot U^\gamma_{2,out}|_{\infty,0}}{\gamma^{3/2}} \\
+|\nabla_{x'} \, \Lambda_D^{-1} \dot U^\gamma_{2,in}|_{0,0}
+\dfrac{\langle \nabla_{x'} \, \Lambda_D^{-1} \dot U^\gamma_{2,in}|_{x_d=0} \rangle_{0}}{\sqrt{\gamma}}
\le C \left( \dfrac{|(\nabla_{x'},\gamma) \, f^\gamma|_{0,0}}{\gamma^2}
+\dfrac{\langle (\nabla_{x'},\gamma) \, g^\gamma\rangle_{0}}{\gamma^{3/2}} \right. \\
\left. +\dfrac{|\dot U^\gamma|_{0,0}+|\nabla_{x'} \, \Lambda_D^{-1}\dot U^\gamma|_{0,0}}{\gamma^2}
+\dfrac{\langle \dot U^\gamma|_{x_d=0} \rangle_0
+\langle \nabla_{x'} \, \Lambda_D^{-1}\dot U^\gamma|_{x_d=0} \rangle_0}{\gamma^{3/2}} \right) \, .
\end{multline}
\end{prop}

\begin{proof}[Proof of Proposition \ref{hard}]
\textbf{1. Simultaneous diagonalization with a new remainder.} We now set $(1-\chi^e) \chi_i =\bchi$, $v :=
\bchi_D\dot U^\gamma =\dot U^\gamma_2$, $w :=Q_D v =(w_{in},w_{out})$ and repeat the proof of Proposition
\ref{ia3} down to line \eqref{ii6} with only one change. Suppressing superscripts $\gamma$, we now have
instead of \eqref{ii6}
\begin{align}\label{ib1}
\partial_d w =-(\bD_1+\bD_0)_D w +r_{0,D} f +r_{-1,D} \dot U +R^a_D v,
\end{align}
where $R^a_D =R^b_D-R^c_D$ with operators $R^b_D,R^c_D$ defined in \eqref{ib2} below. As can be checked
by looking at each line in \eqref{ii6}, we also observe that the remainder term $r_{-1,D} \dot U$ reads
\begin{equation}\label{defreste}
Q_D \, [\cD (\eps U),\bchi_D] \dot U -(Q_{0,D} \, \bD_{0,D} -(Q_0 \, \bD_0)_D) \, \bchi_D \dot U
-Q_{-1,D} \, \cD (\eps U) \, \bchi_D \dot U +\bD_{0,D} \, Q_{-1,D} \, \bchi_D \dot U \, .
\end{equation}
The equalities of \eqref{ii7} should be replaced by
\begin{align}\label{ib2}
\begin{split}
&(a) \, (\partial_d Q_{-1})_D v = R^b_D v \, ,\quad \text{ and }\\
&(b) \, (\bD_1 \, Q_{-1})_Dv=\bD_{1,D} \, (Q_{-1})_Dv +R^c_Dv \, ,
\end{split}
\end{align}
where in view of Remark \ref{iia4} the scalar entries of $R^b_D$ and $R^c_D$ have the forms
\begin{align}\label{ib3}
(\partial_d c(\eps U)) \, a_{-1,D} \quad \text{ and } \quad [\alpha_{1,D},c(\eps U)] \, a_{-1,D} \, ,
\end{align}
respectively. In \eqref{ib3}, $\alpha_1(\zeta)$ denotes one of the diagonal entries of $\bD_1(\zeta)$. Here and
below $a_{-1,D}$ denotes a singular operator of order $-1$ associated to a symbol $a_{-1}(\zeta)$ which may
change from term to term. We claim
\begin{equation}\label{ib3a}
|R^a_D v|_{0,0} \leq \dfrac{C}{\gamma} \, |v|_{0,0} \quad \text{ and } \quad
|\nabla_{x'}R^a_D v|_{0,0} \leq C \, |\nabla_{x'}\Lambda_D^{-1}v|_{0,0} +\dfrac{C}{\gamma} \, |v|_{0,0}.
\end{equation}
Indeed, the first estimate is clear from the definition \eqref{ib3} since $\partial_d c(\eps U)$ is a bounded function
and $[\alpha_{1,D},c(\eps U)]$ is bounded on $L^2$ by the rules of singular calculus. The $L^2$ estimate of
$\nabla_{x'} R^b_D v$ is readily obtained by differentiating the product $(\partial_d c(\eps U)) \, a_{-1,D}v$
with respect to $x'$ so we now focus on the estimate of $\nabla_{x'} R^c_D v$. Setting $r_{0,D} :=[\alpha_{1,D},
c(\eps U)]$, which is a bounded operator on $L^2$ by the rules of singular calculus, we have
\begin{equation*}%\label{ib4}
|\nabla_{x'} \, R^c_Dv|_{0,0} \leq |r_{0,D} \, \nabla_{x'} \, a_{-1,D}v|_{0,0}
+|[\nabla_{x'},r_{0,D}] \, a_{-1,D}v|_{0,0} \leq C \, \left( |\nabla_{x'} \Lambda_D^{-1}v|_{0,0}
+\dfrac{|v|_{0,0}}{\gamma} \right),
\end{equation*}
where we have used Proposition \ref{n31a}(c) to estimate the double commutator term $[\nabla_{x'},r_{0,D}]$.

The estimates \eqref{ib3a} will be used below to handle the source term $R^a_D v$ appearing on the right of
\eqref{ib1}. Unlike what happened in the proof of Proposition \ref{ia3}, \eqref{ib3a} seems to be the best we
can hope for the bad $(1,-1)$ product.

\textbf{2. Incoming modes I.} In place of \eqref{ii9} we now have
\begin{align}\label{ib5}
\partial_d w_j =(i\omega_j)_Dw_j+C_{j,D}w_j+ r_{0,D} f+r_{-1,D}\dot U+R^a_D v, \;j=1,\dots,J
\end{align}
and for the incoming part $w_{in}$ we obtain as in \eqref{ii15}:
\begin{align}\label{ib6}
\gamma^3 \, | w_{in}|_{0,0}^2 +\gamma^2 \, |w_{in}|^2_{\infty,0}
\leq \gamma^2 \, \langle w_{in}|_{x_d=0} \rangle^2_0
+\dfrac{C}{\gamma} \, \left( |\gamma \, f|_{0,0}^2 +|\dot U|^2_{0,0} \right).
\end{align}
This $L^2$ estimate does not cause any problem because we have the good $L^2$ control \eqref{ib3a} of the
additional term $R^a_D v$ appearing on the right of \eqref{ib5}.

We can also repeat the same argument as in step \textbf{4} of the previous proof for the boundary terms because
these estimates did not rely on the support properties of $\chi^e$ (the support properties of $\chi^e$ were only
used to control some error terms in the interior equation). We thus derive again
\begin{equation*}%\label{ib7}
\gamma \, \langle w_{in}|_{x_d=0} \rangle_0 \leq C \, \left( \langle \Lambda_D w_{out}|_{x_d=0} \rangle_0
+\langle \Lambda_D \, \bchi_D g \rangle_0 +\langle \dot U|_{x_d=0} \rangle_0 \right) \, .
\end{equation*}
Using the fact that on the support of $1-\chi^e(\xi',\frac{k\beta}{\eps},\gamma)$, we have
\begin{equation*}%\label{ib10}
|X,\gamma| \leq C \, |\xi',\gamma| \, ,
\end{equation*}
we obtain the estimate
\begin{equation*}%\label{ib11}
\langle \Lambda_D w_{out}|_{x_d=0} \rangle_0 +\langle \Lambda_D \, \bchi_D g \rangle_0 \leq C \,
\langle (\nabla_{x'},\gamma) \, w_{out}|_{x_d=0} \rangle_0
+C \, \langle (\nabla_{x'},\gamma) \, g \rangle_0 \, .
\end{equation*}
%\begin{align}\label{ib12}
%\Lambda_D \, Q_{-1,D}=(\Lambda Q_{-1})_D +R^d_D,
%\end{align}
%where $R^d_D$ is an operator (similar to the remainders in \eqref{ib2}) whose scalar entries have the
%form $[\Lambda_D,c(\eps U)]a_{-1,D}$.
At this stage, we have thus derived the bound
\begin{multline}
\label{ib6prime}
\gamma^3 \, | w_{in}|_{0,0}^2 +\gamma^2 \, |w_{in}|^2_{\infty,0}
\leq \dfrac{C}{\gamma} \, \left( |\gamma \, f|_{0,0}^2 +|\dot U|^2_{0,0} \right)
+C \, \left( \langle (\nabla_{x'},\gamma) \, g \rangle_0^2 +\langle \dot U|_{x_d=0} \rangle_0^2 \right) \\
+C \, \langle (\nabla_{x'},\gamma) \, w_{out}|_{x_d=0} \rangle_0^2 \, .
\end{multline}

\textbf{3. Outgoing modes.} Let $j \in \{J'+1,\dots,J\}$. Taking the real part of the $L^2(\Omega)$ inner product
of \eqref{ib5} with $-\partial_{x_k}^2 w_j$ for $k=0,\dots,d-1$, and with $\gamma^2w_j$, we obtain in place of
\eqref{ii13}
\begin{align}\label{ib8}
\gamma \, |(\nabla_{x'},\gamma) \, w_{out}|_{0,0}^2 +|(\nabla_{x'},\gamma) \, w_{out}|^2_{\infty,0}
\leq \dfrac{C}{\gamma} \, \left(|(\nabla_{x'},\gamma) \, f|_{0,0}^2 +|\dot U|^2_{0,0}
+|\partial_{x_k} \, \Lambda_D^{-1}\dot U|_{0,0}^2 \right) \, .
\end{align}
Here in place of \eqref{ii12} we have used the estimate
\begin{equation*}%\label{ib9}
|(\partial_{x_k} \, r_{-1,D} \dot U,\partial_{x_k} w_j)_{L^2(\Omega)}| \leq \dfrac{C_\delta}{\gamma} \,
|\partial_{x_k} \, \Lambda_D^{-1} \dot U|_{0,0}^2 +\dfrac{C_\delta}{\gamma} \, |\dot U|^2_{0,0}
+\delta \, \gamma \, |\partial_{x_k}w_j|^2_{0,0},
\end{equation*}
along with a similar estimate for $|(\partial_{x_k} R^a_Dv,\partial_{x_k}w_j)_{L^2(\Omega)}|$. Such estimates
follow either from the precise expression of the remainder $\partial_{x_k} \, r_{-1,D} \dot U$ (differentiating each
of the four terms in \eqref{defreste} with respect to $x_k$) or from \eqref{ib3a}. As usual, we then choose $\delta$
small enough so as to absorb terms from right to left.

Adding \eqref{ib6prime} for the incoming modes and \eqref{ib8} for the outgoing modes, we have already shown
how to control the first line in \eqref{ia9} by the quantity on the right of the inequality \eqref{ia9}. At this stage, it
thus only remains to control the term $\nabla_{x'} \, \Lambda_D^{-1} \, \dot U^\gamma_{2,in}$.

\begin{rem}\label{ib13}
\textup{At this point we can see the need to estimate the remaining terms on the left in the estimate \eqref{ia9}
as well as the similar terms on the left in the Kreiss estimate \eqref{ia11} below. We must estimate  those terms
in order to be able to absorb the terms involving $\nabla_{x'} \, \Lambda_D^{-1} \dot U^\gamma$ on the right
side of \eqref{ia9}. We recall that these terms come from the "bad" composition term $\bD_{1,D} \, Q_{-1,D}$
and from $\partial_d Q_{-1,D}$. Observe that there is no need for a separate estimate of $\nabla_{x'} \,
\Lambda_D^{-1} \dot U^\gamma_1$ since, for example,
\begin{equation*}%\label{ib14}
|\nabla_{x'} \, \Lambda_D^{-1} \dot U^\gamma_1|_{0,0} \leq C \, |\dot U^\gamma_1|_{0,0}
\end{equation*}
because $|X,\gamma| \ge C \, |\xi',\gamma|$ on the support of  $\chi^e$.}
\end{rem}

\textbf{4. Incoming modes II.} Here we begin to estimate the terms in the second line of \eqref{ia9}. The problem
satisfied by $v=\dot U^\gamma_2$ has the form \eqref{ii2} with the new cut-off function $\bchi =(1-\chi^e) \,
\chi_i$. We now introduce the functions $\tv :=\Lambda_D^{-1}v$ and $\tv' :=\partial_{x'} \tv$, where $\partial_{x'}$
denotes any of the tangential derivatives $\partial_{x_0},\dots,\partial_{x_{d-1}}$, and see that the function
$\tv'$ satisfies
\begin{align}\label{ic1}
\begin{split}
&\partial_d \tv' +\bA_D \tv' +\cD(\eps U) \tv' =\partial_{x'} \, \Lambda_D^{-1} \, \bchi_D f +\partial_{x'} \, \Lambda_D^{-1}
\, [\cD(\eps U),\bchi_D] \dot U \\
&\qquad \qquad \qquad \qquad \qquad +\partial_{x'} \big( (\cD(\eps U) -\Lambda_D^{-1} \, \cD(\eps U) \, \Lambda_D) \tv
\big) -{\rm d}\cD (\eps U) \cdot \eps \partial_{x'}U \, \tv \, ,\\
&\cB(\eps U) \, \tv'|_{x_d=0} =\partial_{x'} \, \Lambda_D^{-1} \, \bchi_D g +\partial_{x'} \, \Lambda_D^{-1} \,
[\cB(\eps U),\bchi_D]\dot U|_{x_d=0} +\partial_{x'} \, [\cB(\eps U),\Lambda_D^{-1}] v|_{x_d=0}.
\end{split}
\end{align}
We can thus diagonalize the problem for $\tv'$ with the \emph{same} operator $Q_D$ as before. Introducing the
function $\tw' :=Q_D \tv'$, we find that $\tw'$ satisfies
\begin{align}\label{ic3}
\begin{split}
&(a) \, \partial_d \tw' =-(\bD_1+\bD_0)_D\tw' +Q_D \, \partial_{x'} \, \Lambda_D^{-1} \, \bchi_D f
+\dfrac{1}{\gamma} \, r_{0,D} \tv' +\dfrac{1}{\gamma} \, r_{0,D} \dot U
+Q_D \, \partial_{x'} \, \Lambda_D^{-1} \, [\cD(\eps U),\bchi_D] \dot U \, ,\\
&(b) \, \cB(\eps U) \, Q_{in,D} \tw_{in}' =-\cB(\eps U) \, Q_{out,D} \tw_{out}' +\partial_{x'} \, \Lambda_D^{-1} \, \bchi_D g
+\partial_{x'} \, \Lambda_D^{-1} \, [\cB(\eps U),\bchi_D] \dot U \\
&\qquad \qquad \qquad \qquad \qquad +\partial_{x'} \, \Lambda_D^{-1} \, [\Lambda_D,\cB(\eps U)] \tv +r_{-1,D} \tv' \, .\\
%&(c) \, \cB(\eps U) \, Q_{in,D} \partial_k\tw_{in} =-\cB(\eps U) \, Q_{out,D} \, \partial_k \tw_{out} +\partial_k \,
%\Lambda_D^{-1} \, \bchi_D g +\partial_k \, \Lambda_D^{-1} \, r_{-1,D} \dot U\\
%&\qquad \qquad +\partial_k \, [\cB(\eps U),\Lambda_D^{-1}] v +[\cB(\eps U) \, Q_{in,D},\partial_k] \tw_{in}
%+[\cB(\eps U) \, Q_{out,D},\partial_k]\tw_{out} \, ,
\end{split}
\end{align}
where we have collected many terms (for instance the $R^a_D$ operator as in \eqref{ib1}) into remainders of
the form $\gamma^{-1} \, r_{0,D}$. For instance, we have used
\begin{multline*}
Q_D \, \partial_{x'} \big( (\cD(\eps U) -\Lambda_D^{-1} \, \cD(\eps U) \, \Lambda_D) \tv \big) \\
= Q_D \, (\cD(\eps U) -\Lambda_D^{-1} \, \cD(\eps U) \, \Lambda_D) \tv'
+Q_D \, (\partial_{x'} \cD(\eps U) -\Lambda_D^{-1} \, \partial_{x'} \cD(\eps U) \, \Lambda_D) \tv \\
=\dfrac{1}{\gamma} \, r_{0,D} \tv' +\dfrac{1}{\gamma} \, r_{0,D} \tv =\dfrac{1}{\gamma} \, r_{0,D} \tv'
+\dfrac{1}{\gamma} \, r_{0,D} \dot U \, .
\end{multline*}

Next we fix an index $j\in\{1,\dots,J'\}$. Taking the real part of the $L^2(\Omega)$ inner product of \eqref{ic3}$(a)$
with $\tw_j'$, we obtain the standard $L^2$ estimate for incoming modes:
\begin{multline*}%\label{ic4}
\gamma \, |\tw_{in}'|_{0,0}^2 \leq C \, \langle \tw_{in}'|_{x_d=0} \rangle^2_0
+\dfrac{C}{\gamma} \, \left( |\partial_{x'} \, \Lambda_D^{-1}f|_{0,0}^2
+\dfrac{1}{\gamma^2} \, |\partial_{x'} \, \Lambda_D^{-1} \dot U|_{0,0}^2 +\dfrac{1}{\gamma^2} \, |\dot U|^2_{0,0} \right) \\
+\dfrac{C}{\gamma} \, |\partial_{x'} \, \Lambda_D^{-1} \, [\cD(\eps U),\bchi_D] \dot U|_{0,0}^2 \, .
\end{multline*}
The last term on the right of the inequality is estimated by similar techniques as above, namely we write
\begin{equation*}
\partial_{x'} \, \Lambda_D^{-1} \, [\cD(\eps U),\bchi_D] =[\cD(\eps U),\bchi_D] \, \partial_{x'} \, \Lambda_D^{-1}
+[\partial_{x'} \, \Lambda_D^{-1}, \cD(\eps U)] \, \bchi_D -\bchi_D \, [\partial_{x'} \, \Lambda_D^{-1}, \cD(\eps U)] \, ,
\end{equation*}
and then decompose the commutator $[\partial_{x'} \, \Lambda_D^{-1}, \cD(\eps U)]$ as follows:
\begin{equation*}
[\partial_{x'} \, \Lambda_D^{-1}, \cD(\eps U)] =\partial_{x'} \, \big( \Lambda_D^{-1} \, \cD(\eps U) \, \Lambda_D
-\cD(\eps U) \big) \, \Lambda_D^{-1} +(\partial_{x'} \cD(\eps U)) \, \Lambda_D^{-1} \, .
\end{equation*}
In the end we obtain the estimate
\begin{equation}\label{ic4prime}
\gamma \, |\tw_{in}'|_{0,0}^2 \leq C \, \langle \tw_{in}'|_{x_d=0} \rangle^2_0 +\dfrac{C}{\gamma^3} \, \left(
|\partial_{x'} f|_{0,0}^2 +|\partial_{x'} \, \Lambda_D^{-1} \dot U|^2_{0,0} +|\dot U|_{0,0}^2 \right) \, ,
\end{equation}
and we thus wish to control the trace of $\tw_{in}'$.

%We now recall that the operator $r_{-1,D} \dot U$ on the right of \eqref{ic4prime} equals $Q_D[\cD(\eps U),\chi_D]$.
%We can thus derive the upper bounds
%\begin{align}
%|\gamma \, \Lambda_D^{-1} \, r_{-1,D}\dot U|_{0,0}^2 &\le |r_{-1,D}\dot U|_{0,0}^2 \le \dfrac{C}{\gamma^2} \,
%|U|_{0,0}^2 \, ,\notag \\
%|\nabla_{x'} \, \Lambda_D^{-1} \, r_{-1,D}\dot U|_{0,0}^2 &\le \dfrac{1}{\gamma^2} \,
%|\nabla_{x'} \, r_{-1,D}\dot U|_{0,0}^2 \le  \dfrac{C}{\gamma^2} \, \left(
%|r_{-1,D} \, \Lambda_D \, \nabla_{x'} \, \Lambda_D^{-1} \dot U|_{0,0}^2
%+|[r_{-1,D},\nabla_{x'}] \dot U|_{0,0}^2\right) \notag \\
%&\le \frac{C}{\gamma^2} \, \left( |\nabla_{x'} \, \Lambda_D^{-1} \dot U|^2_{0,0} +|\dot U|^2_{0,0} \right) \, .\label{ic13}
%\end{align}
%We have thus obtained so far
%\begin{multline}\label{ic4seconde}
%\gamma \, |(\nabla_{x'},\gamma) \, \tw_{in}|_{0,0}^2 +|(\nabla_{x'},\gamma) \, \tw_{in}|^2_{\infty,0}
%\leq C \, \langle (\nabla_{x'},\gamma) \tw_{in}|_{x_d=0} \rangle^2_0 \\
%+\dfrac{C}{\gamma^3} \, \left( |(\nabla_{x'},\gamma) \, f|_{0,0}^2
%+|\nabla_{x'} \, \Lambda_D^{-1} \dot U|^2_{0,0} +|\dot U|_{0,0}^2 \right) \, ,
%\end{multline}
%where, apart from the boundary integral, the right hand side is similar to that of \eqref{ia9}.

\textbf{5. Control of the trace of $\tw_{in}'$.} Using \eqref{ic3}$(b)$, and arguing as in step \textbf{4} of the
proof of Proposition \ref{ia3}, we obtain the boundary estimate
\begin{multline}\label{ic5}
\gamma \, \langle \tw_{in}'|_{x_d=0} \rangle_0 \leq C \left( \langle \Lambda_D\tw_{out}'|_{x_d=0} \rangle_0
+\langle \partial_{x'} g \rangle_0 +\langle \partial_{x'} \, [\cB(\eps U),\bchi_D] \dot U|_{x_d=0} \rangle_0 \right) \\
+C \, \left( \langle \partial_{x'} \, [\Lambda_D,\cB(\eps U)] \tv|_{x_d=0} \rangle_0 +\langle \tv'|_{x_d=0} \rangle_0 \right) \, .
\end{multline}
The two commutators appearing on the right of \eqref{ic5} are dealt with as in the previous step and we get
\begin{equation*}%\label{ic5prime}
\gamma \, \langle \tw_{in}'|_{x_d=0} \rangle_0 \leq C \left( \langle \Lambda_D\tw_{out}'|_{x_d=0} \rangle_0
+\langle \partial_{x'} g \rangle_0 +\langle \dot U|_{x_d=0} \rangle_0
+\langle \partial_{x'} \, \Lambda_D^{-1} \dot U|_{x_d=0} \rangle_0 \right) \, .
\end{equation*}
Combining with \eqref{ic4prime}, we have derived
\begin{multline}\label{ic4second}
\gamma \, |\tw_{in}'|_{0,0}^2 +\langle \tw_{in}'|_{x_d=0} \rangle^2_0 \le \dfrac{C}{\gamma^2} \,
\langle \Lambda_D\tw_{out}'|_{x_d=0} \rangle_0^2 +\dfrac{C}{\gamma^3} \, \left(
|\partial_{x'} f|_{0,0}^2 +|\partial_{x'} \, \Lambda_D^{-1} \dot U|^2_{0,0} +|\dot U|_{0,0}^2 \right) \\
+\dfrac{C}{\gamma^2} \, \left( \langle \partial_{x'} g \rangle_0^2 +\langle \dot U|_{x_d=0} \rangle_0^2
+\langle \partial_{x'} \, \Lambda_D^{-1} \dot U|_{x_d=0} \rangle_0^2 \right) \, .
\end{multline}
We expect $\Lambda_D\tw_{out}'$ to be comparable to $\partial_{x'} w_{out}$ and thus use \eqref{ib8}; this is
checked and made precise in the next step.

\textbf{6. Relation between $\Lambda_D \tw'$ and $\partial_{x'}w$, and conclusion.} Using the definitions
\begin{equation*}
\tw' =Q_D \, \tv' =Q_D \, \partial_{x'} \, \Lambda_D^{-1}v \quad \text{and} \quad w =Q_D v \, ,
\end{equation*}
and the fact that $\Lambda_D$ commutes with $Q_{0,D}$, we compute
\begin{equation*}
\Lambda_D \tw' =Q_D \, \partial_{x'}v +r_{0,D} \tv' =\partial_{x'} w +r_{0,D} v +r_{0,D} \tv'\, .
\end{equation*}
We have thus derived the bound from above
\begin{equation*}
\dfrac{1}{\gamma^2} \, \langle \Lambda_D\tw_{out}'|_{x_d=0} \rangle_0^2 \le \dfrac{C}{\gamma^2} \, \left(
\langle \partial_{x'} w_{out}|_{x_d=0} \rangle_0^2 +\langle \dot U|_{x_d=0} \rangle_0^2
+\langle \partial_{x'} \, \Lambda_D^{-1} \dot U|_{x_d=0} \rangle_0^2 \right) \, ,
\end{equation*}
which we combine with \eqref{ic4second} and \eqref{ib8} to obtain
\begin{multline}\label{ic40}
\gamma \, |\tw_{in}'|_{0,0}^2 +\langle \tw_{in}'|_{x_d=0} \rangle^2_0 \le \dfrac{C}{\gamma^3} \, \left(
|(\partial_{x'},\gamma) \, f|_{0,0}^2 +|\partial_{x'} \, \Lambda_D^{-1} \dot U|^2_{0,0} +|\dot U|_{0,0}^2 \right) \\
+\dfrac{C}{\gamma^2} \, \left( \langle \partial_{x'} g \rangle_0^2 +\langle \dot U|_{x_d=0} \rangle_0^2
+\langle \partial_{x'} \, \Lambda_D^{-1} \dot U|_{x_d=0} \rangle_0^2 \right) \, .
\end{multline}

It only remains to derive a bound from below to go from $\tw_{in}'$ to $\partial_{x'} \Lambda_D^{-1}
\dot U^\gamma_{2,in}$. We first observe that estimating $\partial_{x'} \Lambda_D^{-1} \dot U^\gamma_{2,in}$
as claimed in \eqref{ia9} amounts to estimating $Q_D \, \partial_{x'} \Lambda_D^{-1} \dot U^\gamma_{2,in}$.
We use the relation
\begin{equation*}
Q_D \, \partial_{x'} \Lambda_D^{-1} \dot U^\gamma_{2,in} =\partial_{x'} \Lambda_D^{-1} w_{in}
-[\partial_{x'} \Lambda_D^{-1},Q_D] \, \dot U^\gamma_{2,in} =\partial_{x'} \Lambda_D^{-1} w_{in}
-[\partial_{x'} \Lambda_D^{-1},Q_{-1,D}] \, \dot U^\gamma_{2,in}\, ,
\end{equation*}
and we use the special "decoupled" form of the coefficients of $Q_{-1}$ to show that the commutator
$[\partial_{x'} \Lambda_D^{-1},Q_{-1,D}]$ reads
\begin{equation*}
[\partial_{x'} \Lambda_D^{-1},Q_{-1,D}] =\dfrac{1}{\gamma^2} \, r_{0,D} +\dfrac{1}{\gamma^2} \, r_{0,D} \,
(\partial_{x'} \, \Lambda_D^{-1}) \, .
\end{equation*}
Similary, we can write
\begin{equation*}
\tw_{in}' =\partial_{x'} \Lambda_D^{-1} w_{in} +\dfrac{1}{\gamma^2} \, r_{0,D} v
+\dfrac{1}{\gamma^2} \, r_{0,D} \, (\partial_{x'} \, \Lambda_D^{-1}) v \, ,
\end{equation*}
so we obtain
\begin{equation*}
Q_D \, \partial_{x'} \Lambda_D^{-1} \dot U^\gamma_{2,in} =\tw_{in}' +\dfrac{1}{\gamma^2} \, r_{0,D} \dot U
+\dfrac{1}{\gamma^2} \, r_{0,D} \, (\partial_{x'} \, \Lambda_D^{-1}) \dot U \, .
\end{equation*}
We have therefore proved that \eqref{ic40} implies that the second line in \eqref{ia9} is controlled
by the terms on the right of \eqref{ia9}. This completes the proof of Proposition \ref{hard}.
\end{proof}

\textbf{IV) Estimate away from the bad set.} The next proposition provides a Kreiss-type estimate for the terms
$\chi_{i,D}\dot U^\gamma$, where $i>N_1$.

\begin{prop}\label{kreisspiece}
Fix $i$ such that $N_1+1\leq i\leq N_2$ and let
$\dot U^\gamma_3:=\chi_{i,D}\dot U^\gamma$.
We have
\begin{multline}\label{ia11}
|\dot U^\gamma_{3}|_{0,0} +\dfrac{\langle \dot U^\gamma_{3}|_{x_d=0} \rangle_{0}}{\sqrt{\gamma}}
+|\nabla_{x'} \, \Lambda_D^{-1}\dot U^\gamma_{3}|_{0,0}
+\dfrac{\langle \nabla_{x'} \, \Lambda_D^{-1}\dot U^\gamma_{3}|_{x_d=0} \rangle_{0}}{\sqrt{\gamma}} \\
\le C\, \left( \dfrac{|f^\gamma|_{0,0} +|\nabla_{x'} \, \Lambda_D^{-1} f^\gamma|_{0,0}}{\gamma}
+\dfrac{\langle g^\gamma \rangle_0 +\langle \nabla_{x'}\, \Lambda_D^{-1} g^\gamma \rangle_0}{\sqrt{\gamma}}
\right.\\
\left. +\dfrac{|\dot U^\gamma|_{0,0} +|\nabla_{x'} \, \Lambda_D^{-1}\dot U^\gamma|_{0,0}}{\gamma^2}
+\dfrac{\langle \dot U^\gamma|_{x_d=0} \rangle_0
+\langle \nabla_{x'} \, \Lambda_D^{-1}\dot U^\gamma|_{x_d=0} \rangle_{0}}{\gamma} \right) \, .
\end{multline}
\end{prop}

\begin{proof}
\textbf{1. $L^2$ estimate.} The first step is to prove the Kreiss-type estimate
\begin{align}\label{ij1}
|\dot U^\gamma_{3}|_{0,0} +\dfrac{\langle \dot U^\gamma_{3}|_{x_d=0} \rangle_{0}}{\sqrt{\gamma}} \leq
C\, \left( \dfrac{| f^\gamma|_{0,0}}{\gamma} +\dfrac{\langle g^\gamma\rangle_{0}}{\sqrt{\gamma}}
+\dfrac{|\dot U^\gamma|_{0,0}}{\gamma^2}+\dfrac{\langle \dot U^\gamma |_{x_d=0} \rangle_0}{\gamma}\right) \, .
\end{align}
For this we define the good set $G\subset \Sigma$ to be a neighborhood of the closure of
$\cup^{N_2}_{i=N_1+1}\cV_i$ such that $G$ is disjoint from $\overline{\Upsilon}$; here the uniform Lopatinskii
condition is satisfied.   The classical construction of Kreiss symmetrizers \cite{K,CP} provides us with an
$N\times N$ symbol $R(\zeta)$, homogeneous of degree $0$, such that for some positive constants $C$, $c$
and $\zeta/|\zeta|\in G$ we have
\begin{align}\label{ii27}
\begin{split}
&(a) \quad R(\zeta) =R(\zeta)^*\\
&(b) \quad -\re (R(\zeta) \, \bA(\zeta)) \geq c\, \gamma \, I_N\\
&(c) \quad R(\zeta) +C\, \cB(0)^* \, \cB(0) \geq c \, I_N.
\end{split}
\end{align}
We take a smooth extension of  $R$ to all $\zeta$ as a symbol of order $0$ such that \eqref{ii27}(a) holds.
Observe that by continuity \eqref{ii27}(c) implies
\begin{equation}\label{ii28}
R(\zeta)+C\, \cB(\eps U)^* \, \cB(\eps U)\geq c \, I_N \text{ for } \eps \text{ small enough.}
\end{equation}

As observed in \cite{W1} we may now use $R_D$, the singular Fourier multiplier associated to the symbol
$R(X,\gamma)$ as a Kreiss symmetrizer for the singular problem. Let $\chi_i =\bchi$, $v:=\bchi_D
\dot U^\gamma$, and denote by $\langle \cdot,\cdot \rangle$ the $L^2$ inner product on $b\Omega$.
Using the equation \eqref{ii2} to expand $\partial_d\langle v,R_Dv\rangle$ and integrating in $x_d$ over
$[0,\infty)$ we obtain
\begin{multline*}%\label{ii29}
-\langle v|_{x_d=0} ,R_Dv|_{x_d=0} \rangle =-2\, \re (R_D\, \bA_Dv,v) \\
-2\, \re(R_D \, \cD(\eps U) v,v) +2 \, \re(R_D \, {\bchi}_D f^\gamma,v)+O(|\dot U^\gamma|_{0,0}^2/\gamma^3).
\end{multline*}
From \eqref{ii27}(b), \eqref{ii28} and the localized G{\aa}rding inequality (Proposition \ref{n27}),
\begin{align}\label{ii30}
\re \langle (R +C \, \cB(\eps U)^* \, \cB(\eps U))_Dv|_{x_d=0} ,v|_{x_d=0} \rangle \geq c \,
\langle v|_{x_d=0} \rangle_{0}^2 -C\, \dfrac{\langle \dot U^\gamma|_{x_d=0}  \rangle^2_{0}}{\gamma},
\end{align}
we easily derive the estimate \eqref{ij1}.

\textbf{2. Estimate of $\nabla_{x'}\Lambda_D^{-1}\dot U^\gamma_3$.} Set $\tv :=\Lambda_D^{-1}v$ and for now
let $\tv'$ denote one of the tangential derivatives $\partial_j \tv$, $j=0,\dots,d-1$. Then $\tv'$ satisfies the system
\eqref{ic1}, where the truncation function $\bchi$ has changed but the forcing terms have exactly the same
expression.
%\begin{align}\label{ij2}
%\begin{split}
%&\partial_d \tv' +\bA_D \tv' +\cD(\eps U) \tv' =\partial_{x'} \, \Lambda_D^{-1} \, {\bchi}_D f
%+\partial_{x'} \, \Lambda_D^{-1} \, r_{-1,D} \dot U +\partial_{x'} \, r_{-1,D} \tv +\eps \, r_{0,D} \tv \ ,\\
%&\cB(\eps U) \tv_{x'}|_{x_d=0} =\partial_{x'} \, \Lambda_D^{-1} \, {\bchi}_D g +\partial_{x'} \, \Lambda_D^{-1} \,
%r_{-1,D}\dot U|_{x_d=0} +\partial_{x'} \, r_{-1,D} \tv|_{x_d=0} +\eps \, r_{0,D} \tv|_{x_d=0} \, ,
%\end{split}
%\end{align}
%where the forcing terms $r_{-1,D} \tv$ in \eqref{ij2} stand either for $(\cD (\eps U) -\Lambda_D^{-1} \, \cD (\eps U) \,
%\Lambda_D) \tv$ or for $(\cB (\eps U) -\Lambda_D^{-1} \, \cB (\eps U) \, \Lambda_D) \tv$ in the boundary conditions.
An argument just like the one that gave the estimate \eqref{ij1} yields
\begin{multline*}%\label{ij3}
|\nabla_{x'} \, \Lambda_D^{-1}\dot U^\gamma_{3}|_{0,0}
+\dfrac{\langle \nabla_{x'} \, \Lambda_D^{-1}\dot U^\gamma_{3}|_{x_d=0} \rangle_{0}}{\sqrt{\gamma}}
\le C\, \left( \dfrac{|\nabla_{x'} \, \Lambda_D^{-1} f^\gamma|_{0,0}}{\gamma}
+\dfrac{\langle \nabla_{x'} \, \Lambda_D^{-1} g^\gamma \rangle_{0}}{\sqrt{\gamma}} \right. \\
\left. +\dfrac{|\dot U^\gamma|_{0,0} +|\nabla_{x'}\Lambda_D^{-1}\dot U^\gamma|_{0,0}}{\gamma^2}
+\dfrac{\langle \dot U^\gamma|_{x_d=0} \rangle_0
+\langle \nabla_{x'} \, \Lambda_D^{-1}\dot U^\gamma|_{x_d=0} \rangle_{0}}{\gamma} \right) \, .
\end{multline*}
Here instead of \eqref{ii30} we have used
\begin{equation*}
\re \langle (R +C \, \cB(\eps U)^* \, \cB(\eps U))_D \tv'|_{x_d=0},\tv'|_{x_d=0} \rangle \geq c \,
\langle \tv'|_{x_d=0} \rangle_{0}^2
-C \, \dfrac{\langle \partial_{x'}\, \Lambda_D^{-1}\dot U^\gamma|_{x_d=0} \rangle^2_{0}}{\gamma} \, ,
\end{equation*}
to recover the estimate of the trace of $\tv'$. The $L^2$ estimates of the forcing terms in the interior and
on the boundary are exactly the same as in steps {\bf 4} and {\bf 5} of the previous proof.
%and, for example, we have estimated
%\begin{multline*}%\label{ij5}
%\re (R_D \, \partial_{x'} \, \Lambda_D^{-1} \, r_{-1,D}\dot U^\gamma,\tv_{x'}) \leq \delta \, \gamma \, |\tv_{x'}|^2_{0,0}
%+\dfrac{C_\delta}{\gamma} \, |\partial_{x'}\Lambda_D^{-1}r_{-1,D}\dot U^\gamma|^2_{0,0} \\
%\le \delta \, \gamma \, |\tv_{x'}|^2_{0,0} +\dfrac{C}{\gamma^3} \, \left(
%|\partial_{x'} \, \Lambda_D^{-1}\dot U^\gamma|^2_{0,0} +|\dot U^\gamma|^2_{0,0} \right) \, ,
%\end{multline*}
%where the final estimate was already used in the proof of Proposition \ref{hard} by using the fact that the $r_{-1,D}$
%operator equals the commutator $[\cD(\eps U),\bchi_D]$.
Since these estimates are actually simpler or similar to those we dealt with in Proposition \ref{hard}, we
feel free to skip the details that should become more or less familiar to the reader at this stage.
\end{proof}

\textbf{V) Conclusion.} We use the previous propositions to complete the proof of Proposition \ref{i5z}.

It remains to estimate $|\dot U^\gamma|_{0,1}$ and $|\dot U^\gamma|_{\infty,0}$.
Summing the estimates \eqref{ia4}, \eqref{ia9}, and \eqref{ia11} over $i\in\{1,\dots,N_2\}$ and absorbing error terms
from the right by taking $\gamma$ large, we derive
\begin{align}\label{i131}
|\dot U^\gamma|_{0,0} +\dfrac{\langle \dot U^\gamma|_{x_d=0} \rangle_0}{\sqrt{\gamma}}
+\dfrac{|\dot U_1^\gamma +\dot U_2^\gamma|_{\infty,0}}{\sqrt{\gamma}}
\leq C(K) \left( \dfrac{|\Lambda_D f^\gamma|_{0,0} +|\nabla_{x'}f^\gamma|_{0,0}}{\gamma^2}
+\dfrac{\langle\Lambda_D g^\gamma \rangle_0 +\langle \nabla_{x'}g^\gamma \rangle_0}{\gamma^{3/2}} \right) \, ,
\end{align}
where we have "forgotten" on the left of the inequality the additional control of $\nabla_{x'} \, \Lambda_D^{-1}
\dot U^\gamma$ (this term has played its role, meaning that it was used to absorb some bad terms appearing
on the right). This gives exactly \eqref{aprioriL2} with the additional control of $\dot U_1^\gamma +\dot U_2^\gamma$
in $L^\infty (L^2)$. This additional property will be used in the proof of Corollary \ref{estimH1} below.
\end{proof}

\begin{proof}[Proof of Corollary \ref{estimH1}]
We first estimate the first order tangential derivatives. We can apply the a priori estimate \eqref{aprioriL2} to the
problem satisfied by $\partial_{(x',\theta_0)} \dot U^\gamma$,  which is obtained by differentiating \eqref{i5}. This
yields
\begin{align}\label{i132}
|\dot U^\gamma|_{0,1} +\dfrac{\langle \dot U^\gamma|_{x_d=0} \rangle_1}{\sqrt{\gamma}} \leq C(K)
\left( \dfrac{|\Lambda_D f^\gamma|_{0,1} +|\nabla_{x'}f^\gamma|_{0,1}}{\gamma^2}
+\dfrac{\langle \Lambda_D g^\gamma \rangle_1 +\langle \nabla_{x'}g^\gamma \rangle_1}{\gamma^{3/2}}\right)
\end{align}
which is the same as \eqref{i6}, except for the absence of $|\dot U^\gamma|_{\infty,0}$ on the left. Here
we were able to treat commutators as forcing terms because, for example,
\begin{equation*}%\label{ii33}
[\cD(\eps U),\partial_{(x',\theta_0)}] \, \dot U^\gamma =-({\rm d}\cD(\eps U) \cdot \eps \partial_{(x',\theta_0)}U) \,
\dot U^\gamma \, ,
\end{equation*}
and the factor of $\eps$ coming out from the commutation allows us to estimate
\begin{equation*}%\label{ii34}
|\Lambda_D \, [\cD(\eps U),\partial_{(x',\theta_0)}] \, \dot U^\gamma|_{0,0} \leq C\, |\dot U^\gamma|_{0,1}.
\end{equation*}

It thus only remains to estimate the norm $|\dot U^\gamma|_{\infty,0}$. For $\delta_2>0$ to be chosen, we
take $0<\delta_1<\delta_2$ and consider a symbol of order zero in the extended calculus, $\chi^e
(\xi',\frac{k\beta}{\eps},\gamma)$, such that
\begin{align*}%\label{ii35}
\begin{split}
&0\leq \chi^e \leq 1 \, ,\\
&\chi^e \left( \xi',\dfrac{k\beta}{\eps},\gamma \right) =1 \text{ on } \left\{
|\xi',\gamma| \leq \delta_1 \, \dfrac{|k\, \beta|}{\eps} \right\} \, ,\\
&\mathrm{supp } \, \chi^e \subset \left\{ |\xi',\gamma| \leq \delta_2 \, \dfrac{|k\, \beta|}{\eps} \right\} \, .
\end{split}
\end{align*}
We then write $\dot U^\gamma =\chi^e_D\dot U^\gamma +(1-\chi^e_D)\dot U^\gamma$ and begin by estimating
$|(1-\chi^e_D)\dot U^\gamma|_{0,\infty}$ by using the Sobolev-type estimate
\begin{align}\label{ii36}
|(1-\chi^e_D)\dot U^\gamma|_{\infty,0} \leq C \, |(1-\chi^e_D)\, \partial_d \dot U^\gamma|_{0,0} +C\,
|(1-\chi^e_D) \dot U^\gamma|_{0,0} \leq C \, |(1-\chi^e_D)\, \partial_d \dot U^\gamma|_{0,0} +C\, |\dot U^\gamma|_{0,0}.
\end{align}
Using the equation \eqref{i5} and the fact that
\begin{equation*}
|X,\gamma| \, \left( 1-\chi^e \left( \xi',\dfrac{k\, \beta}{\eps},\gamma \right) \right) \leq C\, |\xi',\gamma| \, ,
\end{equation*}
we obtain
\begin{multline*}%\label{ii37}
|(1-\chi^e_D) \, \partial_d \dot U^\gamma|_{0,0} \leq |\bA_D \, (1-\chi^e_D) \dot U^\gamma|_{0,0}
+|(1-\chi^e_D) \, \cD \dot U^\gamma|_{0,0} +|(1-\chi^e_D) f^\gamma|_{0,0} \\
\le C\, \left( |\dot U^\gamma|_{0,1} +|f^\gamma|_{0,0} \right)
\leq C\, \left( |\dot U^\gamma|_{0,1} +\dfrac{|\Lambda_D f^\gamma|_{0,1}}{\gamma^2}\right),
\end{multline*}
where the last inequality follows from $|f^\gamma|_{0,0}\leq C|f^\gamma|_{0,1}/\gamma$.
With \eqref{ii36} this gives
\begin{align}\label{ii38}
|(1-\chi^e_D)\dot U^\gamma|_{\infty,0} \leq C \left( |\dot U^\gamma|_{0,1}
+\dfrac{|\Lambda_D (f^\gamma)|_{0,1}}{\gamma^2} \right).
\end{align}

To estimate $|\chi^e_D \dot U^\gamma|_{\infty,0}$ we observe that since $\beta \in\Upsilon$, we have for
$\delta_2>0$ chosen small enough
\begin{equation*}%\label{ii39}
\chi^e \left( \xi',\dfrac{k\, \beta}{\eps},\gamma \right) =\chi^e \left( \xi',\dfrac{k\, \beta}{\eps},\gamma \right) \,
\sum^{N_1}_{i=1} \chi_i (X,\gamma) \, ,
\end{equation*}
for the $\chi_i$ chosen in step \textbf{I)} of the proof of Proposition \ref{i5z}. Thus,
\begin{equation*}%\label{ii40}
|\chi^e_D \dot U^\gamma|_{\infty,0} \leq |\chi^e_D \, (\dot U_1^\gamma +\dot U_2^\gamma)|_{\infty,0}
\le |\dot U_1^\gamma +\dot U_2^\gamma|_{\infty,0} \, ,
\end{equation*}
with $\dot U_1^\gamma$ and $\dot U_2^\gamma$ defined in Propositions \ref{ia3} and \ref{hard}. We can
then apply the a priori estimate \eqref{i131} and obtain
%For fixed $i$ we define   $w=Q_D\chi_D\chi_{i,D}\dot U^\gamma$  as in \eqref{ii3}. From \eqref{ii13} and
%\eqref{ii15} we obtain
%\begin{align}\label{ii41}
%|w|_{\infty,0}\leq C\left(\langle w_{in}\rangle_0+\frac{|f^\gamma|_{0,0}}{\sqrt \gamma}
%+\frac{|\dot U^\gamma|_{0,0}}{\gamma^{3/2}}\right).
%\end{align}
%Substituting the estimate of $\langle w_{in}\rangle_0$ from \eqref{ii25} into \eqref{ii41}, we find
\begin{equation*}%\label{ii42}
|\chi^e_D \dot U^\gamma|_{\infty,0} \leq C \left( \dfrac{|\Lambda_D f^\gamma|_{0,0}
+|\nabla_{x'}f^\gamma|_{0,0}}{\gamma^{3/2}}
+\dfrac{\langle\Lambda_D g^\gamma \rangle_0 +\langle \nabla_{x'}g^\gamma \rangle_0}{\gamma} \right).
\end{equation*}
With \eqref{ii38} and \eqref{i132}, this completes the proof of Corollary \ref{estimH1}.
\end{proof}

Let us quickly observe that the genuine G{\aa}rding's inequality was used only once, in the proof of Proposition
\ref{i5z}, namely in \eqref{ii30}. In all other cases, we only used Plancherel's Theorem for Fourier multipliers.
This explains the slight difference between \eqref{ia11} and \eqref{ia4}, \eqref{ia9} for the powers of $\gamma$.
\bigskip

%\textup{(b) When the boundary condition is linear, that is, when $\psi(\eps U) \equiv \psi(0)$, one can replace
%Assumption \ref{nonlinbc} in Proposition \ref{i5z} by the weaker Assumption \ref{assumption3}, and the proof
%is somewhat simpler.   %For example,  the inequality \eqref{ii24} can be deduced directly from \eqref{ii18} using
%\begin{align}\label{ii43a}
%|w_{in}|\leq C\frac{\Lambda(X,\gamma)}{\gamma}|B Q_{in} w_{in}|,
%\end{align}
%an immediate consequence of \eqref{z1}. In this case the Garding inequality, Proposition \ref{n27}, is not
%needed in \eqref{ii30}.}

Next we ``localize the estimate" to $\Omega_T$. Since\footnote{Here ``$\sim$" denotes equivalence of norms with
constants independent of $\eps$ and $\gamma$.}
\begin{equation*}%\label{i7}
|\Lambda_D f^\gamma|_{0,1} \sim
\left| \left( \gamma,\partial_{x'}+\dfrac{\beta \, \partial_{\theta_0}}{\eps} \right) f^\gamma \right|_{0,1} \sim
\left| \left( \gamma,\partial_{x'}+\dfrac{\beta\, \partial_{\theta_0}}{\eps} \right) f \right|_{0,1,\gamma} \, ,
\end{equation*}
we can rewrite the a priori estimate \eqref{i6} for solutions to the linearized system \eqref{i3} as
\begin{multline}\label{i8}
|\dot U|_{\infty,0,\gamma} +|\dot U|_{0,1,\gamma} +\dfrac{\langle \dot U|_{x_d=0} \rangle_{1,\gamma}}{\sqrt{\gamma}} \\
\le C(K) \left(
\dfrac{|(\gamma,\partial_{x'}+\frac{\beta\partial_{\theta_0}}{\eps})f|_{0,1,\gamma}+|\nabla_{x'}f|_{0,1,\gamma}}{\gamma^2}
+\dfrac{\langle (\gamma,\partial_{x'}+\frac{\beta\partial_{\theta_0}}{\eps})g\rangle_{1,\gamma}
+\langle \nabla_{x'}g\rangle_{1,\gamma}}{\gamma^{3/2}} \right) \, .
\end{multline}

Suppose now that the singular problem \eqref{i3} is posed on $\Omega_T$ instead of $\Omega$. Given
$f \in L^2H^1_T$, one can define a Seeley extension $\tilde f\in L^2H^1$ such that
\begin{equation*}%\label{i9}
\left| \left( \gamma,\partial_{x'}+\dfrac{\beta\, \partial_{\theta_0}}{\eps} \right) \tilde f \right|_{0,1}
+|\nabla_{x'}\tilde f|_{0,1} \le C \left( \left| \left( \gamma,\partial_{x'}+\dfrac{\beta\, \partial_{\theta_0}}{\eps}
\right) f \right|_{0,1,T} +|\nabla_{x'}f|_{0,1,T} \right) \, ,
\end{equation*}
where $C$ is independent of $\gamma$, $\eps$, and $T$. It is readily checked that the same extension satisfies
\begin{align}\label{i10}
\left| \left( \gamma,\partial_{x'}+\dfrac{\beta\, \partial_{\theta_0}}{\eps} \right) \tilde f \right|_{0,1,\gamma}
+|\nabla_{x'}\tilde f|_{0,1,\gamma} \leq C\left( \left| \left( \gamma,\partial_{x'}+\dfrac{\beta\, \partial_{\theta_0}}{\eps}
\right) f \right|_{0,1,\gamma,T}+|\nabla_{x'}f|_{0,1,\gamma,T} \right),
\end{align}
where again $C$ is independent of $\gamma$, $\eps$, and $T$. We claim that changing $f$, $g$, and $U$ in
$\{ t>T \}$ does not affect the solution of \eqref{i3} in $\{ t<T\}$. (This causality principle is discussed further below
together with the existence of solutions to the linearized system \eqref{i3}.) Hence the estimates \eqref{i8} and
\eqref{i10} imply the following estimate for the singular problem on $\Omega_T$:
\begin{multline*}%\label{i11}
|\dot U|_{\infty,0,\gamma,T} +|\dot U|_{0,1,\gamma,T}
+\dfrac{\langle \dot U|_{x_d=0} \rangle_{1,\gamma,T}}{\sqrt{\gamma}} \\
\le C(K) \left( \dfrac{|(\gamma,\partial_{x'}+\frac{\beta\partial_{\theta_0}}{\eps}) f|_{0,1,\gamma,T}
+|\nabla_{x'} f|_{0,1,\gamma,T}}{\gamma^2}
+\dfrac{\langle (\gamma,\partial_{x'}+\frac{\beta\partial_{\theta_0}}{\eps}) g \rangle_{1,\gamma,T}
+\langle \nabla_{x'}g \rangle_{1,\gamma,T}}{\gamma^{3/2}} \right) \, .
\end{multline*}

Let us now consider the linearized singular problem \eqref{i3} on $\Omega_T$ with data of the form $\eps f$, $\eps g$
instead of $f$ and $g$. We note that
\begin{equation*}%\label{i12}
\left| \left( \gamma,\partial_{x'}+\dfrac{\beta \, \partial_{\theta_0}}{\eps} \right) \eps f \right|_{0,1,\gamma,T}
\leq C\, |f|_{0,2,\gamma,T} \text{ and }
\langle \left( \gamma,\partial_{x'}+\dfrac{\beta \, \partial_{\theta_0}}{\eps} \right) \eps g \rangle_{1,\gamma,T}
\leq C\, \langle g \rangle_{2,\gamma,T}.
\end{equation*}
Let us write the linearized operators on the left sides of \eqref{i3}(a) and (b) as $\bL'(\eps U)\dot U$ and $\bB'(\eps U)
\dot U$ respectively, and define
\begin{equation*}%\label{i13}
\cL'_\eps(U) \dot U :=\dfrac{1}{\eps} \, \bL'(\eps U)\dot U \, ,\quad
\cB'_\eps(U) \dot U :=\dfrac{1}{\eps} \, \bB'(\eps U) \dot U.
\end{equation*}
We have proved:

\begin{prop}\label{i14}
Fix $K>0$ and suppose $|\eps \partial_d U|_{C^{0,M_0-1}_T} +|U|_{C^{0,M_0}_T}\leq K$ for $\eps\in (0,1]$. There
exist positive constants $\eps_0(K)$, $\gamma_0(K)$ such that solutions of the singular problem
\begin{align}\label{i15}
\begin{split}
&\cL'_\eps(U)\dot U=f \text{ on } \Omega_T \, ,\\
&\cB'_\eps(U)\dot U=g \text{ on } b\Omega_T \, ,\\
&\dot U=0\text{ in }t<0,
\end{split}
\end{align}
satisfy
\begin{align}\label{i16}
|\dot U|_{\infty,0,\gamma,T} +|\dot U|_{0,1,\gamma,T}
+\dfrac{\langle \dot U|_{x_d=0} \rangle_{1,\gamma,T}}{\sqrt{\gamma}}
\leq C(K) \, \left( \dfrac{|f|_{0,2,\gamma,T}}{\gamma^2} +\dfrac{\langle g \rangle_{2,\gamma,T}}{\gamma^{3/2}} \right)
\end{align}
for $0<\eps \leq \eps_0(K)$, $\gamma \geq \gamma_0(K)$, and the constant $C(K)$ only depends on $K$.

The same estimate holds if $\cB(\eps U)$ in \eqref{i3} is replaced by $\cB(\eps U,\eps\cU)$ given in \eqref{ca2},
and $\cD(\eps U)$ is replaced by $\cD(\eps U,\eps\cU)$ given in \eqref{caa}, as long as there holds $|\eps \partial_d
(U,\cU)|_{C^{0,M_0-1}_T} +|U,\cU|_{C^{0,M_0}_T}\leq K$ for $\eps\in (0,1]$.
\end{prop}

\subsection{Well-posedness of the linearized singular equations}

\emph{\quad} In this short section, we explain why the analysis in \cite{C} gives existence and uniqueness of a solution
to the linearized singular problem \eqref{i15} for which the estimate \eqref{i16} holds. First of all, we can define a dual
problem for \eqref{i3} that reads
\begin{align}\label{dual}
\begin{split}
&\partial_d \dot U +\mathbb{A}^* \left( \partial_{x'}+\dfrac{\beta \, \partial_{\theta_0}}{\eps} \right)
\dot U +\tilde{\cD} (\eps U) \, \dot U =f(x,\theta_0) \quad \text{ on } \Omega \, ,\\
&\cM (\eps U) \, \dot U|_{x_d=0} =g(x',\theta_0) \, ,
\end{split}
\end{align}
where $\mathbb{A}^*$ is obtained from $\mathbb{A}$ by first multiplying the system by the constant matrix $B_d$,
then by integrating by parts on $\Omega$ and eventually by multiplying by $(B_d^T)^{-1}$. The zero order term is
also changed accordingly. Following the standard procedure described for instance in \cite[Chapter 4.4]{BS}, the
matrix $\cM$ giving the adjoint boundary conditions is chosen such that for all $v$ sufficiently close to the origin,
there holds
\begin{equation*}
B_D =\cB(v)_1^T \, \cB(v) +\cM(v)^T \, \cM_1(v) \, ,
\end{equation*}
where $\cB_1(v)$ and $\cM_1(v)$ are additional matrices depending smoothly on $v$.

The expression of $\mathbb{A}^*$ shows that this singular operator coincides with the operator obtained by applying
the substitution $\partial_{x'} \rightarrow \partial_{x'}+\beta \partial_{\theta_0}/\eps$ to the dual operator
\begin{equation*}
\dt +\sum_{j=1}^d B_j^T \, \partial_j =-L_0(\partial)^* \, .
\end{equation*}
It is known from the analysis in \cite[Chapter 8.3]{BS} that the latter constant multiplicity hyperbolic operator with
boundary conditions given by $\cM(v)$ gives rise to a boundary value problem in the "backward" WR class (one
just has to replace $\gamma$ by $-\gamma$ for this dual problem). When we apply the singular transformation
$\partial_{x'} \rightarrow \partial_{x'}+\beta \partial_{\theta_0}/\eps$ to the boundary value problem defined by
$(L_0(\partial)^*,\cM(\eps U))$, we can reproduce the analysis of the former section and show that the same
type of a priori estimate as in Proposition \ref{i5z} holds for \eqref{dual}.

For all fixed $\eps>0$ small enough, we have thus proved that both the forward problem \eqref{i3} and its dual
problem \eqref{dual} satisfy an a priori estimate with a loss of one tangential derivative. The estimates depend
very badly on $\eps$ because  singular derivative $\partial_{x'}+\beta \partial_{\theta_0}/\eps$ is estimated by
$1/\eps$ times the tangential $H^1$ norm with respect to $(x',\theta_0)$. Nevertheless, we can at this stage
reproduce the arguments of \cite{C} to show the existence and uniqueness of $L^2$ solutions to \eqref{i3}
when the source terms $f$ and $g$ satisfy $f,\partial_{\theta_0}f,\partial_{x'}f \in L^2(\Omega_T)$, $g \in
H^1(b\Omega_T)$. The analysis is actually much simpler than in \cite{C} because most of the technical
difficulties in \cite{C} arise from commutations with the hyperbolic operator. Here the hyperbolic operator
has constant coefficients so commutation with any scalar Fourier multiplier is exact. The analysis in \cite{C}
also shows that weak solutions are limit of strong solutions when the hyperbolic operator has constant
coefficients\footnote{Weak solutions are only "semi-strong" solutions when the hyperbolic operator has
variable coefficients.} so we can show that weak solutions satisfy the energy estimate \eqref{aprioriL2} with
constants that are \emph{uniform} with respect to the small parameter $\eps$. Such global in time estimates
imply the causality principle that "future does not affect the past" and can be localized to $\Omega_T$ by the
extension procedure previously described.

\subsection{Tame estimates}\label{tameex}

\emph{\quad} In this section we prove higher derivative estimates for the linearized singular problem \eqref{i3},
first in the ``pre-tame" form of Proposition \ref{i17}, and then in the final "tame" form of Proposition \ref{i33a},
which is suitable for Nash-Moser iteration. Propositions \ref{i25} and \ref{i30a} give pre-tame and tame estimates
for second derivatives.

\begin{nota}\label{i17a}
(a)  Let $L^\infty W^{1,\infty}\equiv L^\infty(\overline{\bR}_+,W^{1,\infty}(b\Omega))$ with norm $|U|_{L^\infty W^{1,\infty}}
:= |U|^*$. We also write $|U|_{L^\infty(\Omega)}=|U|_*$, $\langle V\rangle_{L^\infty(b\Omega)}=\langle V\rangle_*$,
$\langle V\rangle_{W^{1,\infty}(b\Omega)}=\langle V\rangle^*$, $|U|_{L^\infty(\Omega_T)}=|U|_{*,T}$, etc..

(b) For $k\in\mathbb{N}$, let $\partial^k$ denote the collection of tangential operators $\partial^\alpha_{(x',\theta_0)}$
with $|\alpha|=k$ ($\alpha$ is a multi-index). Sometimes $\partial^k$ is used to denote a particular member of this
collection. Set $\partial^0\phi=\phi$.

(c) For $k\in\{1,2,3,\dots\}$, denote by $\partial^{\langle k\rangle}\phi$ the set of products of the form $(\partial^{\alpha_1}
\phi_{i_1}) \cdots (\partial^{\alpha_r}\phi_{i_r})$ where $1\leq r\leq k$, $\alpha_1+\cdots\alpha_r=k$, $\alpha_i\geq 1$.
Set $\partial ^{\langle 0\rangle}\phi=1$.

(d) For $r\geq 0$, let $[r]$ denote the smallest integer greater than $r$.
\end{nota}

\noindent Our first goal is to prove the following ``pre-tame" estimate for solutions to \eqref{i15}.

\begin{prop}\label{i17}
Fix $K>0$ and suppose $|\eps\partial_d U|_{C^{0,M_0-1}}+|U|_{C^{0,M_0}}\leq K$ for $\eps\in (0,1]$. For $s\geq 0$
in any fixed finite interval there exist positive constants $\eps_0(K)$, $\gamma_0(K)$ such that the solution to the
linearized singular problem \eqref{i15} satisfies
\begin{multline}\label{i18}
|\dot U|_{\infty,s,\gamma,T} +|\dot U|_{0,s+1,\gamma,T}
+\dfrac{\langle \dot U|_{x_d=0} \rangle_{s+1,\gamma,T}}{\sqrt{\gamma}} \\
\le C(K) \left( \dfrac{|f|_{0,s+2,\gamma,T}}{\gamma^2} +\dfrac{\langle g\rangle_{s+2,\gamma,T}}{\gamma^{3/2}}
+\dfrac{|U|_{0,s+2,\gamma,T} \, |\dot U|_{*,T}}{\gamma^2}
+\dfrac{\langle U|_{x_d=0} \rangle_{s+2,\gamma,T} \, \langle \dot U|_{x_d=0}\rangle_{*,T}}{\gamma^{3/2}} \right) \, ,
\end{multline}
for $0<\eps\leq \eps_0(K)$ and $\gamma \geq \gamma_0(K)$.
\end{prop}

\begin{proof}
The problem satisfied by $\partial^s \dot U$ is
\begin{align*}%\label{i19}
\begin{split}
&\cL'_\eps(U)\partial^s \dot U=\partial^sf+ \dfrac{1}{\eps} \, [\cD(\eps U),\partial^s]\dot U \, ,\\
&\cB'_\eps(U)\partial^s \dot U=\partial^s g +\dfrac{1}{\eps} \, [\cB(\eps U),\partial^s]\dot U.
\end{split}
\end{align*}
In applying the estimate \eqref{i18} to this problem we must, for example, compute $\partial^2
([\cD(\eps U),\partial^s]\dot U)$, which is a sum of terms of the form\footnote{More precisely, each component
is a sum of such terms.}
\begin{equation*}%\label{i20}
\tilde{\cD} (\eps U) \, \partial^{\langle j\rangle}(\eps U) \, \partial^k \dot U,\text{ where }j+k=s+2, \;j\geq 1,
\end{equation*}
and $\tilde{\cD}$ is some smooth function of its argument. Since $j\geq 1$, we can rewrite this as
\begin{equation*}%\label{i21}
\tilde{\cD}(\eps U) \, \partial^{\langle {j-1}\rangle}(\eps U) \, \partial(\eps U) \, \partial^k \dot U.
\end{equation*}
Using  Moser estimates we obtain
\begin{equation*}%\label{i22}
\left| \dfrac{1}{\eps} \, \tilde{\cD}(\eps U) \, \partial^{\langle {j-1}\rangle}(\eps U) \, \partial(\eps U) \, \partial^k \dot U 
\right|_{0,\gamma,T} \leq C(K) \, |\dot U|_{*,T} \, |U|_{0,s+2,\gamma,T} +C(K) \, |\dot U|_{0,s+1,\gamma,T}.
\end{equation*}
The contribution from the final term on the right can be absorbed by taking $\gamma$ large enough; thus
this explains the third term on the right in \eqref{i18}. The final term on the right in \eqref{i18} arises by the
same argument applied to the boundary commutator.
\end{proof}

Next we prove estimates for the second derivatives
\begin{align*}%\label{i23}
%\begin{split}
&\cL''_\eps(U)(\dot U^a,\dot U^b) =\partial_v \cD(\eps U) \, (\dot U^a,\dot U^b) \, ,\\
&\cB''_\eps(U)(\dot U^a,\dot U^b) =\partial_v \cB(\eps U) \, (\dot U^a,\dot U^b) \, ,
%\end{split}
\end{align*}
where we use the notation
\begin{equation*}%\label{i24}
\partial_v \cD(\eps U) \, (\dot U^a,\dot U^b) :=\sum^N_{i=1} \left( \partial_{v_i}\cD(\eps U) \, \dot U^a_i \right) 
\, \dot U^b.
\end{equation*}

\begin{prop}\label{i25}
 We have
\begin{align*}%\label{i26}
%\begin{split}
&(a)\;|\cL''_\eps(U)(\dot U^a,\dot U^b)|_{\infty,s,\gamma,T}\leq \\
&\qquad C(|U|_{*,T})\left(|\dot U^a|_{\infty,s,\gamma,T}|\dot U^b|_{*,T}+|\dot U^b|_{\infty,s,\gamma,T} 
|\dot U^a|_{*,T}+\eps |U|_{\infty,s,\gamma,T}|\dot U^a|_{*,T}|\dot U^b|_{*,T}\right) \, ,\\
&(b)\;|\cL''_\eps(U)(\dot U^a,\dot U^b)|_{0,s+1,\gamma,T}\leq \\
&\qquad C(|U|_{*,T})\left(|\dot U^a|_{0,s+1,\gamma,T}|\dot U^b|_{*,T}+|\dot U^b|_{0,s+1,\gamma,T} 
|\dot U^a|_{*,T}+\eps |U|_{0,s+1,\gamma,T}|\dot U^a|_{*,T}|\dot U^b|_{*,T}\right) \, ,\\
&(c)\;\langle\cB''_\eps(U)(\dot U^a,\dot U^b)\rangle_{s,\gamma,T}\leq \\
&\qquad C(\langle U\rangle_{*,T})\left(\langle \dot U^a\rangle_{s,\gamma,T}\langle\dot U^b\rangle_{*,T} 
+\langle\dot U^b\rangle_{s,\gamma,T}\langle\dot U^a\rangle_{*,T} 
+\eps \langle U\rangle_{s,\gamma,T}\langle\dot U^a\rangle_{*,T}\langle\dot U^b\rangle_{*,T}\right).
%\end{split}
\end{align*}
\end{prop}

\begin{proof}
For $t\leq s$ one computes $\partial^t(\cL''_\eps(U)(\dot U^a,\dot U^b))$, which is a sum of terms of the form
\begin{equation*}%\label{i27}
\tilde{\cD}(\eps U) \, \partial^{\langle k\rangle}(\eps U) \, \partial^l\dot U^a \, \partial ^m\dot U^b, \text{ where }k+l+m=t.
\end{equation*}
 Thus, the first estimate follows directly from Moser estimates.  The remaining estimates are proved the same way.
\end{proof}

In the iteration scheme of section \ref{nashex} we will use $H^s_T$ spaces on the boundary, while in the interior we 
use the following spaces.

\begin{defn}\label{Espaces}
For $s\in\{0,1,2,\dots\}$ let
\begin{align*}%\label{i28}
%\begin{split}
&E^s_T=CH^s_T\cap L^2H^{s+1}_T, \text{ with the norm }|U(x,\theta_0)|_{E^s_T}:=|U|_{\infty,s,T}+|U|_{0,s+1,T}\\
&E^s_{\gamma,T}=CH^s_{\gamma,T}\cap L^2H^{s+1}_{\gamma,T}, \text{ with the norm } 
|U(x,\theta_0)|_{E^s_{\gamma,T}} := |U|_{\infty,s,\gamma,T} +|U|_{0,s+1,\gamma,T}.
%\end{split}
\end{align*}
 \end{defn}

\begin{rem}\label{embed}
\textup{By Sobolev embedding we have
\begin{align*}%\label{i29}
%\begin{split}
&s\geq [(d+1)/2] \Rightarrow E^s_T\subset CH^s_T\subset L^\infty(\Omega_T)\\
&s\geq[(d+1)/2]+1\Rightarrow E^s_T\subset CH^s_T\subset L^\infty(\overline{\bR}_+,W^{1,\infty}({b\Omega_T}))\\
&s\geq [(d+1)/2]+M_0\Rightarrow E^s_T\subset CH^s_T\subset C^{0,M_0}_T.\\
%\end{split}
\end{align*}
Note that $E^s_T$ is a Banach algebra for $s\geq [(d+1)/2]$.}
\end{rem}

By Proposition \ref{i25} and Remark \ref{embed} we immediately obtain:

\begin{prop}[Tame estimates for second derivatives]\label{i30a}
Let $b_0=[(d+1)/2]$ and suppose $s\geq 0$ lies in some finite interval. Then
\begin{align*}%\label{i30}
%\begin{split}
&(a)\;|\cL''_\eps(U)(\dot U^a,\dot U^b)|_{E^s_{\gamma,T}}\leq \\
&\qquad C(|U|_{E^{b_0}_{T}})\left(|\dot U^a|_{E^s_{\gamma,T}}|\dot U^b|_{E^{b_0}_{T}}
+|\dot U^b|_{E^s_{\gamma,T}}|\dot U^a|_{E^{b_0}_{T}}+\eps |U|_{E^s_{\gamma,T}} 
|\dot U^a|_{E^{b_0}_{T}} |\dot U^b|_{E^{b_0}_{T}}\right)\\
&(b)\langle\cB''_\eps(U)(\dot U^a,\dot U^b)\rangle_{s,\gamma,T}\leq \\
&\qquad C(\langle U\rangle_{b_0,T})\left(\langle \dot U^a\rangle_{s,\gamma,T}\langle\dot U^b\rangle_{b_0,T}
+\langle\dot U^b\rangle_{s,\gamma,T}\langle\dot U^a\rangle_{b_0,T}
+\eps \langle U\rangle_{s,\gamma,T}\langle\dot U^a\rangle_{b_0,T}\langle\dot U^b\rangle_{b_0,T}\right).
%\end{split}
\end{align*}
\end{prop}

In order to obtain a tame estimate for the linearized system suitable for Nash-Moser iteration, we must recast 
estimate \eqref{i18} without the $L^\infty$ norms of $\dot U^a$ and $\dot U^b$ on the right. First of all, we fix 
the paramater $K>0$. For instance, one may take $K=1$. This choice is arbitrary because we are interested 
in a small data result\footnote{If we were interested in a small time result for a given source term $G$, one 
would need to fix the constant $K$ in terms of $G$ and the parameters $\gamma,T$ would be chosen accordingly.}. 
We then choose constants $\eps_0(K)$, $\gamma_0(K)$ as in Proposition \ref{i17} so that the the estimate 
\eqref{i18} holds for $s \in [0,\tilde \alpha]$, where $\tilde \alpha$ is defined in \eqref{j6a}. \emph{For the 
remainder of Section \ref{tameex} and in Section \ref{nashex}, the parameter $K$ is fixed, and $\gamma$ is 
also fixed as $\gamma =\gamma_0(K)$.}

Let
\begin{equation*}%\label{i31}
\kappa := |U|_{0,\mu_0,\gamma,T} +\langle U|_{x_d=0} \rangle_{\mu_0,\gamma,T} \text{ where } \mu_0:=[(d+1)/2]+2.
\end{equation*}
Applying \eqref{i18} with $s=\mu_0-2$ we obtain for $0<\eps\leq\eps_0$:
\begin{multline}\label{i32}
|\dot U|_{\infty,\mu_0-2,\gamma,T}+|\dot U|_{0,\mu_0-1,\gamma,T}+\langle \dot U|_{x_d=0}\rangle_{\mu_0-1,\gamma,T} \\
\le C(K,\gamma) \, \left( |f|_{0,\mu_0,\gamma,T} +\langle g\rangle_{\mu_0,\gamma,T} 
+(|\dot U|_*+\langle \dot U\rangle_*) \, \kappa \right).
\end{multline}
By Remark \ref{embed} if $\kappa$ is chosen small enough, we can absorb the last term on the right in \eqref{i32} and 
obtain with a new constant $C$:
\begin{align}\label{i33}
|\dot U|_* +\langle \dot U|_{x_d=0} \rangle_* \leq C\, \big( |f|_{0,\mu_0,\gamma,T}+\langle g\rangle_{\mu_0,\gamma,T} 
\big) \, .
\end{align}
Substituting \eqref{i33} in  \eqref{i18}, we have proved

\begin{prop}[Tame estimate for the linearized system]\label{i33a}
Let $K$ and $\gamma=\gamma(K)$ be as fixed in Proposition \ref{i17} and suppose $|\eps \partial_dU|_{C^{0,M_0-1}} 
+|U|_{C^{0,M_0}}\leq K$ for $\eps\in (0,1]$.  Let  $\mu_0 = [\frac{d+1}{2}] + 2$ and $s\in[0,\tilde\alpha]$, where 
$\tilde{\alpha}$ is defined in \eqref{j6a}. There exist positive constants $\kappa_0(\gamma,T)$, $\eps_0$, and 
$C$  such that if
\begin{equation*}%\label{i34}
|U|_{0,\mu_0,\gamma,T} +\langle U|_{x_d=0} \rangle_{\mu_0,\gamma,T} \le \kappa_0,
\end{equation*}
then solutions $\dot U$ of the linearized system \eqref{i15} satisfy for $0<\eps \leq \eps_0$:
\begin{multline*}%\label{i35}
|\dot U|_{E^s_{\gamma,T}} +\langle \dot U|_{x_d=0} \rangle_{s+1,\gamma,T} \\
\le C\, \left[ |f|_{0,s+2,\gamma,T} +\langle g\rangle_{s+2,\gamma,T} +\left( |f|_{0,\mu_0\gamma,T} 
+\langle g\rangle_{\mu_0,\gamma,T} \right) \, \left(|U|_{0,s+2,\gamma,T} 
+\langle U|_{x_d=0} \rangle_{s+2,\gamma,T} \right) \right].
\end{multline*}
%The constants $C$ and $\kappa$ depend on $K=|U|_{C^{0,M_0}}$ since, e.g., $\gamma=\gamma(K)$ does.
\end{prop}
% Since now $|U|_{E^s_{\gamma,T}}\sim |U|_{E^s_{T}}$, we drop the subscript $\gamma$ in all norms.

\newpage
\section{Profile equations}\label{profile}

\subsection{The key subsystem in the $3\times 3$ strictly hyperbolic case}

%\qquad
To simplify the exposition we first treat the case of a $3\times 3$ strictly hyperbolic system and a boundary
frequency $\beta$ for which there is one single resonance in which two incoming modes interact to produce
an outgoing mode. This case already contains the main difficulties and is exactly the one we emphasize in
the example of Appendix \ref{exeuler}. We shall explain later on the relatively small changes needed to treat
the general case of systems satisfying the assumptions of Section \ref{assumptions}.

The leading profile is decomposed as
\begin{align}\label{c1}
\cV^0(x,\theta_1,\theta_2,\theta_3) =\sigma_1(x,\theta_1) \, r_1 +\sigma_3(x,\theta_3) \, r_3
\end{align}
where $\phi_2$ is the outgoing phase and the resonant triple $(n_1,n_2,n_3) \in \Z^3 \setminus \{ 0\}$ satisfies
\begin{align}\label{c2}
n_1\phi_1=n_2\phi_2+n_3\phi_3.
\end{align}
We can thus write
\begin{align}\label{c3}
\cV^0_{inc}=\sigma_1(x,\theta_1)\, r_1 +\sigma_3(x,\theta_3) \, r_3 \, ,\qquad \cV^1_{out}=\tau_2(x,\theta_2) \, r_2.
\end{align}
Furthermore we have
\begin{align}\label{c4}
\cV^0_{inc}|_{x_d=0,\theta_1=\theta_3=\theta_0} =a(x',\theta_0)\, e =a(x',\theta_0) \, (e_1+e_3) \, ,
\end{align}
so (recall that $e=e_1+e_3$, where $e_i\in \mathrm{span} \{r_i\}$, spans $\ker B \cap \E^s (\beta)$)
\begin{align}\label{c4a}
\sigma_i(x',0,\theta_0)\, r_i =a(x',\theta_0)\, e_i, \; \; i=1,3,
\end{align}
which determines the trace of $\sigma_i$ in terms of $a$.

Applying the operators $E_i$ for $i=1,3$ to \eqref{12}(a) and for $i=2$ to \eqref{12}(b) and using Corollary \ref{12c}
for \eqref{12}(c), we obtain the following system for the unknowns $(\sigma_1,\tau_2,\sigma_3,a)$, where
$\cA(x',\theta_0)$ denotes the unique function with mean zero such that $\partial_{\theta_0}\cA=a$:
\begin{align}\label{c5}
\begin{split}
&X_{\phi_1}\sigma_1 +c_1\, \sigma_1=0\\
&X_{\phi_3}\sigma_3 +c_3\, \sigma_3=0\\
&X_{\phi_2}\tau_2 +c_0\, \tau_2+c_2\int_0^{2\pi} \sigma_{1,n_1} \left(
x,\dfrac{n_2}{n_1} \, \theta_2 +\dfrac{n_3}{n_1} \, \theta_3 \right) \, \sigma_3(x,\theta_3) \, d\theta_3=0\\
&X_{Lop} \, \cA +c_4 \, \cA +c_5 \, \tau_2|_{x_d=0} +c_6 \, (a^2)^* =-b\cdot G^* \text{ on } b\Omega_T \, ,
\end{split}
\end{align}
where the first three equations hold on $\Omega_T$, and the constants $c_i$ are readily computed real constants.
Here $\sigma_{1,n_1}(x,\theta_1)$ is the image of the function $\sigma_1$ under the \emph{preparation map}
\begin{align}\label{c7a}
\sigma_1(x,\theta_1)=\sum_{k\in\bZ} f_k(x)e^{ik\theta_1} \longrightarrow \sum_{k\in\bZ}f_{kn_1}(x)e^{ikn_1\theta_1},
\end{align}
a map designed so that the integral in \eqref{c5} picks out resonances in the product of $\sigma_1$ and
$\sigma_3$.\footnote{Interaction integrals like the one in \eqref{c5} are discussed further in \cite{CGW1},
Proposition 2.13.}

Differentiating with respect to $\theta_0$, we rewrite the last equation of \eqref{c5} as
\begin{align}\label{c6}
X_{Lop}a +c_4 \, a +c_5 \, \partial_{\theta_0}\tau_2|_{x_d=0} +c_6 \, \partial_{\theta_0}(a^2)
=-b\cdot \partial_{\theta_0}G^* =:g \text{ on } b\Omega_T.
\end{align}
We now set $V:=\left(\sigma_1(x,\theta_1),\sigma_3(x,\theta_3),\tau_2(x,\theta_2),a(x',\theta_0)\right)$ and define
the interior and boundary operators for the leading profile system:
\begin{align}\label{c7}
\begin{split}
&\cL(V) :=\begin{pmatrix}
X_{\phi_1}\sigma_1 +c_1\, \sigma_1\\
X_{\phi_3}\sigma_3 +c_3\, \sigma_3\\
X_{\phi_2}\tau_2 +c_0\, \tau_2 +c_2\int_0^{2\pi} \sigma_{1,n_1} \left(
x,\dfrac{n_2}{n_1} \, \theta_2 +\dfrac{n_3}{n_1} \, \theta_3 \right) \, \sigma_3(x,\theta_3) \, d\theta_3
\end{pmatrix}\\
&\cB(V) :=X_{Lop}a +c_4\, a +c_5\, \partial_{\theta_0}\tau_2|_{x_d=0} +c_6\, \partial_{\theta_0}(a^2).
\end{split}
\end{align}

In this notation the profile subsystem becomes
\begin{align}\label{c8}
\begin{split}
&\cL(V)=0\text{ in }\Omega_T \, ,\\
&\cB(V)=g\text{ in }b\Omega_T \, ,\\
&V=0\text{ in }t\leq 0,
\end{split}
\end{align}
where the additional relations \eqref{c4a} hold giving the traces of $\sigma_1,\sigma_3$ in terms of $a$. The following existence result for the key subsystem is proved in section \ref{nashprof} using the tame estimates derived below in section \ref{tameprof}.

\begin{prop}\label{c8b}
Fix $T>0$, let $\alpha_0 :=\left[\frac{d+1}{2}\right]+1$, $\alpha :=2\alpha_0+4$, $\tilde\alpha :=2\alpha-\alpha_0$,
and suppose $g\in H^{\tilde \alpha-2}(b\Omega_T)$. Rewrite $V$ as $V=(V',a)$. If $\langle g \rangle_\alpha$ is
small enough, then there exists a solution $V$ of the profile subsystem  \eqref{c8} with $V'\in H^{\alpha-1}(\Omega_T)$,
$(V'|_{x_d=0},a)\in H^{\alpha-1}(b\Omega_T)$.
\end{prop}

\begin{rem}\label{c8a}
1)  \textup{Although the original problem is semilinear with a nonlinear zero-order boundary condition, the profile
system \eqref{c7} has a \emph{quasilinear} first-order boundary operator and an interior operator that includes
%several first-order linear terms and
a nonlinear, nonlocal, integro-pseudodifferential operator given by the interaction integral. The nonlocality arises
both from the $d\theta_3$-integration and from the pseudodifferential operator $\sigma_1\to\sigma_{1,n_1}$.}

2) \textup{Attempts to solve the system \eqref{c8} by a standard Picard iteration lead to a (fatal) loss of a derivative
from one iterate to the next. The reason is that $\sigma_1$ and $\sigma_3$ have the regularity of $a$ (incoming
transport equation), and therefore $\tau_2$ has the same regularity as $a$. However, the equation for $a$ involves
the derivative $\partial_{\theta_0}\tau_2|_{x_d=0}$ and this term induces the loss. Thus, we shall use Nash-Moser
iteration to prove Proposition \ref{c8b}.}
\end{rem}

\subsection{Tame estimates}\label{tameprof}

\qquad With $V=(\sigma_1,\sigma_3,\tau_2,a)$ and $\dot  V=(\dot\sigma_1,\dot\sigma_3,\dot\tau_2,\dot a)$, we
compute the first derivatives of $\cL$ and $\cB$:
\begin{align}\label{d1}
\begin{split}
&(a) \;\cL'(V)\dot V =\begin{pmatrix}
X_{\phi_1} \dot \sigma_1 +c_1\, \dot \sigma_1\\
X_{\phi_3} \dot \sigma_3 +c_3\, \dot \sigma_3\\
X_{\phi_2} \dot\tau_2 +c_0\, \dot\tau_2
+c_2 \int_0^{2\pi} \sigma_{1,n_1} \left( x,\dfrac{n_2}{n_1}\theta_2+\dfrac{n_3}{n_1}\theta_3 \right) \,
\dot\sigma_3(x,\theta_3) \, d\theta_3 \\
\qquad +c_2 \int_0^{2\pi} \sigma_{3,n_3} \left( x,-\dfrac{n_2}{n_3}\theta_2+\dfrac{n_3}{n_1}\theta_1 \right) \,
\dot\sigma_1(x,\theta_1) \, d\theta_1 \end{pmatrix}\\
&(b)\;\cB'(V)\dot V =X_{Lop}\dot a +c_4\, \dot a +c_5\, \partial_{\theta_0}\dot\tau_2|_{x_d=0}
+2c_6 \, (a\partial_{\theta_0}\dot a+\dot a\partial_{\theta_0}a).
\end{split}
\end{align}
Here we have used the property
\begin{align}\label{d2}
\int_0^{2\pi} \sigma_{3,n_3} \left( x,-\dfrac{n_2}{n_3}\theta_2+\dfrac{n_3}{n_1}\theta_1 \right) \,
\dot\sigma_1(x,\theta_1) \, d\theta_1 =\int_0^{2\pi} \dot\sigma_{1,n_1}
\left( x,\dfrac{n_2}{n_1}\theta_2+\dfrac{n_3}{n_1}\theta_3 \right) \, \sigma_3(x,\theta_3) \, d\theta_3,
\end{align}
which follows readily by looking at the Fourier series of the factors of the integrand. For the second derivatives
we obtain
\begin{align}\label{d3}
\begin{split}
&\cL''(V)(\dot V^a,\dot V^b) =c_2 \, \begin{pmatrix}
0 \\
0 \\
\int_0^{2\pi} \dot\sigma^a_{1,n_1} \left( x,\dfrac{n_2}{n_1}\theta_2 +\dfrac{n_3}{n_1}\theta_3 \right) \,
\dot\sigma^b_3(x,\theta_3) \, d\theta_3 \\
\qquad +\int_0^{2\pi} \dot\sigma^b_{1,n_1} \left( x,\dfrac{n_2}{n_1}\theta_2+\dfrac{n_3}{n_1}\theta_3 \right) \,
\dot\sigma^a_3(x,\theta_3) \, d\theta_3 \end{pmatrix} \\
&\cB''(V)(\dot V^a,\dot V^b) =2c_6 \, (\dot a^a \, \partial_{\theta_0}\dot a^b +\dot a^b \, \partial_{\theta_0}\dot a^a).
\end{split}
\end{align}

%\begin{nota}\label{d4}
%1. For a given $T>0$ we denote the interior norm $|V|_{H^s_\gamma(\Omega_T)}$ by $|V|_{s,\gamma}$ and the boundary norm $\langle %V\rangle_{H^s_\gamma(b\Omega_T)}$ by   $\langle V\rangle_{s,\gamma}.$   We always take $s$ to be a nonnegative integer.

%2. Similarly, we set  $|V|_{H^s(\Omega_T)}=|V|_{s}$ and $\langle V\rangle_{H^s(b\Omega_T})=\langle V\rangle_{s}.$

%3. For $r\geq 0$ let $[r]$ denote the smallest integer greater than $r$.
%\end{nota}

\begin{prop}[Tame estimates for second derivatives]\label{d5}
a) Let $b_0$ be the smallest integer greater than $\frac{d+2}{2}$ and let $s\geq 0$.  We have
\begin{align}\label{d6}
|\cL''(V)(\dot V^a,\dot V^b)|_{s,\gamma} \leq C \left( |\dot V^a|_{s,\gamma} \, |\dot V^b|_{b_0}
+|\dot V^b|_{s,\gamma} \, |\dot V^a|_{b_0} \right),
\end{align}
where $C$ is independent of $V$, $\gamma$, and $T$.

b) Let $c_0$ be the smallest integer greater than $\frac{d+1}{2}+1$ and let $s\geq 0$.  We have
\begin{align}\label{d7}
\langle \cB''(V)(\dot V^a,\dot V^b) \rangle_{s,\gamma} \leq C \left( \langle \dot V^a \rangle_{s+1,\gamma}
\, \langle \dot V^b \rangle_{c_0} +\langle \dot V^b \rangle_{s+1,\gamma} \, \langle \dot V^a \rangle_{c_0}\right),
\end{align}
where $C$ is independent of $V$, $\gamma$, and $T$.

In both estimates \eqref{d6}, \eqref{d7}, the constant $C$ can be chosen independent of $s$ in any fixed
finite interval.
\end{prop}

\begin{proof}
\textbf{a. } Moser estimates imply
\begin{align}\label{d7a}
|\dot\sigma^a_{1,n_1} (x,\frac{n_2}{n_1}\theta_2+\frac{n_3}{n_1}\theta_3) \, \dot\sigma^b_3(x,\theta_3)
|_{H^s_\gamma(x,\theta_2)} \leq C \, ( |\dot\sigma^a_{1,n_1}|_{s,\gamma} \, |\dot\sigma^b_3|_{L^\infty}
+|\dot\sigma^a_{1,n_1}|_{L^\infty} \, |\dot\sigma^b_3|_{H^s_\gamma(x)} ),
\end{align}
since $\dot\sigma^b_3$ is independent of $\theta_2$.  We have
\begin{align}\label{d7b}
\int^{2\pi}_0|\dot\sigma^b_3(x,\theta_3)|_{H^s_\gamma(x)} \, d\theta_3 \leq
C \, |\dot\sigma^b_3|_{L^2(\theta_3,H^s_\gamma(x))} \leq C\, |\dot\sigma^b_3|_{s,\gamma}.
\end{align}
The estimate \eqref{d6} now follows by Sobolev embedding and the fact that
\begin{align}\label{d8a}
|\dot\sigma^a_{1,n_1}|_{s,\gamma} \leq |\dot\sigma^a_{1}|_{s,\gamma}.
\end{align}

\textbf{b. }Again Moser estimates imply
\begin{align}\label{d8}
\langle\dot a^a\partial_{\theta_0}\dot a^b\rangle_{s,\gamma}\leq C \, \left(
\langle\dot a^a\rangle_{s,\gamma} \, \langle\partial_{\theta_0}\dot a^b\rangle_{L^\infty}
+\langle\dot a^a\rangle_{L^\infty} \, \langle\partial_{\theta_0}\dot a^b\rangle_{s,\gamma}\right),
\end{align}
so the estimate \eqref{d7} follows by Sobolev embedding.
\end{proof}

Next we derive tame energy estimates for the linearized problem
\begin{align}\label{d9}
\begin{split}
&\cL'(V)\dot V=f\text{ in }\Omega_T\\
&\cB'(V)\dot V=g\text{ in }b\Omega_T\\
&V=0 \text{ in }t<0,
\end{split}
\end{align}
where $f$ and $g$ vanish in $t<0$.  We begin with a simple proposition:
\begin{prop}\label{d10}
1. If the phase $\phi_p$ is incoming, solutions  of
\begin{align}\label{d10a}
X_{\phi_p}\sigma_p+c_p\sigma_p=h \text{ in }\Omega_T,\;\; \sigma_p=0\text{ in } t<0
\end{align}
satisfy for $\gamma$ large (depending on $c_p$):
\begin{align}\label{d11}
\sqrt{\gamma}|\sigma_p|_{s,\gamma}\leq C\left(\langle\sigma_p\rangle_{s,\gamma}+\frac{|h|_{s,\gamma}}{\sqrt{\gamma}}\right).
\end{align}

2. If the phase $\phi_p$ is outgoing, solutions of \eqref{d10a} satisfy for $\gamma$ large (depending on $c_p$):
\begin{align}\label{d12}
\sqrt{\gamma}|\sigma_p|_{s,\gamma}+\langle\sigma_p\rangle_{s,\gamma}\leq C\frac{|h|_{s,\gamma}}{\sqrt{\gamma}}.
\end{align}

3. Solutions in $\omega_T$ of
\begin{align}
X_{Lop}\dot a+c_4\dot a+2c_6( a\partial_{\theta_0}\dot a+\dot a\partial_{\theta_0} a)=g,\;\;\dot a=0\text{ in }t<0
\end{align}
 satisfy for $C_K$, $\gamma\geq \gamma_K$ (where $K=\langle a\rangle_{W^{1,\infty}}$):
\begin{align}\label{d13}
\sqrt{\gamma}\langle\dot a\rangle_{s,\gamma}\leq \frac{C_K}{\sqrt{\gamma}}\left(\langle g\rangle_{s,\gamma}+\langle a\rangle_{s+1,\gamma}\langle\dot a\rangle_{W^{1,\infty}}\right).
\end{align}
The second term on the right in \eqref{d13} does not appear in the $s=0$ estimate.
\end{prop}

\begin{proof}
\textbf{1. }To prove \eqref{d12} with $s=0$, one considers the problem satisfied by $e^{-\gamma t}\sigma_p$,
multiplies the equation by $e^{-\gamma t}\sigma_p$, integrates $dxd\theta_p$ on $\Omega_T$, and performs
obvious integrations by parts.  One then applies the $L^2$ estimate to the problem satisfied by tangential derivatives
$\gamma^{s-|\beta|}\partial^\beta_{x',\theta_p}\sigma_p$, $|\beta|\leq s$. Normal derivatives are estimated using the
equation and the tangential estimates.  The proof of \eqref{d12} is similar.

\textbf{2. }The proof of \eqref{d13} is similar, but in the higher derivative estimates one now has forcing terms that are
commutators involving $a$.  The commutators are linear combinations of terms of the form
\begin{align}\label{d14}
\gamma^{s-|\beta|}(\partial^{\beta_1}_{x',\theta_0}a)(\partial^{\beta_2}_{x',\theta_0}\partial_{\theta_0}\dot a)
\text{ where }|\beta_1|+|\beta_2|=|\beta|, \;|\beta_1|\geq 1,
\end{align}
or linear combinations of terms of the form
\begin{align}\label{d15}
\gamma^{s-|\beta|}(\partial^{\beta_1}_{x',\theta_0}\dot a)(\partial^{\beta_2}_{x',\theta_0}\partial_{\theta_0}a)
\text{ where }|\beta_1|+|\beta_2|=|\beta|, \;|\beta_2|\geq 1,
\end{align}
Applying Moser estimates to \eqref{d14} after writing $\partial^{\beta_1}a=\partial^{\beta_1'}\partial a$,  we obtain
\begin{align}\label{d16}
\begin{split}
&\langle\gamma^{s-|\beta|}(\partial^{\beta_1}_{x',\theta_0}a)
(\partial^{\beta_2}_{x',\theta_0}\partial_{\theta_0}\dot a)\rangle_{0,\gamma}\leq C \left(
\langle\partial a\rangle_{L^\infty}\langle\partial_{\theta_0}\dot a\rangle_{m-1,\gamma}
+\langle\partial a\rangle_{m-1,\gamma}\langle\partial_{\theta_0}\dot a\rangle_{L^\infty}\right)\\
&\qquad\leq C\left(\langle a\rangle_{W^{1,\infty}}\langle\dot a\rangle_{s,\gamma}+\langle a\rangle_{s,\gamma}
\langle\dot a\rangle_{W^{1,\infty}}\right).
\end{split}
\end{align}
The factor $C_K/\sqrt{\gamma}$ on the forcing term in the $L^2$ estimate allows the first term on the right to be
absorbed by taking $\gamma$ large.

The estimate of \eqref{d15} is similar, but we do not split the $\partial^{\beta_2}$ derivative, and after absorbing
a term we are left with $\frac{C_K}{\sqrt{\gamma}}\langle\dot a\rangle_{L^\infty}\langle a\rangle_{s+1,\gamma}$
on the right.
\end{proof}

We now use Proposition \ref{d10} to estimate solutions of the linearized problem \eqref{d9} by treating the interaction
integrals in \eqref{d1}(a) and the term $c_5\partial_{\theta_0}\tau_2$ in \eqref{d1}(b) as additional forcing terms. Setting
\begin{align}\label{d16a}
V_{inc,n}:=(\sigma_{1,n_1},\;\sigma_{3,n_3}),\; V_{inc}=(\sigma_1,\sigma_3), \;V_{out}=\tau_2,
\end{align}
estimating interaction integrals as in \eqref{d7a},\eqref{d7b}, and using \eqref{d8a}, we obtain immediately
\begin{align}\label{d17}
\begin{split}
&\sqrt{\gamma}|\dot V_{out}|_{s,\gamma}+\langle\dot V_{out}\rangle_{s,\gamma}\leq \frac{C}{\sqrt{\gamma}}
\left(|f|_{s,\gamma}+|V_{inc,n}|_{L^\infty}|\dot V_{inc}|_{s,\gamma}+|V_{inc}|_{s,\gamma}|\dot V_{inc}|_{L^\infty}\right)\\
&\sqrt{\gamma}|\partial_{\theta}\dot V_{out}|_{s,\gamma}+\langle\partial_{\theta_0}\dot V_{out}\rangle_{s,\gamma}\leq
\frac{C}{\sqrt{\gamma}}\left(|\partial_{\theta}f|_{s,\gamma}+|\partial_{\theta} V_{inc,n}|_{L^\infty}|\dot V_{inc}|_{s,\gamma}
+|\partial_{\theta} V_{inc}|_{m,\gamma}|\dot V_{inc}|_{L^\infty}\right)
\end{split}
\end{align}
and
\begin{align}\label{d18}
\begin{split}
&\sqrt{\gamma}|\dot V_{inc}|_{s,\gamma}\leq C\left(\langle\dot V_{inc}\rangle_{s,\gamma}
+\dfrac{|f|_{s,\gamma}}{\sqrt{\gamma}}\right)\\
&\sqrt{\gamma}\langle\dot V_{inc}\rangle_{s,\gamma}\leq
\frac{C_K}{\sqrt{\gamma}}\left(\langle g\rangle_{s,\gamma}+\langle \partial_{\theta_0}\dot V_{out}\rangle_{s,\gamma}
+\langle V_{inc}\rangle_{s+1,\gamma}\langle\dot V_{inc}\rangle_{W^{1,\infty}}\right).
\end{split}
\end{align}
This leads to the following ``pre-tame" estimate.

\begin{prop}\label{d19}
Let  $\mu_0 = [\frac{d+1}{2}] + 2$, fix $K_1>0$, and suppose $|V_{inc}|_{\mu_0}\leq K_1$.\footnote{In this Proposition
$\mu_0=[\frac{d}{2}] + 2$ would work, but we make the above choice so as not to have to redefine $\mu_0$ later.} For
$s\geq 0$ in any fixed finite interval, there exist constants $C(K_1)$, $\gamma(K_1)$ such that for $\gamma\geq
\gamma(K_1)$, solutions of the linearized problem \eqref{d9} satisfy
\begin{align}\label{d20}
\begin{split}
&\sqrt{\gamma}|\dot V_{out},\partial_{\theta}\dot V_{out},\dot V_{inc}|_{s,\gamma}+\langle\dot V_{out},\partial_{\theta_0}\dot V_{out}\rangle_{s,\gamma}+\sqrt{\gamma}\langle\dot V_{inc}\rangle_{s,\gamma}\leq\\
&\qquad\qquad\frac{C(K_1)}{\sqrt{\gamma}}\left(|f|_{s+1,\gamma}+\langle g\rangle_{s,\gamma}+|V_{inc}|_{s+1,\gamma}|\dot V_{inc}|_{L^\infty}+\langle V_{inc}\rangle_{s+1,\gamma}\langle\dot V_{inc}\rangle_{W^{1,\infty}}\right).
\end{split}
\end{align}
\end{prop}

\begin{proof}
We add the estimates \eqref{d17}, \eqref{d18} and absorb the terms
\begin{align}\label{d21}
\frac{C_K}{\sqrt{\gamma}}\left(\langle\dot V_{inc}\rangle_{s,\gamma}+\langle\partial_{\theta_0}\dot V_{out}\rangle_{s,\gamma}+|V_{inc,n},\partial_{\theta} V_{inc,n}|_{L^\infty}|\dot V_{inc}|_{s,\gamma}\right)
\end{align}
by taking $\gamma$ large, after observing that
\begin{align}\label{d22}
 |V_{inc,n},\partial_{\theta}V_{inc,n}|_{L^\infty}\leq C|V_{inc,n}|_{\mu_0}\leq C|V_{inc}|_{\mu_0}\text{ and }K=\langle V_{inc}\rangle_{W^{1,\infty}}\leq C|V_{inc}|_{\mu_0}.
\end{align}
\end{proof}

We now set $\tilde\alpha:=2\alpha-\alpha_0$ as in Proposition \ref{c8b} and choose constants $C(K_1)$, $\gamma(K_1)$
as in Proposition \ref{d19} corresponding to the interval $s\in[0,\tilde\alpha]$\footnote{The choice of $\tilde\alpha$ is
explained in section \ref{nashprof}.}. \emph{In the remainder of section \ref{tameprof} and also in section \ref{nashprof},
$\gamma$ is fixed as $\gamma=\gamma(K_1)$.}
%so from now on we shall write $|U|_s$ instead of $|U|_{s,\gamma}$ for $|U|_{H^s_\gamma(\Omega_T)}$.

To obtain a tame estimate we need to remove the terms depending on $\dot V_{inc}$ on the right side of \eqref{d20}. Let
\begin{align}\label{d23}
K_2=|V_{inc}|_{\mu_0,\gamma}+\langle V_{inc}\rangle_{\mu_0,\gamma}.
\end{align}
Applying \eqref{d20} with $s=\mu_0-1$ we obtain
\begin{align}\label{d24}
\sqrt{\gamma}|\dot V_{inc}|_{\mu_0-1,\gamma}+\sqrt{\gamma}\langle\dot V_{inc}\rangle_{\mu_0-1,\gamma}\leq
\frac{C(K_1)}{\sqrt{\gamma}}\left[|f|_{\mu_0,\gamma}+\langle g\rangle_{\mu_0-1,\gamma}+\left(|\dot V_{inc}|_{L^\infty}
+\langle\dot V_{inc}\rangle_{W^{1,\infty}}\right)K_2\right].
\end{align}
By Sobolev embedding if $K_2=K_2(\gamma,T)$ is chosen small enough, we can absorb the last term on the right in
\eqref{d24} and obtain with a new $C$
\begin{align}\label{d25}
|\dot V_{inc}|_{L^\infty}+\langle\dot V_{inc}\rangle_{W^{1,\infty}}\leq
C\left(|f|_{\mu_0,\gamma}+\langle g\rangle_{\mu_0-1,\gamma}\right)
\end{align}

For $\gamma$ fixed as above, setting $|U|_{s,\gamma}=|U|_s$ now and substituting \eqref{d25} in  \eqref{d20}, we
obtain the estimate in the following Proposition.

\begin{prop}[Tame estimate for the linearized system]\label{d26}
Let  $\mu_0 = [\frac{d+1}{2}] + 2$ and $s\in[0,\tilde\alpha]$.  There exists $\kappa=\kappa(\gamma,T)>0$ and a constant
$C$ depending on $\kappa$ such that if
\begin{align}\label{d27a}
 |V_{inc}|_{\mu_0}+\langle V_{inc}\rangle_{\mu_0}<\kappa,
 \end{align}
 then solutions of the linearized system \eqref{d9} satisfy
\begin{align}\label{d27}
|\dot V|_s+\langle\dot V\rangle_s\leq C\left[|f|_{s+1}+\langle g\rangle_s+\left(|f|_{\mu_0}+\langle g\rangle_{\mu_0-1}\right)
\left(|V|_{s+1}+\langle V\rangle_{s+1}\right)\right].
\end{align}
\end{prop}

\begin{proof}
We have proved the a priori estimate \eqref{d27} for sufficiently smooth solutions of the linearized system.
The existence of such solutions now follows by standard arguments, which we summarize here for completeness.

The unknown in the linearized system \eqref{d9} is $(\dot\sigma_1,\dot\sigma_3,\dot\tau_2,\dot a)$. We can solve
the linearized system by putting the terms that involve $\partial_{\theta_0}\dot\tau_2$ or $\partial_{\theta_0}\dot a$
on the right and replacing the operator $\partial_{\theta_0}$, when it acts on those terms, by a finite difference
operator $\partial_{\theta_0}^h$:
\begin{align}\label{dd1}
\begin{split}
&X_{\phi_1}\dot\sigma_1^h+c_1\dot\sigma_1^h=f_1\\
&X_{\phi_3}\dot\sigma_3^h+c_3\dot\sigma_3^h=f_2\\
&X_{\phi_2}\dot\tau_2^h+c_0\dot\tau_2^h=f_3-c_2\int_0^{2\pi}\sigma_{1,n_1}(x,\frac{n_2}{n_1}\theta_2+\frac{n_3}{n_1}\theta_3)\dot\sigma_3^h(x,\theta_3)d\theta_3-\\
&\quad \qquad\qquad c_2\int_0^{2\pi}\sigma_{3,n_3}(x,-\frac{n_2}{n_3}\theta_2+\frac{n_3}{n_1}\theta_1)\dot\sigma_1^h(x,\theta_1)d\theta_1\\
&X_{Lop}\dot a^h+c_4\dot a^h+2c_6\dot a^h\partial_{\theta_0}a=g-c_5\partial_{\theta_0}^h\dot\tau_2^h-2c_6a\partial_{\theta_0}^h\dot a^h.
\end{split}
\end{align}
For fixed $h\in (0,1]$ we can solve this system by Picard iteration, where $n$-th iterates appear on the right and
$(n+1)$-st iterates appear on the left.  All iterates are $0$ in $t<0$ and the  iterates with index zero are all $0$.

We then need an estimate that is uniform in $h$.  This can be done by repeating the existing proof of tame estimates,
using the operator $\partial_{\theta_0}^h$ in place of $\partial_{\theta_0}$.   This gives an estimate like \eqref{d27}:
 \begin{align}\label{dd27}
|\dot V^h|_s+\langle\dot V^h\rangle_s\leq C\left[|f|_{s+1}+\langle g\rangle_s+\left(|f|_{\mu_0}+\langle g\rangle_{\mu_0-1}\right)\left(|V|_{s+1}+\langle V\rangle_{s+1}\right)\right],
\end{align}
where $\dot V^h:=(\dot\sigma_1^h,\dot\sigma_3^h,\dot\tau_2^h,\dot a^h)$ and $C$ is uniform for $h\in (0,1]$.
Passing to a subsequence, we obtain the desired solution of the linearized system.
\end{proof}

\begin{rem}[Short time, given data]\label{d28}
\textup{For a given $T>0$ let $K_1$ and $\gamma=\gamma(K_1)$ be as in Proposition \ref{d19}. As we saw above
to obtain a tame estimate we need to take  $|\cV_{inc}|_{\mu_0}+\langle\cV_{inc}\rangle_{\mu_0}$ small.  In our
formulation of Theorem \ref{main} $T$ is fixed ahead of time and we achieve \eqref{d27a} by taking $G$ small in
an appropriate norm on $\Omega_T$. For a given $G$ as in \eqref{1} vanishing in $t<0$, another way to proceed
is to shrink $T$; that is, to work on $\Omega_{T_1}$ where $0<T_1<T$ is chosen so that $\gamma_1:=1/T_1\geq
\gamma(K_1)$ and so that
\begin{align}\label{d29}
|V_{inc}|_{H^{\mu_0}_{\gamma_1}(\Omega_{T_1})}+\langle V_{inc}\rangle_{H^{\mu_0}_{\gamma_1}(\omega_{T_1})}
\end{align}
is small enough to absorb the terms involving $\dot V_{inc}$ on the right in \eqref{d24}.  One again obtains an estimate
of the form \eqref{d27}, where now
\begin{align}\label{d30}
|U|_s:=|U|_{H^{s}_{\gamma_1}(\Omega_{T_1})}.
\end{align}
The  iteration scheme described in section \ref{nashprof} applies with no essential change  to this situation as well.}
\end{rem}

\subsection{The key subsystem in the general case}

Recall that $\{1,\dots,M\}=\cO\cup\cI$, where $\cO$ and $\cI$ contain the indices corresponding to outgoing and
incoming phases. We further decompose $\cO=\cO_1\cup \cO_2$, where $\cO_1$ consists of indices $m$ such
that $\phi_m$ is part of at least one triple of resonant phases with the property that the other two phases in that
triple are incoming. For a given $m\in\cO_1$ the phase $\phi_m$ might belong to more than one such triple.

Now instead of \eqref{c3}  we have
\begin{align}\label{g1}
\cV^0_{inc}=\sum_{m\in\cI}\sum_{k=1}^{\nu_{k_m}}\sigma_{m,k}(x,\theta_m)r_{m,k} \text{ and }
\cV^1_{out}=\sum_{m\in\cO_1}\sum_{k=1}^{\nu_{k_m}}\tau_{m,k}(x,\theta_m)r_{m,k},
\end{align}
since terms $\tau_{m,k}$ in the expansion of $\cV^1_{out}$ vanish if $m\in\cO_2$ as a consequence of \eqref{6}
and $\cV^0=\cV^0_{inc}$. Recalling that
\begin{align}\label{g1a}
e=\sum_{m\in\cI}\sum_{k=1}^{\nu_{k_m}}e_{m,k},\text{ where }e_{m,k}\in\mathrm{span} \{r_{m,k}\},
\end{align}
we see that in place of \eqref{c4} we now have
\begin{align}\label{g2}
\cV^0_{inc}|_{x_d=0;\;\theta_m=\theta_0,m\in I}=a(x',\theta_0)e=\sum_{m\in\cI}\sum_{k=1}^{\nu_{k_m}}a e_{m,k}=\sum_{m\in\cI}\sum_{k=1}^{\nu_{k_m}}\sigma_{m,k}(x',0,\theta_0)r_{m,k},
\end{align}
and thus
\begin{align}\label{g3}
\sigma_{m,k}(x',0,\theta_0)r_{m,k}=a(x',\theta_0) e_{m,k} \text{ for }m\in\cI,\; k=1,\dots,\nu_{k_m}.
\end{align}

Next we derive the formulas for $\cL(V)$ and $\cB(V)$ in the general case.  The unknown is now
\begin{align}\label{g4}
V=\left(\sigma_{m,k},m\in \cI,k=1,\dots,\nu_{k_m};\;\;\tau_{m,k},m\in\cO_1,k=1,\dots,\nu_{k_m};\;\;a\right)
\end{align}
Suppose $q\in\cO_1$ and that $(\phi_p,\phi_{q},\phi_s)$ is a resonant triple such that
\begin{align}
n_p\phi_p=n_{q}\phi_{q}+n_s\phi_s\text{ where } p,s\in \cI\text{ and }\mathrm{gcd}(n_p,n_{q},n_s)=1.
\end{align}
Applying the projectors $E_{m,k}$, $m\in\cI$, $k=1,\dots,\nu_{k_m}$ to \eqref{12}(a) and the projectors
$E_{q,l}$, $q\in\cO_1$, $l=1,\dots,\nu_{k_q}$ to \eqref{12}(b) we obtain
\begin{align}\label{g5}
\begin{split}
&\qquad\qquad \cL(V)=\\
&\begin{pmatrix}X_{\phi_{m}}\sigma_{m,k}+c_{m,k}\sigma_{m,k};\;m\in \cI,k=1,\dots,\nu_{k_m}\\X_{\phi_q}\tau_{q,l}
+c_{q,l}\tau_{q,l}+\sum_{k=1}^{\nu_{k_p}}\sum_{k'=1}^{\nu_{k_s}}d^{k,k'}_{q,l}\int^{2\pi}_0
(\sigma_{p,k})_{n_p}(x,\frac{n_q}{n_p}\theta_q+\frac{n_s}{n_p}\theta_s)\sigma_{s,k'}(x,\theta_s)d\theta_s+\\
(\mathrm{similar});q\in\cO_1, l=1,\dots,\nu_{k_q}\end{pmatrix}.
\end{split}
\end{align}
Here ``(similar)" denotes a finite sum of families of integrals similar to the family given explicitly in \eqref{g5}.
Here ``family" refers to the sum $\sum_{k,k'}$. One such family corresponds to each distinct resonant triple
involving the outgoing phase $\phi_q$ and two incoming phases.\footnote{We do not distinguish between
$(\phi_p,\phi_q,\phi_s)$ and $(\phi_p,\phi_s,\phi_q)$. We do distinguish between $(\phi_p,\phi_q,\phi_s)$ and
$(\phi_p,\phi_q,\phi_t)$.} The values of the real constants $c_{m,k}$, $d^{k,k'}_{q,l}$ are not important for our
analysis, but for example the $d^{k,k'}_{q,l}$ are given by\footnote{We have suppressed indices $r$, $s$ on
the $d^{k,k'}_{q,l}$.}
\begin{align}\label{g6}
d^{k,k'}_{q,l}=\frac{1}{2\pi}\ell_{q,l}\cdot[\partial_vD(0)(r_{p,k},r_{s,k'})+\partial_vD(0)(r_{s,k'},r_{p,k})].
\end{align}
By a computation similar to the one that produced \eqref{c6} we obtain from \eqref{12}(c)
\begin{align}\label{g7}
\cB(V)=X_{Lop}a+f_1 a+\sum_{q\in\cO_1}\sum_{l=1}^{\nu_{k_q}}f_{q,l}\partial_{\theta_0}\tau_{q,l}+f_2\partial_{\theta_0}(a^2).
\end{align}
for some real constants $f_1$, $f_2$, $f_{q,l}$.  For example, we have $f_2=-b\cdot[\psi'(0)(e,e)]$.
Thus, the system \eqref{12} may be rewritten
\begin{align}\label{g8}
\begin{split}
&\cL(V)=0\text{ in }\Omega_T\\
&\cB(V)=-b\cdot\partial_{\theta_0}G^*:=g\text{ on }b\Omega_T\\
&V=0\text{ in }t<0,
\end{split}
\end{align}
where the relations \eqref{g3} hold.

It is now a simple matter to write out the expressions for the first and second derivatives of $\cL$ and $\cB$.
For example, just as the interaction integral in \eqref{c7} gave rise to two integrals in the expression \eqref{d1} for
$\cL'(V)$ in the $3\times 3$ case, it is clear that each integral in \eqref{g5} will give rise to two integrals in the new
expression for $\cL'(V)$. The tame estimates for second derivatives are proved exactly as before, and Proposition
\ref{d5} holds verbatim in the general case. Proposition \ref{d10} is used exactly as before to prove estimates for
the linearized system. With the unknown $V$ as given in \eqref{g4} and after defining $V_{inc}$, $V_{out}$,
$\dot V_{inc}$, $\dot V_{out}$ in the obvious way, we see that the ``pre-tame" estimate of Proposition \ref{d19}
and the tame estimate of Proposition \ref{d26} hold verbatim in the general case. The iteration scheme of section
\ref{nashprof} depends only on the tame estimates. Thus, it applies here without change and Proposition \ref{c8b}
holds verbatim in the general case.

Once the key subsystem is solved, we can easily complete the solution of the full profile system \eqref{b8s}, \eqref{6}.
The precise result for the full system is proved below in Theorem \ref{mainprof}.

\newpage
\section{Error analysis}\label{error}

\emph{\quad} Here we carry out the error analysis sketched in section \ref{errori}, beginning with the proof of 
Proposition \ref{solve}.

\begin{proof}[Proof of Proposition \ref{solve}]
\textbf{1. Noncharacteristic modes.} We write
\begin{equation*}%\label{m1}
F(x,\theta) =F_0(x) +\sum_{\alpha\notin \cC} F_\alpha(x) \, e^{i\alpha\cdot\theta} 
+\sum_{m=1}^M \sum_{\alpha \in \cC_m \setminus \{ 0\}} F_\alpha(x) \, e^{i\alpha\cdot\theta},
\end{equation*}
and recall that the sums are finite. Set
\begin{equation*}%\label{m2}
n_\alpha=\sum_{j=1}^M\alpha_j\text{ and }\uomega=(\uomega_1,\dots,\uomega_M).
\end{equation*}
Since $EF=0$, we first note that $F_0$ vanishes. For any $\alpha$, there holds
\begin{equation*}%\label{m3}
\left(F_\alpha(x)e^{i\alpha\cdot\theta}\right)|_{\theta \rightarrow (\theta_0,\xi_d)} =F_\alpha(x) \, 
e^{in_\alpha\theta_0+i(\alpha\cdot\uomega)\xi_d},
\end{equation*}
and when $\alpha\notin\cC$ we look for $U_\alpha(x)$ such that
\begin{align}\label{m4}
\cL_0(\partial_{\theta_0},\partial_{\xi_d})U_\alpha(x)e^{in_\alpha\theta_0+i(\alpha\cdot\uomega)\xi_d}=F_\alpha(x)e^{in_\alpha\theta_0+i(\alpha\cdot\uomega)\xi_d}.
\end{align}
This holds if and only if
\begin{equation*}%\label{m5}
iL( n_\alpha\beta,\alpha\cdot\uomega)U_\alpha=F_\alpha.
\end{equation*}
The matrix on the left is invertible so we obtain a solution of \eqref{m4} for $\alpha\notin\cC$.

\textbf{2. Characteristic modes.} When $\alpha\in\cC_m\setminus \{ 0\}$ we have $\alpha\cdot\uomega 
=n_\alpha \, \uomega_m$, so
\begin{equation*}%\label{m6}
\left( F_\alpha(x) \, e^{i\alpha\cdot\theta}\right)|_{\theta \rightarrow (\theta_0,\xi_d)} =F_\alpha(x) \, 
e^{in_\alpha(\theta_0+\uomega_m\xi_d)}.
\end{equation*}
We can write
\begin{equation*}%\label{m6a}
\sum_{\alpha\in\cC_m\setminus \{ 0\}} F_\alpha(x) \, e^{in_\alpha(\theta_0+\uomega_m\xi_d)} 
%=\sum_{k\in \bZ\setminus 0}\sum_{\{\alpha\in\cC_m\setminus 0, n_\alpha=k\}} F_\alpha(x) e^{ik(\theta_0+\uomega_m\xi_d)}
=\sum_{k\in\bZ\setminus \{ 0\}} \cF_{m,k}(x) \, e^{ik(\theta_0+\uomega_m\xi_d)} \, ,
%&\quad +\sum_{k\in \bZ}\sum_{\{\alpha\in\cC_m\setminus 0, n_\alpha=k\}}(I-P_m) F_\alpha e^{ik(\theta_0+\uomega_m\xi_d)}
\end{equation*}
where
\begin{equation*}%\label{m6c}
\cF_{m,k}(x) := \sum_{\{\alpha\in\cC_m\setminus 0, n_\alpha=k\}} F_\alpha(x).
\end{equation*}
Since $E_mF=0$, we have for each $k\in\bZ\setminus \{ 0\}$ that $P_m\cF_{m,k}(x)=0$, so now we look for 
$U_{m,k}(x)$ such that
\begin{equation*}%\label{m7}
\cL_0(\partial_{\theta_0},\partial_{\xi_d}) U_{m,k}(x) \, e^{ik(\theta_0+\uomega_m\xi_d)} 
=(I-P_m) \, \cF_{m,k} \, e^{ik(\theta_0+\uomega_m\xi_d)}.
\end{equation*}
The latter relation holds if and only if
\begin{equation*}%\label{m8}
iL(k\beta,k\uomega_m) \, U_{m,k}(x) =ikL(d\phi_m) \, U_{m,k}(x)=(I-P_m)\cF_{m,k}(x),
\end{equation*}
which is solvable even though $L(d\phi_m)$ is singular. Finally, we take
\begin{equation*}%\label{m9}
\cU(x,\theta_0,\xi_d) =\sum_{\alpha\notin\cC} U_\alpha(x) \, e^{in_\alpha\theta_0+i(\alpha\cdot\uomega)\xi_d} 
+\sum_{m=1}^M \sum_{k \in \bZ \setminus \{ 0\}} U_{m,k}(x) \, e^{ik(\theta_0+\uomega_m\xi_d)} \, ,
\end{equation*}
which solves \eqref{27} as claimed.
\end{proof}

The existence theorems for profiles and for the exact solution to the singular system, Theorems \ref{mainprof} 
and \ref{k8b} respectively, are stated and proved in section \ref{iteration}; we shall only use the statement of these 
theorems here. In order to formulate the main result of this section we must make some preliminary choices.\\

\textbf{Choice of $\alpha$ and $\tilde\alpha$.} The conditions on the boundary datum $G(x',\theta_0)$ are slightly
different in Theorems \ref{mainprof} and \ref{k8b}. We need to choose $\alpha$, $\tilde\alpha$, and $G(x',\theta_0)$
so that both Theorems apply simultaneously. We also need $\alpha$ large enough so that we can apply Proposition
\ref{i14} in the step \eqref{m26} of the error analysis below. These conditions are  met if we take
\begin{align}\label{choice}
\alpha=\max\left(d+9, [(d+1)/2]+M_0+3\right)\text{ and } \tilde{\alpha}=2\alpha-[(d+1)/2]
\end{align}
and choose $G\in H^{\tilde{\alpha}}(b\Omega_T)$ such that $\langle G\rangle_{H^{\alpha+2}(b\Omega_T)}$ is small
enough. As in Theorem \ref{main}, we now rename the numbers in \eqref{choice} as $a$ and $\tilde a$.  Applying
Theorems \ref{mainprof} and \ref{k8b}, we now have for $0<\eps\leq \eps_0$ an exact solution $U_\eps(x,\theta_0)
\in E^{a-1}(\Omega_T)$ to the singular system \eqref{15} and profiles $\cV^0(x,\theta)\in H^{a-1}(\Omega_T)$,
$\cV^1(x,\theta)\in  H^{a-2}(\Omega_T)$ satisfying the equations \eqref{b8s} and \eqref{6}.\\

\textbf{Approximation.} Fix $\delta>0$. Using the Fourier series of $\cV^0$ and $\cV^1$, we choose trigonometric
polynomials $\cV^0_p(x,\theta)$ and $\cV^1_p(x,\theta)$ such that
\begin{align}\label{m10}
|\cV^0-\cV^0_p|_{H^{a-1}(\Omega_T)}<\delta,\quad |\cV^1-\cV^1_p|_{H^{a-2}(\Omega_T)}<\delta.
\end{align}
We can smooth the coefficients so that $\cV^0_p$ and $\cV^1_p$ lie in $H^\infty(\Omega_T)$ and so that \eqref{m10} 
still holds. Having made these choices, we can now state the main result of this section, which yields the final 
convergence result of Theorem \ref{main} as an immediate corollary.

\begin{theo}\label{main2}
We make the same Assumptions as in Theorem \ref{main} and let $a$ and $\tilde a$ be as just chosen.
Consider the leading order approximate solution to the singular semilinear system \eqref{15} given by
\begin{align}\label{m13a}
\cU^0_\eps(x,\theta_0):=\cV^0(x,\theta)|_{\theta \rightarrow (\theta_0,\frac{x_d}{\eps})},
\end{align}
and let $U_\eps(x,\theta_0)\in E^{a-1}(\Omega_T)$ be the exact solution to \eqref{15} just obtained. Then
\begin{align}\label{m14a}
\lim_{\eps\to 0}|U_\eps(x,\theta_0)-\cU^0_\eps(x,\theta_0)|_{E^{a-3}(\Omega_T)}=0.
\end{align}
\end{theo}

The following lemma, which is proved in \cite{CGW1}, Lemmas 2.7 and 2.25, by a simple argument based on Fourier
series, is an important tool in the proof.

\begin{lem}[Relation between norms]\label{m11}
For $m\in\bN$ suppose $f(x,\theta_j)\in H^{m+1}(\Omega_T)$, and set $f_\eps(x,\theta_0)
=f(x,\theta_0+\uomega_j\frac{x_d}{\eps})$. Then
\begin{align}\label{m12}
|f_\eps|_{E^m_T}\leq C|f|_{H^{m+1}(\Omega_T)}.
\end{align}
\end{lem}

\begin{proof}[Proof of Theorem \ref{main2}]
We shall fill in the sketch provided in section \ref{errori}.

\textbf{1. } First we use Proposition \ref{solve} to construct $\cU^2_p(x,\theta_0,\xi_d)$ satisfying
\begin{align}\label{m13}
\cL_0(\partial_{\theta_0},\partial_{\xi_d})\cU^2_p =\left[ -(I-E)\left(L(\partial)\cV^1_p+D(0)\cV^1_p 
+\partial_vD(0) (\cV^0_p,\cV^0_p) \right) \right]|_{\theta \rightarrow (\theta_0,\xi_d)}.
\end{align}
The function $\cU^2_p$ is  a trigonometric polynomial of the form \eqref{28} with $H^\infty$ coefficients. We then 
define the corrected approximate solution
\begin{align}\label{m14}
\cU_\eps(x,\theta_0) :=\left( \cV^0(x,\theta)+\eps\cV^1(x,\theta) \right)|_{\theta \rightarrow (\theta_0,\frac{x_d}{\eps})} 
+\eps^2 \, \cU^2_p(x,\theta_0,\frac{x_d}{\eps}).
\end{align}
Since $\cV^1\in H^{a-2}(\Omega_T)$, Lemma \ref{m11} implies $\cU_\eps\in E^{a-3}(\Omega_T)$.

\textbf{2. } Next we explain \eqref{30} and make precise the norms used on the right there. Using the identity 
\eqref{24} we compute
\begin{equation}\label{m15}
\bL_\eps(\cU_\eps) =\eps \, \left[ (\cL_0(\partial_{\theta_0},\partial_{\xi_d})\cU^2_p)|_{\xi_d=\frac{x_d}{\eps}} 
+\left(L(\partial)\cV^1+D(0)\cV^1+\partial_vD(0)\cV^0\cV^0\right)|_{\theta \rightarrow (\theta_0,\frac{x_d}{\eps})} 
\right] +O(\eps^2).
\end{equation}
Here the profile equations \eqref{3}(a),(b) imply that the terms of order $\eps^{-1}$ and $\eps^0$ vanish.
%The profile equation \eqref{6}(b) implies
%\begin{align}\label{m17}
%L(\partial)\cV^1+D(0)\cV^1+\partial_vD(0)\cV^0\cV^0=(I-E)\left(L(\partial)\cV^1+D(0)\cV^1+\partial_vD(0)\cV^0\cV^0\right),
%\end{align}
Using \eqref{m13} we can rewrite the  coefficient of $\eps$ in \eqref{m15} as
\begin{multline}\label{m16}
\left[ L(\partial)(\cV^1-\cV^1_p)+D(0)(\cV^1-\cV^1_p) +\partial_vD(0)(\cV^0\cV^0-\cV^0_p\cV^0_p) \right] 
|_{\theta \rightarrow (\theta_0,\frac{x_d}{\eps})} \\
+\left[ E\left(L(\partial)\cV^1_p+D(0)\cV^1_p+\partial_vD(0)\cV^0_p\cV^0_p\right) \right] 
|_{\theta \rightarrow (\theta_0,\frac{x_d}{\eps})} :=A+B \, .
\end{multline}
Using \eqref{m10}, Lemma \ref{m11}, and the fact that $E^s(\Omega_T)$ is a Banach algebra for $s\geq [(d+1)/2]$, 
we see that
\begin{align}\label{m17}
|A|_{E^{a-4}(\Omega_T)}<K\delta.
\end{align}

To estimate $B$ let
\begin{align}\label{m17a}
F=L(\partial)\cV^1+D(0)\cV^1+\partial_vD(0)\cV^0\cV^0 \text{ and }F_p 
=L(\partial)\cV^1_p+D(0)\cV^1_p+\partial_vD(0)\cV^0_p\cV^0_p.
\end{align}
The profile equation \eqref{6}(b) implies $EF=0$. Using continuity of the multiplication map \eqref{b4c}, we see that 
\eqref{m10} implies\footnote{Here $H^{a-3;2}_T$ denotes the space defined in \eqref{b4a}, but with the obvious 
restriction on the domain of $t$.}
\begin{align}\label{m17c}
|F-F_p|_{H^{a-3;2}_T}<K \delta.
\end{align}
From  the continuity of $E:H^{s;2}_T\to H^{s;1}_T$ and  Lemma \ref{m11} we then obtain
\begin{align}\label{m17b}
|B|_{E^{a-4}(\Omega_T)}=\left|(EF_p)|_{\theta \rightarrow (\theta_0,\frac{x_d}{\eps})}\right|_{E^{a-4}(\Omega_T)} 
=\left|(E(F-F_p))|_{\theta \rightarrow (\theta_0,\frac{x_d}{\eps})}\right|_{E^{a-4}(\Omega_T)}<K\delta.
\end{align}

\textbf{3. }The $O(\eps^2)$ terms in \eqref{m15} consist of
\begin{align}\label{m18}
\left| \eps^2 \, \left(L(\partial)\cU^2_p(x,\theta_0,\xi_d)\right)|_{\theta \rightarrow (\theta_0,\frac{x_d}{\eps})} 
\right|_{E^{a-4}(\Omega_T)}\leq \eps^2 C(\delta),
\end{align}
as well as terms coming from the Taylor expansion of $D(\eps\cU_\eps)\cU_\eps)$ like $(\eps^2 \partial_vD(0) 
\cV^0\cV^1)|_{\theta \rightarrow (\theta_0,\frac{x_d}{\eps})}$, all of which satisfy an estimate like \eqref{m18}. 
Setting $R_\eps(x,\theta_0):=\bL_\eps(\cU_\eps)$, we have shown
\begin{align}\label{m19}
|R_\eps|_{E^{a-4}(\Omega_T)}\leq \eps(K\delta+C(\delta)\eps).
\end{align}

\textbf{4. }The boundary profile equations \eqref{4} and the fact that the traces of $\cV^0$ and $\cV^1$ lie in 
$H^{a-1}(b\Omega_T)$ and $H^{a-2}(b\Omega_T)$, respectively,  imply
\begin{align}\label{m20}
\left\langle r_\eps(x',\theta_0)\right\rangle_{H^{a-2}(b\Omega_T)} \leq C(\delta) \, \eps^2,\text{ where } 
r_\eps :=\psi(\eps\cU_\eps) \, \cU_\eps -\eps G(x',\theta_0).
\end{align}
Indeed, these $O(\eps^2)$ terms include
\begin{align}\label{m21}
\left\langle \eps^2 \, B \, \cU^2_p(x',0,\theta_0,0)\right\rangle_{H^{a-2}(b\Omega_T)} \leq C(\delta) \eps^2 \, ,
\end{align}
and other terms satisfying the same estimate coming from the Taylor expansion of $\psi(\eps\cU_\eps)\cU_\eps$.

\textbf{5. }Next we consider the singular problem satisfied by the difference $W_\eps:=U_\eps-\cU_\eps$:
\begin{align}\label{m22}
\begin{split}
&\partial_d W_\eps +\bA\left(\partial_{x'}+\frac{\beta\partial_{\theta_0}}{\eps}\right)W_\eps 
+D_2(\eps U_\eps,\eps \cU_\eps)W_\eps=-R_\eps\\
&\psi_2(\eps U_\eps,\eps \cU_\eps)W_\eps=-r_\eps\text{ on }x_d=0\\
&W_\eps=0\text{ in }t<0,
\end{split}
\end{align}
where
\begin{align}\label{m23}
\begin{split}
&D_2(\eps U_\eps,\eps \cU_\eps)W_\eps:=D(\eps U_\eps)U_\eps-D(\eps \cU_\eps)\cU_\eps=\\
&\qquad D(\eps U_\eps)W_\eps +\left(\int^1_0\partial_vD(\eps\cU_\eps+s\eps (U_\eps-\cU_\eps))ds\right) \,  
(W_\eps,\eps\cU_\eps) \, ,
\end{split}
\end{align}
and $\psi_2(\eps U_\eps,\eps \cU_\eps)W_\eps$ is defined similarly.   Since $U_\eps\in E^{a-1}(\Omega_T)$ and 
$\cU_\eps\in E^{a-3}(\Omega_T)$ a short computation shows
%\begin{align}\label{m24}
%D_2(\eps \cU_\eps,\eps W_\eps)W_\eps=D(\eps U_\eps)W_\eps+\partial_vD(\eps U)(W_\eps,\eps U_\eps)+O(\eps^2 C(\delta))=D_1(\eps %U_\eps)W_\eps+O(\eps^2 C(\delta)),
%\end{align}
\begin{align}\label{m24}
\psi_2(\eps U_\eps,\eps \cU_\eps)W_\eps=\psi(\eps U_\eps)W_\eps+\partial_v\psi(\eps U)(W_\eps,\eps\cU_\eps) 
+O(C(\delta)\eps^2) = \cB(\eps U,\eps \cU)W_\eps+O(C(\delta)\eps^2),
\end{align}
where the error term is measured in $H^{a-3}(b\Omega_T)$ and $\cB$ is defined in \eqref{ca2}.  Similarly,
\begin{align}\label{m24z}
D_2(\eps U_\eps,\eps \cU_\eps)W_\eps= \cD(\eps U,\eps \cU)W_\eps+O(C(\delta)\eps^2) \text{ in }E^{a-3}(\Omega_T).
\end{align}
Thus, using \eqref{m19} and \eqref{m20} we find
\begin{align}\label{m25}
\begin{split}
&\partial_d W_\eps+\bA\left(\partial_{x'}+\frac{\beta\partial_{\theta_0}}{\eps}\right)W_\eps 
+\cD(\eps U_\eps,\eps \cU_\eps)W_\eps=\eps(K\delta+C(\delta)\eps)\text{ in }E^{a-4}(\Omega_T)\\
&\cB(\eps U_\eps,\eps \cU_\eps)W_\eps|_{x_d=0}=O(C(\delta)\eps^2)\text{ in }H^{a-3}(b\Omega_T)\\
&W_\eps=0\text{ in }t<0.
\end{split}
\end{align}
Applying the estimate of Proposition \ref{i14} we obtain
\begin{align}\label{m26}
|W_\eps|_{E^0(\Omega_T)}\leq K\delta+C(\delta)\eps,
\end{align}
which implies
\begin{align}
|U_\eps-\cU^0_\eps|_{E^0(\Omega_T)}\leq K\delta+C(\delta)\eps.
\end{align}
Fixing first $\delta$ small and then letting $\eps\to 0$ we have shown
\begin{align}\label{m27}
\lim_{\eps\to 0}|U_\eps-\cU^0_\eps|_{E^0(\Omega_T)}=0.
\end{align}
The family $U_\eps-\cU^0_\eps$, $0<\eps\leq \eps_0$, is bounded in $E^{a-2}(\Omega_T)$, so by interpolation
\eqref{m27} implies
\begin{align}\label{m28}
\lim_{\eps\to 0}|U_\eps-\cU^0_\eps|_{E^{a-3}(\Omega_T)}=0
\end{align}
as required.
\end{proof}

\section{Nash-Moser schemes}\label{iteration}

\subsection{Iteration scheme for profiles}\label{nashprof}

\qquad A good reference for the Nash-Moser scheme is \cite{AG}.  The method depends on having a family of smoothing operators with the following properties.   For $T>0$, $s\geq 0$, and $\gamma\geq 1$, we let:
\begin{align}\label{e1}
F^s_\gamma(\Omega_T):=\{u\in H^s_\gamma(\Omega_T), u=0 \text{ for }t<0\}.
\end{align}

\begin{lem}[\cite{A}, section 4]\label{e2}
There exists  a family of operators $S_\theta:F^0_\gamma(\Omega_T)\to\cap_{\beta\geq 0}F^\beta_\gamma(\Omega_T)$ such that:
\begin{align}\label{e3}
\begin{split}
&(a)\;|S_\theta u|_\beta\leq C\theta^{(\beta-\alpha)_+}|u|_\alpha\text{ for all }\alpha,\beta\geq 0\\
&(b)\;|S_\theta u-u|_\beta\leq C\theta^{(\beta-\alpha)}|u|_\alpha,\;\;0\leq\beta\leq \alpha\\
&(c)\;|\frac{d}{d\theta} S_\theta u|_\beta\leq C\theta^{(\beta-\alpha-1)}|u|_\alpha\text{ for all }\alpha,\beta\geq 0.
\end{split}
\end{align}
The constants are uniform for $\alpha$, $\beta$ in a bounded interval.

There is another family of operators $\tilde S_\theta$ acting on functions defined on the boundary and satisfying the above properties with respect to the norms $\langle u\rangle_s$ on $b\Omega_T$.\footnote{For $u$ defined on $\Omega_T$ we do not necessarily have equality of $(S_\theta u)|_{x_d=0}$ and $\tilde S_\theta (u|_{x_d=0}).$}.

\end{lem}

\textbf{Description of the scheme. } Our goal is to solve the problem \eqref{c8}:
\begin{align}\label{e2a}
\begin{split}
&\cL(V)=0\text{ in }\Omega_T\\
&\cB(V)=g\text{ in }b\Omega_T\\
&V=0\text{ in }t\leq 0.
\end{split}
\end{align}
The scheme starts with $V_0=0$.  Assume that $V_k$ are already given for $k=1,\dots,n$ and satisfy $V_k=0$ for $t<0$.  We define
\begin{align}\label{e4}
V_{n+1}=V_n+\dot V_n,
\end{align}
where the increment $\dot V_n$ is specified below.  Given $\theta_0\geq 1$, we set $\theta_n:=(\theta_0^2+n)^{1/2}$ and work with the smoothing operators $S_{\theta_n}$.   We decompose
\begin{align}\label{e5}
\cL(V_{n+1})-\cL(V_{n})=\cL'(V_n)\dot V_n+e_n'=\cL'(S_{\theta_n}V_n)\dot V_n+e_n'+e_n'',
\end{align}
where $e_n'$ denotes the usual ``quadratic error" of Newton's scheme and $e_n''$ the ``substitution error".  Similarly,
\begin{align}\label{e6}
\begin{split}
&\cB((V_{n+1})|_{x_d=0})-\cB((V_{n})|_{x_d=0})=\cB'((V_n)|_{x_d=0}))(\dot V_n|_{x_d=0}))+e_n'=\\
&\qquad\qquad\cB'((S_{\theta_n}V_n)|_{x_d=0})(\dot V_n|_{x_d=0})+\tilde e_n'+\tilde e_n''.
\end{split}
\end{align}

  The increment $\dot V_n$ is computed by solving the linearized problem
\begin{align}\label{e7}
\begin{split}
&\cL'(S_{\theta_n}V_n)\dot V_n=f_n\\
&\cB'((S_{\theta_n}V_n)|_{x_d=0})(\dot V_n|_{x_d=0})=g_n\\
&\dot V_n=0\text{ in }t<0,
\end{split}
\end{align}
where $f_n$ and $g_n$ are computed as we now describe.

  We set $e_n:=e_n'+e_n''$ and $\tilde e_n:=\tilde e_n'+\tilde e_n''$.   Given
\begin{align}\label{e8}
\begin{split}
&V_0:=0, \;\;f_0:=0,\;\; g_0:=\tilde S_{\theta_0}g, \;\;E_0:=0, \;\;\tilde E_0:=0,\\
&V_1,\dots,V_n\\
&f_1,\dots,f_{n-1},\; \;g_1,\dots,g_{n-1},\;\; e_0,\dots,e_{n-1},\;\;  \tilde e_0,\dots,\tilde e_{n-1},
\end{split}
\end{align}
we first compute for $n\geq 1$  the accumulated errors
\begin{align}\label{e9}
E_n:=\sum^{n-1}_{k=0}e_k,\;\;     \tilde E_n:=\sum^{n-1}_{k=0}\tilde e_k.
\end{align}
We then compute $f_n$ and $g_n$ from the equations
\begin{align}\label{e10}
\sum^n_{k=0}f_k+S_{\theta_n}E_n=0,\;\;\;\sum^n_{k=0}g_k+\tilde S_{\theta_n}\tilde E_n=\tilde S_{\theta_n} g,
\end{align}
solve  \eqref{e7} for $\dot V_n$, and finally compute $V_{n+1}$  from \eqref{e4}.

Next $e_n$, $\tilde e_n$ can be computed from \footnote{In the estimates of $e_n$ and  $\tilde e_n$, we instead use the formulas \eqref{f1},\eqref{f5}}
\begin{align}\label{e11}
\begin{split}
&\cL(V_{n+1})-\cL(V_n)=f_n+e_n\\
&\cB((V_{n+1})|_{x_d=0})-\cB((V_{n})|_{x_d=0})=g_n+\tilde e_n. \end{split}
\end{align}
 Thus the order of construction is
 \begin{align}\label{e12}
\dots\to (e_{n-1},\tilde e_{n-1})\to(E_{n},\tilde E_{n})\to (f_n,g_n)\to \dot V_n\to V_{n+1}\to (e_n,\tilde e_n)\to\dots \end{align}
Adding \eqref{e11} from $0$ to $n$ and using \eqref{e10} gives
\begin{align}\label{e13}
\begin{split}
&\cL(V_{n+1})=(I-S_{\theta_n})E_n+e_n\\
&\cB((V_{n+1})|_{x_d=0})-g=(\tilde S_{\theta_n}-I)g+(I-\tilde S_{\theta_n})\tilde E_n+\tilde e_n.
\end{split}
\end{align}
Since $S_{\theta_n}\to I$ and $\tilde S_{\theta_n}\to I$ as $n\to \infty$ and we expect $(e_n,\tilde e_n)\to 0$, we formally obtain a solution of \eqref{e2a} in the limit as $n\to \infty$.

\textbf{Induction assumption. }Let $\Delta_n:=\theta_{n+1}-\theta_n$ and observe that
\begin{align}\label{e14}
\frac{1}{3\theta_n}\leq \Delta_n=\sqrt{\theta_n^2+1}-\theta_n\leq \frac{1}{2\theta_n} \text{ for all }n\in\bN.
\end{align}
With $\mu_0=\left[\frac{d+1}{2}\right]+2$ as in Proposition \ref{d26}, we now set $\alpha_0:=\mu_0-1$ and  fix a choice of integers $\alpha_0<\alpha<\tilde\alpha$, whose values are explained below:
\begin{align}\label{e15}
\alpha=2\alpha_0+4\text{ and }\tilde\alpha=2\alpha-\alpha_0.
\end{align}
Given $\delta>0$ our induction assumption is:\\ \emph{}\\
\textbf{(H}$_{n-1}$\textbf{)}\emph{}\quad For all $k=0,\dots,n-1$ and for all $s\in [0,\tilde\alpha]\cap\bN$
\begin{align}\label{e16}
|\dot V_k|_s+\langle\dot V_k\rangle_s \leq \delta \theta_k^{s-\alpha-1}\Delta_k.
\end{align}

The main step in the proof of Theorem \ref{mainprof} is to show that for correctly chosen parameters $\delta>0$ (small) and $\theta_0\geq 1$ (large) and for small enough $g$, (H$_{n-1}$) implies (H$_n$).  At the end we will verify that (H$_0$) holds for
for small enough $g$.

First we state some easy consequences of (H$_{n-1}$).
\begin{lem}\label{e19}
If $\theta_0$ is large enough, then for $k=0,\dots,n$ and all integers $s\in[0,\tilde\alpha]$ we have
\begin{align}\label{e20}
|V_k|_s+\langle V_k\rangle_s\leq\begin{cases}C\delta \theta_k^{(s-\alpha)_+},\;\alpha\neq s\\C\delta\log\theta_k,\;\;\alpha=s\end{cases}.
\end{align}
\end{lem}
\begin{proof}
 This follows by writing $V_k=V_0+\sum_{j=0}^{k-1}\dot V_j$ and using the triangle inequality and an elementary comparison between Riemann sums and integrals.
\end{proof}

\begin{lem}\label{e21}
If $\theta_0$ is large enough, then for $k=0,\dots,n$ and all integers $s\in[0,\tilde\alpha+2]$ we have
\begin{align}\label{e22}
|S_{\theta_k}V_k|_s\leq\begin{cases}C\delta \theta_k^{(s-\alpha)_+},\;\alpha\neq s\\C\delta\log\theta_k,\;\;\alpha=s\end{cases}.
\end{align}
Moreover, for  $k=0,\dots,n$ and all integers $s\in[0,\tilde\alpha]$ we have
\begin{align}\label{e23}
|(I-S_{\theta_k})V_k|_s\leq\begin{cases}C\delta \theta_k^{s-\alpha}\log\theta_k,\;s\leq \alpha \\C\delta\theta_k^{s-\alpha},\;\;s>\alpha\end{cases}.
\end{align}
\end{lem}

\begin{proof}
This follows from Lemma \ref{e19} and the properties of the $S_{\theta}$.  For example,
we have
\begin{align}\label{e24}
\begin{split}
&|(I-S_{\theta_k})V_k|_s\leq 2|V_k|_s\leq C\delta\theta^{s-\alpha}\text{ for }s>\alpha\\
&|(I-S_{\theta_k})V_k|_s\leq C\theta^{s-\alpha}|V_k|_\alpha\leq  C\delta \theta^{s-\alpha}\log \theta_k\text{ for }s\leq\alpha.
\end{split}
\end{align}
\end{proof}

\textbf{Estimate of the quadratic errors. }From \eqref{e5} and \eqref{e6} we have
\begin{align}\label{f1}
\begin{split}
&(a) e_k'=\cL(V_{k+1})-\cL(V_k)-\cL'(V_k)\dot V_k=\int^1_0 (1-\tau)\cL''(V_k+\tau\dot V_k)(\dot V_k,\dot V_k)d\tau\\
&(b)\tilde e_k'=\cB(V_{k+1})-\cB(V_k)-\cB'(V_k)\dot V_k=\int^1_0 (1-\tau)\cB''(V_k+\tau\dot V_k)(\dot V_k,\dot V_k)d\tau\\
\end{split}
\end{align}
where the arguments in \eqref{f1}(b) are evaluated at $x_d=0$.

\begin{lem}\label{f2}
1)  For large enough $\theta_0$ we have for all $k=0,\dots,n-1$ and all integer $s\in [0,\tilde \alpha]$
\begin{align}\label{f3}
|e_k'|_s\leq C\delta^2\theta_k^{L_1(s)-1}\Delta_k,
\end{align}
where $L_1(s)=s+\alpha_0-2\alpha-2$.

2) For large enough $\theta_0$ we have for all $k=0,\dots,n-1$ and all integer $s\in [0,\tilde \alpha-1]$
\begin{align}\label{f4}
\langle \tilde e_k'\rangle _s\leq C\delta^2\theta_k^{L_2(s)-1}\Delta_k,
\end{align}
where $L_2(s)=s+\alpha_0-2\alpha-1$.
\end{lem}

\begin{proof}
Using \eqref{f1}(a), Proposition \ref{d5}, and the fact that $\alpha_0>b_0$,  we have
\begin{align}
|e_k'|_s\leq C |\dot V_k|_s |\dot V_k|_{\alpha_0}.
\end{align}
The estimate \eqref{f3} then follows by applying the assumption \eqref{e16} and using $\Delta_k\sim \frac{1}{\theta_k}$.
The estimate \eqref{f4} is proved the same way; the restriction $s\in [0,\tilde \alpha-1]$ reflects the loss of one derivative in \eqref{d7}.

\end{proof}

\textbf{Estimate of the substitution errors. }From \eqref{e5} and \eqref{e6} we have
\begin{align}\label{f5}
\begin{split}
&(a)\;e_k''=\int^1_0 \cL''(S_{\theta_k}V_k+\tau (V_k-S_{\theta_k}V_k)) \left(\dot V_k,(I-S_{\theta_k})V_k\right)d\tau\\
&(b)\;\tilde e_k''=\int^1_0 \cB''(S_{\theta_k}V_k+\tau (V_k-S_{\theta_k}V_k)) \left(\dot V_k,(I-S_{\theta_k})V_k\right)d\tau\\
\end{split}
\end{align}
where in \eqref{f5}(b) we have, for example, written $S_{\theta_k}V_k$ for $(S_{\theta_k}V_k)|_{x_d=0}$.

\begin{lem}\label{f6}
1)  For large enough $\theta_0$ we have for all $k=0,\dots,n-1$ and all integer $s\in [0,\tilde \alpha]$
\begin{align}\label{f7}
|e_k''|_s\leq C\delta^2\theta_k^{L_3(s)-1}\Delta_k,
\end{align}
where $L_3(s)=s+\alpha_0-2\alpha+1$.

2) For large enough $\theta_0$ we have for all $k=0,\dots,n-1$ and all integer $s\in [0,\tilde \alpha-2]$
\begin{align}\label{f8}
\langle \tilde e_k''\rangle _s\leq C\delta^2\theta_k^{L_4(s)-1}\Delta_k,
\end{align}
where $L_4(s)=s+\alpha_0-2\alpha+3$.
\end{lem}

\begin{proof}
Using \eqref{f5}(a) and Proposition \ref{d5}  we obtain
\begin{align}\label{f9}
|e_k''|_s\leq C\left(|\dot V_k|_s|(I-S_{\theta_k})V_k|_{\alpha_0}+|(I-S_{\theta_k})V_k|_s|\dot V_k|_{\alpha_0}\right).
\end{align}
The estimate \eqref{f7} now follows from (H$_{n-1}$) and Lemma \ref{e21}.  The estimate \eqref{f8} is proved the same way, after using the  trace estimate
\begin{align}\label{f10}
\langle (I-S_{\theta_k})V_k\rangle_{s+1}\leq C |(I-S_{\theta_k})V_k|_{s+2}.
\end{align}
The restriction $s\in [0,\tilde \alpha-2]$ reflects the subscript $s+2$ in \eqref{f10}.
\end{proof}

\textbf{Estimate of $(E_n,\tilde E_n)$ and $(f_n,g_n)$. }\;\;Since $e_k=e_k'+e_k''$ and $\tilde e_k=\tilde e_k'+\tilde e_k''$, we have
\begin{lem}\label{f11}
There exists $\theta_0$ sufficiently large so that
\begin{align}\label{f12}
|E_n|_{\tilde \alpha}\leq C\delta^2 \theta_n^{L_3(\tilde \alpha)} \text{ and } \langle\tilde E_n\rangle_{\tilde \alpha-2}\leq C\delta^2\theta_n^{L_4(\tilde \alpha-2)}.
\end{align}
\end{lem}
\begin{proof}
Viewing $E_n=\sum_{k=0}^{n-1}e_k$ as a Riemann sum and using $L_3(\tilde \alpha)>0$,\footnote{This determines $\tilde\alpha$ in \eqref{e15}.} we obtain the estimate of $E_n$ from \eqref{f3} and \eqref{f7}.  Since $L_4(\tilde \alpha-2)>0$, the estimate of $\tilde E_n$ is similar.
\end{proof}

From \eqref{e10} we have
\begin{align}\label{f13}
\begin{split}
&f_n=-(S_{\theta_n}-S_{\theta_{n-1}})E_{n-1}-S_{\theta_n}e_{n-1}\\
&g_n=(\tilde S_{\theta_n}-\tilde S_{\theta_{n-1}})g-(\tilde S_{\theta_n}-\tilde S_{\theta_{n-1}})\tilde E_{n-1}-\tilde S_{\theta_n}\tilde e_{n-1}
\end{split}
\end{align}

\begin{lem}\label{f14}
There exists $\theta_0$ sufficiently large so that for $s\in [0,\tilde \alpha+1]$ we have
\begin{align}\label{f15}
\begin{split}
&(a)\;|f_n|_s\leq C\delta^2\theta_n^{L_3(s)-1}\Delta_n\\
&(b)\;\langle g_n\rangle_s\leq C\delta^2\theta_n^{L_4(s)-1}\Delta_n+C\theta_n^{s-\alpha-1}\langle g\rangle_\alpha\Delta_n.
\end{split}
\end{align}
\end{lem}

\begin{proof}
Using \eqref{e3}(c), \eqref{f12}, and $s-\tilde \alpha+L_3(\tilde \alpha)=L_3(s)$, we find
\begin{align}\label{f16}
|(S_{\theta_n}-S_{\theta_{n-1}})E_{n-1}|_s\leq C\int^{\theta_n}_{\theta_{n-1}}\theta^{s-\tilde \alpha-1}|E_{n-1}|_{\tilde \alpha}d\theta\leq C\delta^2 \theta _{n-1}^{L_3(s)-1}\Delta_n.
\end{align}
From \eqref{f3},\eqref{f7} and the properties of $S_\theta$ we readily obtain
\begin{align}
|S_{\theta_n}e_{n-1}|_s\leq C\delta^2 \theta_n^{L_3(s)-1}\Delta_n,
\end{align}
and this gives \eqref{f15}(a).

The first term on the right in \eqref{f15}(b) arises similarly.  With
\begin{align}\label{f16a}
\langle (\tilde S_{\theta_n}-\tilde S_{\theta_{n-1}})g\rangle_s\leq C\int^{\theta_n}_{\theta_{n-1}}\theta^{s- \alpha-1} \langle g\rangle_{\alpha}d\theta\leq C\theta_n^{s- \alpha-1} \langle g\rangle_{\alpha}\Delta_n
\end{align}
we obtain \eqref{f15}(b).
\end{proof}

\textbf{Induction step.} We claim that for $\delta>0$ sufficiently small, the estimate for the linearized system \eqref{d27}
applies to \eqref{e7} and gives for $s\in [0,\tilde\alpha]$:
\begin{align}\label{f17}
|\dot V_n|_s+\langle\dot V_n\rangle_s\leq
C\left[|f_n|_{s+1}+\langle g_n\rangle_s+\left(|f_n|_{\alpha_0+1}+\langle g_n\rangle_{\alpha_0}\right)\left(|S_{\theta_n}V_n|_{s+1}+\langle S_{\theta_n}V_n\rangle_{s+1}\right)\right]
\end{align}
Indeed, \eqref{e22} and $\alpha>\alpha_0+2$ imply that for $\delta>0$ small enough, the requirement \eqref{d27a} holds.\footnote{We use a trace estimate like \eqref{f18} here as well.}
For the terms involving $f_n$ and $g_n$, except $\langle g_n\rangle_{\alpha_0}$, we substitute directly into \eqref{f17} the corresponding estimates from Lemma \ref{f14}.  For $\langle g_n\rangle_{\alpha_0}$ we have
\begin{align}\label{f17a}
\langle g_n\rangle_{\alpha_0}\leq C\left(\delta^2\theta_n^{L_4(\alpha_0)-1}\Delta_n+\theta_n^{-\alpha-2}\langle g\rangle_{\alpha_0+\alpha+1}\Delta_n\right),
\end{align}
where the last term arises from \eqref{f16a} with $s=\alpha_0$ and $\alpha$ replaced by $\alpha+\alpha_0+1$.  We also use
\begin{align}\label{f18}
\langle S_{\theta_n}V_n\rangle_{s+1}\leq |S_{\theta_n}V_n|_{s+2}\leq C\delta \theta_n^{(s+2-\alpha)_++1}
\end{align}
and a similar estimate for $|S_{\theta_n}V_n|_{s+1}$, which follow directly from \eqref{e22}.

Since $L_4(s)>L_3(s+1)$ this gives for $s\in [0,\tilde\alpha]$:
\begin{align}\label{f19}
\begin{split}
&|\dot V_n|_s+\langle\dot V_n\rangle_s\leq\\
 &\quad C\left[\delta^2\theta_n^{L_4(s)-1}\Delta_n+\theta_n^{s-\alpha-1}\langle g\rangle_\alpha\Delta_n+\left(\delta^2\theta_n^{L_4(\alpha_0)-1}\Delta_n+\theta_n^{-\alpha-2}\langle g\rangle_{\alpha_0+\alpha+1}\Delta_n\right)\delta\theta_n^{(s+2-\alpha)_++1}\right].
\end{split}
\end{align}
For $s\in [0,\tilde \alpha]$ the parameters $\alpha_0$ and $\alpha$ (recall \eqref{e15}) satisfy:
\begin{align}\label{f20}
\begin{split}
&L_4(s)\leq s-\alpha\\
&L_4(\alpha_0)+ (s+2-\alpha)_++1\leq s-\alpha\\
&(s+2-\alpha)_+<s.
\end{split}
\end{align}
Thus, we have proved (H$_n$), which is the content of the following lemma.
\begin{lem}[H$_n$]\label{f21}
If $\delta>0$ and $\langle g\rangle_{\alpha}/\delta$ are sufficiently small and $\theta_0$ sufficiently large, we have
\begin{align}\label{f24}
\;|\dot V_n|_s+\langle\dot V_n\rangle_s \leq \delta \theta_n^{s-\alpha-1}\Delta_n \text{ for all integer }s\in [0,\tilde\alpha].
\end{align}
\end{lem}

Still assuming (H$_{n-1}$) we now show:

\begin{lem}\label{f21a}
Suppose $n\geq 1$.  If $\delta>0$ is sufficiently small and $\theta_0$ sufficiently large, we have
\begin{align}\label{f21b}
\begin{split}
&(a)\;|\cL(V_n)|_s\leq \delta \theta_n^{s-\alpha-1}\text{ for all integer }s\in [0,\tilde\alpha]\\
&(b)\; \langle\cB(V_n)-g\rangle_s \leq \delta \theta_n^{s-\alpha-1}\text{ for all integer }s\in [0,\tilde\alpha-2].
\end{split}
\end{align}
\end{lem}

\begin{proof}
From \eqref{e13} we have
\begin{align}\label{f22}
\begin{split}
&(a)\;|\cL(V_n)|_s\leq |(I-S_{\theta_{n-1}})E_{n-1}|_s+|e_{n-1}|_s\\
&(b)\;\langle B(V_n)-g\rangle_s\leq\langle(\tilde S_{\theta_{n-1}}-I)g\rangle_s+\langle(I-\tilde S_{\theta_{n-1}})\tilde E_{n-1}\rangle_s+\langle\tilde e_{n-1}\rangle_s.
\end{split}
\end{align}
Using \eqref{e3} and the above estimates of $E_{n-1}$ and $e_{n-1}$,  we find
\begin{align}\label{f23}
\begin{split}
&|(I-S_{\theta_{n-1}})E_{n-1}|_s\leq C\theta_n^{s-\tilde\alpha}|E_{n-1}|_{\tilde \alpha}\leq C\delta^2\theta_n^{(s-\alpha-1)+(\alpha_0+2-\alpha)}\\
&|e_{n-1}|_s\leq C\delta^2\theta_n^{L_3(s)-1}\Delta_n,
\end{split}
\end{align}
which imply \eqref{f21b}(a) since $\alpha_0+2-\alpha<0$ and $L_3(s)<s-\alpha$.

The last two terms on the right in \eqref{f22}(b) are estimated similarly.  To finish we use
\begin{align}\label{f25}
\langle(\tilde S_{\theta_{n-1}}-I)g\rangle_s\leq C\theta_n^{s-(\tilde\alpha-2)}\langle g\rangle_{\tilde \alpha-2}\text{ for }s\leq \tilde\alpha-2
\end{align}
and observe that $s-\tilde\alpha+2<s-\alpha-1$.
\end{proof}

We now fix $\delta$ and $\theta_0$ as above and check (H$_0$).
\begin{lem}\label{f26}
If $\langle g\rangle_\alpha$ is small enough, then (H$_0$) holds.
\end{lem}

\begin{proof}
Applying the estimate for the linearized system to
\begin{align}
\begin{split}
&\cL'(0)\dot V_0=0\\
&\cB'(0)\dot V_0=S_{\theta_0}g
\end{split}
\end{align}
we obtain for integer $s\in [0,\tilde\alpha]$:

\begin{align}
|\dot V_0|_s+\langle\dot V_0\rangle_s\leq C\langle S_{\theta_0}g\rangle_s\leq C\begin{cases}\theta_0^{s-\alpha}\langle g\rangle_\alpha, \;s\geq \alpha\\\langle g\rangle_\alpha,\;s<\alpha\end{cases}.
\end{align}
Thus, (H$_0$) holds if $\langle g\rangle_\alpha$ is small enough.

\end{proof}

We can now complete the proof of Proposition \ref{c8b}.

\begin{proof}[Proof of Proposition \ref{c8b}]
We have
\begin{align}
V_n=V_{n-1}+\dot V_{n-1}=V_0+\sum^{n-1}_{k=0}\dot V_k=\sum^{n-1}_{k=0}\dot V_k.
\end{align}
Let $\mu=\alpha-1$.  Since $\theta_k\sim \sqrt{k}$ we have by (H$_n$)
\begin{align}
\sum_{k=0}^\infty |\dot V_k|_\mu+\sum_{k=0}^\infty \langle \dot V_k\rangle_\mu\leq \delta \sum_k \theta_k^{-2}\Delta_k\leq C\sum_k k^{-\frac{3}{2}}<\infty.
\end{align}
Thus, for some $V$ as described in Proposition \ref{c8b}, $V_k\to V$ in $H^\mu(\Omega_T)$ and $V_k|_{x_d=0}\to V|_{x_d=0}$ in  $H^\mu(\Omega_T)$.
This implies
\begin{align}
\cL(V_k)\to\cL(V)\text{ in }H^{\mu-1}(\Omega_T)\text{ and }\cB(V_k|_{x_d=0})\to\cB(V|_{x_d=0})\text{ in } H^{\mu-1}(b\Omega_T).
\end{align}
Applying Lemma \ref{f21a} with $s=\mu-1$ we conclude that $V$ is a solution of the profile system \eqref{c8}.
\end{proof}

\emph{\quad}Having solved the key subsystem we can now easily complete the solution of the full profile system \eqref{b8s}, \eqref{6} and obtain the following result.
\begin{theo}\label{mainprof}
Fix $T>0$, let $\alpha_0=\left[\frac{d+1}{2}\right]+1$, $\alpha=2\alpha_0+4$,  $\tilde\alpha=2\alpha-\alpha_0$, and suppose
$G\in H^{\tilde \alpha-1}(\Omega_T)$.    If $\langle G\rangle_{\alpha+1}$ is small enough, there exist solutions
 \begin{align}
 \cV^0=\cV^0_{inc}\in H^{\alpha-1}(\Omega_T),\; \cV^1=\underline{\cV}^1+\cV^1_{inc}+\cV^1_{out}\in H^{\alpha-2}(\Omega_T)
 \end{align}
 of the full profile system \eqref{b8s}, \eqref{6} satisfying\footnote{Here when we write $\cV^0_{inc}\in H^{\alpha-1}(\Omega_T)$, for example, we mean that the individual components of $\cV^0_{inc}$ lie in that space.}
\begin{align}\label{h2}
\begin{split}
&\cV^0=E\cV^0\in H^{\alpha-1}(\Omega_T),\; \cV^0_{inc}|_{x_d=0,\theta_j=\theta_0}\in H^{\alpha-1}(b\Omega_T)\\
&\cV^1_{out}=E\cV^1_{out} \in H^{\alpha-1}(\Omega_T),\;(E\cV^1_{out})|_{x_d=0,\theta_j=\theta_0}\in H^{\alpha-1}(b\Omega_T)\\
&\underline{\cV}^1\in H^{\alpha-2}(\Omega_T),\;(I-E)\cV^1_{inc}\in H^{\alpha-2}(\Omega_T),\;E\cV^1_{inc}\in H^{\alpha-2}(\Omega_T).
\end{split}
\end{align}
These statements remain true if $\alpha$ is increased and if $\tilde\alpha\geq 2\alpha-\alpha_0$.
\end{theo}

\begin{proof}
\emph{\quad}After the subsystem \eqref{12} is solved we know $\cV^0=\cV^0_{inc}=E\cV^0_{inc}$, $\cV^1_{out}=E\cV^1_{out}$, and $\alpha$, and these functions have the regularity described in Proposition \ref{c8b}.  Taking the mean of equations \eqref{6}(b)(c)(d), using the fact that the mean of the quadratic term in \eqref{6}(b) lies in $H^{\alpha-1}(\Omega_T)$, and applying the result of \cite{C} to the resulting weakly stable system, we conclude $\underline{\cV}^1\in H^{\alpha-2}(\Omega_T)$.   From \eqref{6}(a) we find
\begin{align}\label{h1}
(I-E)\cV^1=(I-E)\cV^1_{inc}\in H^{\alpha-2}(\Omega_T).
\end{align}
It remains to determine $E\cV^1_{inc}$.  Since the solvability condition \eqref{10a} holds, we can make a choice of
$E\cV^1_{inc}|_{x_d=0,\theta_j=\theta_0}\in H^{\alpha-2}(b\Omega_T)$ satisfying the boundary equation \eqref{10},
whose right side is now known and lies in $H^{\alpha-2}(b\Omega_T)$.\footnote{All terms on the right in \eqref{10}
lie in $H^{\alpha-1}(b\Omega_T)$, except the term involving $L(\partial)$. That term is actually more regular than
$H^{\alpha-2}(b\Omega_T)$, but we do not wish to introduce more refined spaces to capture this.}Finally, we determine
the components of $E\cV_{inc}$ by solving the transport equations determined by \eqref{6}(b), the choice of initial data,
and the initial condition \eqref{6}(d).  Observe that the interaction integrals corresponding to the quadratic term in
\eqref{6}(b) lie in $H^{\alpha-1}(\Omega_T)$.
\end{proof}

\subsection{Iteration scheme for the exact solution}\label{nashex}

\qquad The Nash-Moser scheme for the exact solution will use the scale of spaces $E^s_{\gamma,T}$ on $\Omega_T$
and $H^s_{\gamma,T}$ on $b\Omega_T$. Since  $T$ was fixed at the start and $\gamma$ was fixed in section
\ref{tameex}, we now drop these subscripts in the notation for norms and function spaces. For $s\geq 0$ we let
\begin{align}\label{j1}
\bF^s:=\{u(x,\theta_0)\in E^s, u=0 \text{ for }t<0\}.
\end{align}
Moreover, we shall now denote $E^s$ norms simply by $|U|_s$ and $H^s$ norms by $\langle U\rangle_s$.

\begin{lem}\label{j2}
There exists  a family of operators $S_\theta:\bF^0\to\cap_{\beta\geq 0}\bF^\beta$ such that:
\begin{align}\label{j3}
\begin{split}
&(a)\;|S_\theta u|_\beta\leq C\theta^{(\beta-\alpha)_+}|u|_\alpha\text{ for all }\alpha,\beta\geq 0\\
&(b)\;|S_\theta u-u|_\beta\leq C\theta^{(\beta-\alpha)}|u|_\alpha,\;\;0\leq\beta\leq \alpha\\
&(c)\;|\frac{d}{d\theta} S_\theta u|_\beta\leq C\theta^{(\beta-\alpha-1)}|u|_\alpha\text{ for all }\alpha,\beta\geq 0.
\end{split}
\end{align}
The constants are uniform for $\alpha$, $\beta$ in a bounded interval.

There is a family of operators $\tilde S_\theta$ acting on functions defined on the boundary and satisfying the
above properties with respect to the norms $\langle u\rangle_s$ on $b\Omega_T$, and we have
\begin{align}\label{j3a}
(S_\theta u)|_{x_d=0}=\tilde S_\theta(u|_{x_d=0}).
\end{align}
\end{lem}

\begin{proof}
Let $\tilde S_\theta$ be a standard family of smoothing operators, for example as in \cite{A}, acting in the $(x',\theta_0)$
variables on the scale of spaces $H^s$. For $U\in E^s$ simply treat $x_d$ as a parameter and define
\begin{align}\label{j4}
S_\theta U=\tilde S_\theta U(\cdot,x_d,\cdot).
\end{align}
The properties \eqref{j3} then follow immediately from the corresponding properties of the operators $\tilde S_\theta$.
\end{proof}

To avoid excessive repetition we use the  notation and arguments of section \ref{nashprof} as much as possible,
and just point out where changes are needed. Thus, we now denote the solution to the semilinear singular problem
\eqref{15} by $V$ instead of $U$,  write $g$ instead of $G$, and rewrite \eqref{15} as
\begin{align}\label{j5}
\begin{split}
&\cL(V)=0\text{ on }\Omega_T\\
&\cB(V)=g\text{ on }b\Omega_T\\
&V=0\text{ in }T<0,
\end{split}
\end{align}
where
\begin{align}\label{j5a}
\begin{split}
&\cL(V):=\frac{1}{\eps}\left(\partial_d V+\mathbb{A}\left(\partial_{x'}+\frac{\beta \partial_{\theta_0}}{\eps}\right)V+D(\eps V)V\right)\\
&\cB(V):=\frac{1}{\eps}\left(\psi(\eps V)V\right).
\end{split}
\end{align}

We now let\footnote{The parameter $\tilde\alpha$ is determined so that $L_2(\tilde\alpha)>0$ for $L_2(s)$ as in Lemma \ref{j13}. The definition of  $\alpha$ is chosen so that $\alpha_1<\alpha$ and the conditions \eqref{j30} hold.}
\begin{align}\label{j6a}
\alpha_0=\left[\frac{d+1}{2}\right],\; \alpha_1=\left[\frac{d+1}{2}\right]+M_0,\; \alpha=\max(2\alpha_0+3,\alpha_1+1), \; \tilde\alpha=2\alpha-\alpha_0.
\end{align}
The main result of this section is the following proposition.
\begin{theo}\label{k8b}
Fix $T>0$, define $\alpha$, $\alpha_0$, and $\tilde\alpha$ as in \eqref{j6a}, and suppose
$g\in H^{\tilde \alpha}$, where $g$ is the same as the function denoted $G$ in \eqref{15}.  There exists $\eps_0>0$ such that if $\langle g\rangle_{\alpha+2}$ is small enough, there exists a solution $V$ of
the system  \eqref{j5} on $\Omega_T$ for $0<\eps\leq \eps_0$ with $V\in E^{\alpha-1}$, $V|_{x_d=0}\in H^{\alpha}$.  Thus, $U_\eps=V$ is a solution of
the singular system \eqref{15} on $\Omega_T$ for $0<\eps\leq \eps_0$.
These statements remain true if $\alpha$ is increased and if $\tilde\alpha\geq 2\alpha-\alpha_0$.
\end{theo}

The linearized singular problem \eqref{i15} is now written
\begin{align}\label{j6}
\begin{split}
&\cL'(V)\dot V=f\text{ on }\Omega_T\\
&\cB'(V)\dot V=g\text{ on }b\Omega_T\\
&\dot V=0\text{ in }t<0.
\end{split}
\end{align}
With this notation the description of the scheme in section \ref{nashprof}  starting at line \eqref{e2a} applies here
word for word down to line \eqref{e14}.

\begin{rem}\label{j7}
 \textup{(a) In order to apply the tame estimate of Proposition \ref{i33a} to the linearized system \eqref{j6}, by
 Sobolev embedding (Remark \ref{embed}) it suffices to have
\begin{align}\label{j7a}
|\eps\partial_d V|_{\alpha_1-1}+|V|_{\alpha_1}< K'\text{ for }\eps\in (0,1], \text{ and }|V|_{\alpha_0+2}<\kappa
\end{align}
for some constant $K'$ depending  on $K$ and $\kappa$ as in Proposition \ref{i33a}.  In fact we use the slightly
weaker (because we use $E^s$ norms on the right) estimate for $s\in[0,\tilde \alpha]$:
\begin{align}\label{j8}
|\dot V|_s+\langle\dot V\rangle_{s+1}\leq C\left[|f|_{s+1}+\langle g\rangle_{s+2}+\left(|f|_{\alpha_0+1}+\langle g\rangle_{\alpha_0+2}\right)\left(|U|_{s+1}+\langle U\rangle_{s+2}\right)\right].
\end{align}}

 \textup{(b)  By Proposition \ref{i30a} when $|V|_{\alpha_0}\leq K'$,  the tame estimates for second derivatives now take the  form
\begin{align}\label{j9}
\begin{split}
&(a)\;|\cL''(V)(\dot V^a,\dot V^b)|_s\leq C\left(|\dot V^a|_s|\dot V^b|_{\alpha_0}+|\dot V^b|_s|\dot V^a|_{\alpha_0}+\eps |V|_s|\dot V^a|_{\alpha_0}|\dot V^b|_{\alpha_0}\right)\\
&(b)\langle\cB''(V)(\dot V^a,\dot V^b)\rangle_{s}\leq C\left(\langle \dot V^a\rangle_{s}\langle\dot V^b\rangle_{\alpha_0}+\langle\dot V^b\rangle_{s}\langle\dot V^a\rangle_{\alpha_0}+\eps \langle V\rangle_{s}\langle\dot V^a\rangle_{\alpha_0}\langle\dot V^b\rangle_{\alpha_0}\right).
\end{split}
\end{align}}
\end{rem}

With $\alpha$ and $\tilde{\alpha}$ redefined as in \eqref{j6a}, for a given $\delta>0$ the induction hypothesis
($H_{n-1}$) is now\\

\textbf{(H}$_{n-1}$\textbf{)}\emph{}\quad For all $k=0,\dots,n-1$ and for all $s\in [0,\tilde\alpha]\cap\bN$
\begin{align}\label{j9a}
|\dot V_k|_s+\langle\dot V_k\rangle_{s+1} \leq \delta \theta_k^{s-\alpha-1}\Delta_k.
\end{align}

Lemmas \ref{e19} and \ref{e21} are now replaced, with no real change in the proofs, by the following two lemmas.

\begin{lem}\label{j9b}
If $\theta_0$ is large enough, then for $k=0,\dots,n$ and all integers $s\in[0,\tilde\alpha]$ we have
\begin{align}\label{j20}
|V_k|_s+\langle V_k\rangle_{s+1}\leq\begin{cases}C\delta \theta_k^{(s-\alpha)_+},\;\alpha\neq s\\C\delta\log\theta_k,\;\;\alpha=s\end{cases}.
\end{align}
\end{lem}

\begin{lem}\label{j21}
If $\theta_0$ is large enough, then for $k=0,\dots,n$ and all integers $s\in[0,\tilde\alpha+2]$ we have
\begin{align}\label{j22}
|S_{\theta_k}V_k|_s+\langle S_{\theta_k}V_k\rangle_{s+1}\leq\begin{cases}C\delta \theta_k^{(s-\alpha)_+},\;\alpha\neq s\\C\delta\log\theta_k,\;\;\alpha=s\end{cases}.
\end{align}
For  $k=0,\dots,n$ and all integers $s\in[0,\tilde\alpha]$ we have
\begin{align}\label{j23}
|(I-S_{\theta_k})V_k|_s+\langle(I-S_{\theta_k})V_k\rangle_{s+1}\leq\begin{cases}C\delta \theta_k^{s-\alpha}\log\theta_k,\;s\leq \alpha \\C\delta\theta_k^{s-\alpha},\;\;s>\alpha\end{cases}.
\end{align}
\end{lem}

We have used \eqref{j4} for the estimate on traces in Lemma \ref{j21}.   In place of Lemma \ref{f2} we now have:

\begin{lem}\label{j10}
1)  For large enough $\theta_0$ and small enough $\delta$ we have for all $k=0,\dots,n-1$ and all integer $s\in [0,\tilde \alpha]$
\begin{align}\label{j11}
|e_k'|_s\leq C\delta^2\theta_k^{L_1(s)-1}\Delta_k,
\end{align}
where $L_1(s)=\max(s+\alpha_0-2\alpha-2,(s-\alpha)_++2\alpha_0-2\alpha-1)$.

2) For large enough $\theta_0$ and small enough $\delta$ we have for all $k=0,\dots,n-1$ and all integer $s\in [0,\tilde \alpha]$
\begin{align}\label{j12}
\langle \tilde e_k'\rangle _{s+1}\leq C\delta^2\theta_k^{L_1(s)-1}\Delta_k.
\end{align}
\end{lem}

\begin{proof}
Again we use the formulas in \eqref{f1}.
By Lemma \ref{j9b} and ($H_{n-1}$) we see that for $\delta$ small enough, $|V_k+\tau \dot V_k|_{\alpha_0}\leq K'$, so we can apply  the estimates \eqref{j9}. The new definition of $L_1(s)$ reflects the third term on the right in the estimates \eqref{j9}.

\end{proof}

In place of Lemma \ref{f6}, the estimate of substitution errors, we now have:

\begin{lem}\label{j13}
1)  For large enough $\theta_0$ and small enough $\delta$ we have for all $k=0,\dots,n-1$ and all integer $s\in [0,\tilde \alpha]$
\begin{align}\label{j14}
|e_k''|_s\leq C\delta^2\theta_k^{L_2(s)-1}\Delta_k,
\end{align}
where $L_2(s)=\max(s+\alpha_0-2\alpha+1,(s-\alpha)_++2\alpha_0-2\alpha+2)$.

2) For large enough $\theta_0$ and small enough $\delta$ we have for all $k=0,\dots,n-1$ and all integer $s\in [0,\tilde \alpha]$
\begin{align}\label{j15}
\langle \tilde e_k''\rangle _{s+1}\leq C\delta^2\theta_k^{L_2(s)-1}\Delta_k.
\end{align}
\end{lem}

\begin{proof}
Again we use the formulas \eqref{f5}. By Lemma \ref{e21} we have $|S_{\theta_k}V_k+\tau(I-S_{\theta_k})V_k|_{\alpha_0}\leq K'$ for $\delta$ small enough, so we can apply the estimates \eqref{j9}.  When estimating the right sides of \eqref{j9} we use, for example,
\begin{align}\label{j16}
|(I-S_{\theta_k})V_k|_s\leq C\delta \theta_k^{s-\alpha+1}.
\end{align}
%The  estimate \eqref{j15}  has the same form as \eqref{j14} for the interval of $s$ because of Remark \ref{j6aa} and, of course, %the similarity in the estimates \eqref{j9}.
\end{proof}

In place of Lemma \ref{f11}, the estimate of accumulated errors,  we now have:
\begin{lem}\label{j17}
There exist $\theta_0$ sufficiently large and $\delta_0$ sufficiently small  so that for $0<\delta\leq\delta_0$
\begin{align}\label{j18}
|E_n|_{\tilde \alpha}\leq C\delta^2 \theta_n^{L_2(\tilde \alpha)} \text{ and } \langle\tilde E_n\rangle_{\tilde \alpha+1}\leq C\delta^2\theta_n^{L_2(\tilde \alpha)}.
\end{align}
\end{lem}
\begin{proof}
Since $\tilde \alpha=2\alpha-\alpha_0$, we have $L_2(\tilde \alpha)>0$, so the proof is the same as that of Lemma \ref{f11}.
\end{proof}

The new version of Lemma \ref{f14}, the estimate of $f_n$ and $g_n$, is:

\begin{lem}\label{j19}
There exist $\theta_0$ sufficiently large and $\delta_0$ sufficiently small so that for $s\in [0,\tilde \alpha+1]$, $0<\delta\leq\delta_0$ we have
\begin{align}\label{j25}
\begin{split}
&(a)\;|f_n|_s\leq C\delta^2\theta_n^{L_2(s)-1}\Delta_n\\
&(b)\;\langle g_n\rangle_{s+1}\leq C\delta^2\theta_n^{L_2(s)-1}\Delta_n+C\theta_n^{s-\alpha-2}\langle g\rangle_{\alpha+2}\Delta_n.
\end{split}
\end{align}
\end{lem}

\begin{proof}
Since $s-\tilde\alpha+L_2(\tilde{\alpha})\leq L_2(s)$, the proof of Lemma \ref{f14} can be repeated here.
\end{proof}

\textbf{Induction step. }For $\delta>0$ sufficiently small, the estimate for the linearized system \eqref{j8} applies to \eqref{e7} and gives for $s\in [0,\tilde\alpha]$:
\begin{align}\label{j26}
|\dot V_n|_s+\langle\dot V_n\rangle_{s+1}\leq
C\left[|f_n|_{s+1}+\langle g_n\rangle_{s+2}+\left(|f_n|_{\alpha_0+1}+\langle g_n\rangle_{\alpha_0+2}\right)\left(|S_{\theta_n}V_n|_{s+1}+\langle S_{\theta_n}V_n\rangle_{s+2}\right)\right]
\end{align}
Indeed,  \eqref{j22} implies that for $\delta>0$ small enough, $S_{\theta_n}V_n$ satisfies the requirement \eqref{j7a}.\footnote{Here we use $\alpha_1<\alpha$.   Also, the term $|\eps
\partial_d(S_{\theta_n}V_n)|_{\alpha_1-1}$ is estimated using the equation \eqref{e7}.}
For the terms involving $f_n$ and $g_n$, except $\langle g_n\rangle_{\alpha_0+2}$, we substitute directly into \eqref{j26} the corresponding estimates from Lemma \ref{j19}.  For $\langle g_n\rangle_{\alpha_0+2}$ we have
\begin{align}\label{j27}
\langle g_n\rangle_{\alpha_0+2}\leq C\left(\delta^2\theta_n^{L_2(\alpha_0+1)-1}\Delta_n+\theta_n^{-\alpha-2}\langle g\rangle_{\alpha_0+\alpha+3}\Delta_n\right),
\end{align}
where the last term arises from an estimate like \eqref{f16a} with $s=\alpha_0+2$ and $\alpha$ replaced by $\alpha+\alpha_0+3$.  We also use
\begin{align}\label{j28}
\langle S_{\theta_n}V_n\rangle_{s+2}\leq C\delta \theta_n^{(s+1-\alpha)_++1}
\end{align}
and a similar estimate for $|S_{\theta_n}V_n|_{s+1}$, which follow directly from \eqref{j22}.

Making these substitutions in \eqref{j26} gives for $s\in [0,\tilde\alpha]$:
\begin{align}\label{j29}
\begin{split}
&|\dot V_n|_s+\langle\dot V_n\rangle_{s+1}\leq\\
 &\quad C\left[\delta^2\theta_n^{L_2(s+1)-1}\Delta_n+\theta_n^{s-\alpha-1}\langle g\rangle_{\alpha+2}\Delta_n+\left(\delta^2\theta_n^{L_2(\alpha_0+1)-1}\Delta_n+\theta_n^{-\alpha-2}\langle g\rangle_{\alpha_0+\alpha+3}\Delta_n\right)\delta\theta_n^{(s+1-\alpha)_++1}\right].
\end{split}
\end{align}
For $s\in [0,\tilde \alpha]$ the parameters $\alpha_0$ and $\alpha$ (recall \eqref{j6a}) satisfy:
\begin{align}\label{j30}
\begin{split}
&(a)\;L_2(s+1)\leq s-\alpha\\
&(b)\;L_2(\alpha_0+1)+ (s+1-\alpha)_++1\leq s-\alpha\\
&(c)\;(s+1-\alpha)_+<s.
\end{split}
\end{align}
Thus, we have proved (H$_n$), which is the content of the following lemma.
\begin{lem}[H$_n$]\label{j31}
If $\delta>0$ and $\langle g\rangle_{\alpha+2}/\delta$ are sufficiently small and $\theta_0$ sufficiently large, we have
\begin{align}\label{j32}
\;|\dot V_n|_s+\langle\dot V_n\rangle_{s+1} \leq \delta \theta_n^{s-\alpha-1}\Delta_n \text{ for all integer }s\in [0,\tilde\alpha].
\end{align}
\end{lem}

Still assuming (H$_{n-1}$) we now show:

\begin{lem}\label{k21a}
Suppose $n\geq 1$.  If $\delta>0$ is sufficiently small and $\theta_0$ sufficiently large, we have
\begin{align}\label{k21b}
\begin{split}
&(a)\;|\cL(V_n)|_s\leq \delta \theta_n^{s-\alpha-1}\text{ for all integer }s\in [0,\tilde\alpha]\\
&(b)\; \langle\cB(V_n)-g\rangle_{s+1} \leq \delta \theta_n^{s-\alpha-1}\text{ for all integer }s\in [0,\tilde\alpha].
\end{split}
\end{align}
\end{lem}

\begin{proof}
From \eqref{e13} we have
\begin{align}\label{k22}
\begin{split}
&(a)\;|\cL(V_n)|_s\leq |(I-S_{\theta_{n-1}})E_{n-1}|_s+|e_{n-1}|_s\\
&(b)\;\langle B(V_n)-g\rangle_{s+1}\leq\langle(\tilde S_{\theta_{n-1}}-I)g\rangle_{s+1}+\langle(I-\tilde S_{\theta_{n-1}})\tilde E_{n-1}\rangle_{s+1}+\langle\tilde e_{n-1}\rangle_{s+1}.
\end{split}
\end{align}
Using \eqref{j3} and the above estimates of $E_{n-1}$ and $e_{n-1}$,  we find
\begin{align}\label{k23}
\begin{split}
&|(I-S_{\theta_{n-1}})E_{n-1}|_s\leq C\theta_n^{s-\tilde\alpha}|E_{n-1}|_{\tilde \alpha}\leq C\delta^2\theta_n^{(s-\alpha-1)+(\alpha_0-\alpha)}\\
&|e_{n-1}|_s\leq C\delta^2\theta_n^{L_2(s)-1}\Delta_n,
\end{split}
\end{align}
which imply \eqref{k21b}(a) since $\alpha_0-\alpha<0$ and $L_2(s)<s-\alpha$.

The last two terms on the right in \eqref{k22}(b) are estimated similarly.  To finish we use
\begin{align}\label{k25}
\langle(\tilde S_{\theta_{n-1}}-I)g\rangle_{s+1}\leq C\theta_n^{s+1-\tilde\alpha}\langle g\rangle_{\tilde \alpha}\text{ for }s\leq \tilde\alpha-1,
\end{align}
and observe that $s-\tilde\alpha+1<s-\alpha-1$.
\end{proof}

We now fix $\delta$ and $\theta_0$ as above and check (H$_0$).
\begin{lem}\label{k26}
If $\langle g\rangle_{\alpha+2}$ is small enough, then (H$_0$) holds.
\end{lem}

\begin{proof}
Applying the estimate for the linearized system to
\begin{align}
\begin{split}
&\cL'(0)\dot V_0=0\\
&\cB'(0)\dot V_0=S_{\theta_0}g
\end{split}
\end{align}
we obtain for integer $s\in [0,\tilde\alpha]$:
\begin{align}
|\dot V_0|_s+\langle\dot V_0\rangle_{s+1}\leq C\langle S_{\theta_0}g\rangle_{s+2}\leq C\begin{cases}\theta_0^{s-\alpha}\langle g\rangle_{\alpha+2}, \;s\geq \alpha\\\langle g\rangle_{\alpha+2},\;s<\alpha\end{cases}.
\end{align}
Thus, (H$_0$) holds if $\langle g\rangle_{\alpha+2}$ is small enough.
\end{proof}

We can now complete the proof of Theorem \ref{k8b}.

\begin{proof}[Proof of Theorem \ref{k8b}]
We have
\begin{align}
V_n=V_{n-1}+\dot V_{n-1}=V_0+\sum^{n-1}_{k=0}\dot V_k=\sum^{n-1}_{k=0}\dot V_k.
\end{align}
Let $\mu=\alpha-1$.  Since $\theta_k\sim \sqrt{k}$ we have by (H$_n$)
\begin{align}
\sum_{k=0}^\infty |\dot V_k|_\mu+\sum_{k=0}^\infty \langle \dot V_k\rangle_{\mu+1}\leq \delta \sum_k \theta_k^{-2}\Delta_k\leq C\sum_k k^{-\frac{3}{2}}<\infty.
\end{align}
Thus, for some $V$ as described in Proposition \ref{k8b}, $V_k\to V$ in $E^\mu$ and $V_k|_{x_d=0}\to V|_{x_d=0}$ in  $H^{\mu+1}$ (in fact, uniformly for $0<\eps\leq \eps_0$).
Lemma \ref{k21a} applied with $s=\mu-1$ now implies that $V$ is a solution of the semilinear system \eqref{j5}.
\end{proof}

\newpage
\appendix
\section{A calculus of singular pseudodifferential operators}
\label{calc}

Here we summarize the parts of the singular calculus constructed in \cite{CGW2} that are needed in this article,
and we also prove some additional smoothing properties for some of the operators that appear in the proof of
Propositions \ref{ia3} and \ref{hard}.

\subsection{Symbols}

\emph{\quad} Our singular symbols are built from the following sets of classical symbols.
\begin{defn}\label{n1}
Let $\cO\subset \bR^N$ be an open subset that contains the origin.  For $m\in\bR$ we let $\bfS^m(\cO)$ denote
the class of all functions $\sigma:\cO\times \bR^d\times [1,\infty)\to \bC^M$, $M \ge 1$, such that $\sigma$ is
$C^\infty$ on $\cO \times \bR^d$ and for all compact sets $K\subset \cO$:
\begin{equation*}%\label{n2}
\sup_{v\in K} \, \sup_{\xi'\in\bR^d} \, \sup_{\gamma\geq 1} \, (\gamma^2+|\xi'|^2)^{-(m-|\nu|)/2} \,
|\partial^\alpha_v\partial_{\xi'}^\nu \sigma(v,\xi',\gamma)| \leq C_{\alpha,\nu,K}.
\end{equation*}
\end{defn}

Let ${\mathcal C}^k_b(\bR^d \times \bT)$, $k \in \N$, denote the space of continuous and bounded functions
on $\bR^d \times \bR$ that are $2\pi$-periodic in their last argument, and whose derivatives up to order $k$
are continuous and bounded.

\begin{defn}[Singular symbols]\label{n3}
Let $m\in \bR$, $n\in\bN$, and fix $\beta\in \bR^d\setminus 0$.  We let $S^m_n$ denote the family of functions
$(a_{\eps,\gamma})_{\eps\in (0,1],\gamma\geq 1}$ that are constructed as follows:
\begin{align}\label{n4}
\forall \, (x',\theta_0,\xi',k)\in \bR^d \times \bT \times \bR^d \times \bZ, \;
a_{\eps,\gamma}(x',\theta_0,\xi',k) =\sigma \left(\eps V(x',\theta_0),\xi'+\dfrac{k\, \beta}{\eps},\gamma \right),
\end{align}
where $\sigma \in \bfS^m(\cO)$ and $V\in \cC^n_b(\bR^d\times\bT)$. Below and in the main text we often set
\begin{equation*}%\label{n4a}
X:=\xi'+\dfrac{k\, \beta}{\eps}.
\end{equation*}
\end{defn}

\subsection{Singular pseudodifferential operators}

\emph{\quad} To each symbol $a_{\eps,\gamma}$ as in \eqref{n4}, we associate a singular pseudodifferential
operator $\mathrm{Op}^{\eps,\gamma}(a)$ whose action on Schwartz class functions $u\in\cS(\bR^d\times \bT:\bC^M)$
is defined by
\begin{align}\label{n5}
\mathrm{Op}^{\eps,\gamma}(a)u(x',\theta_0):=\frac{1}{(2\pi)^{d+1}} \, \sum_{k\in\bZ} \int_{\bR^d}
{\rm e}^{i\, x'\cdot\xi'+i \,k\, \theta_0} \, \sigma \left( \eps V(x',\theta_0),\xi'+\dfrac{k\, \beta}{\eps},\gamma \right)
\, \hat u(\xi',k) \, {\rm d}\xi' \, ,
\end{align}
where $\hat u(\xi',k)$ denotes the Fourier transform at $\xi'$ of the $k$-th Fourier coefficient of $u$ with respect to
$\theta_0$. When $a_{\eps,\gamma}$ is defined as in \eqref{n4}, below and in the main text of the article, we will
often write $\sigma(\eps V(x,\theta_0),X,\gamma)$ in place of $a_{\eps,\gamma}(x',\theta_0,\xi',k)$, and $\sigma_D$
in place of $\mathrm{Op}^{\eps,\gamma}(a)$. In particular, we let $\Lambda_D$ denote the singular Fourier multiplier
associated to the function
\begin{equation*}%\label{n7}
\Lambda (X,\gamma) :=(\gamma^2+|X|^2)^{1/2}.
\end{equation*}

When $V(x',x_d,\theta_0)$ depends also on  a normal variable $x_d\geq 0$, we define the associated family of
operators depending on the parameter $x_d$ in the obvious way. The pseudodifferential calculus takes place
only in the tangential directions $(x',\theta_0)$. To discuss mapping properties we first define ``singular" Sobolev
spaces as follows.

\begin{defn}\label{n8}
We let
\begin{equation*}%\label{n9}
H^{s,\eps}(\bR^d \times \bT) :=\left\{ u\in\cS'(\bR^d\times \bT) :
\sum_{k\in\bZ} \int_{\bR^d} (1+|X|^2)^s \, |\hat u(\xi',k)|^2 {\rm d}\xi'<\infty \right\}.
\end{equation*}
This space is equipped with the family of norms\footnote{In this appendix we use $|\cdot|$ instead of
$\langle\cdot\rangle$ in the notation for norms on $\bR^d\times\bT$, but otherwise we retain notation
from the main text.}
\begin{equation*}%\label{n10}
|u|^2_{H^{s,\eps},\gamma} :=\dfrac{1}{(2\pi)^{d+1}} \, \sum_{k\in\bZ} \int_{\bR^d}
(\gamma^2 +|X|^2)^s \, |\hat u(\xi',k)|^2 {\rm d}\xi'.
\end{equation*}
\end{defn}

Observe that for $s$ fixed the space $H^{s,\eps}$ depends on $\eps$ with no obvious inclusion if $\eps_1<\eps_2$.
However, for fixed $\eps>0$ the norms $|\cdot|_{H^{s,\eps},\gamma_1}$ and $|\cdot|_{H^{s,\eps},\gamma_2}$ are
equivalent.

The next Proposition describes some of the mapping properties of these operators. Detailed proofs can be found
in \cite{CGW2}. The constant $C$ is always independent of $\eps\in (0,1]$ and $\gamma\geq 1$, and we denote
the $L^2 (\bR^d\times \bT)$ norm by $|\cdot |_0$ (which corresponds to $s=0$ in Definition \ref{n8}).

\begin{prop}[Mapping properties] \label{n11}
(a) Suppose $\sigma(\eps V(x,\theta_0),X,\gamma)\in S^m_n$, where $n\geq d+1$ and $m\leq 0$.
Then $\sigma_D:L^2(\bR^d\times \bT) \to L^2(\bR^d\times \bT)$ and
\begin{equation*}%\label{n12}
|\sigma_D u|_0 \leq \dfrac{C}{\gamma^{|m|}} \, |u|_0.
\end{equation*}

(b) Suppose $\sigma(\eps V(x,\theta_0),X,\gamma)\in S^m_n$, where $n\geq d+1$ and $m> 0$.
Then $\sigma_D:H^{m,\eps}(\bR^d\times \bT) \to L^2(\bR^d\times \bT)$ and
\begin{equation*}%\label{n13}
|\sigma_D u|_0 \leq C\, |u|_{H^{m,\eps},\gamma}.
\end{equation*}

(c) \emph{[Smoothing property]} Suppose $\sigma(\eps V(x,\theta_0),X,\gamma)\in S^{-1}_n$, where
$n\geq d+2$. Then $\sigma_D:L^2(\bR^d\times \bT) \to H^{1,\eps}(\bR^d\times \bT)$ and
\begin{equation*}%\label{n14}
|\sigma_D u|_{H^{1,\eps},\gamma} \leq C\, |u|_0.
\end{equation*}

(d) Suppose $\sigma(\eps V(x,\theta_0),X,\gamma)\in S^{0}_n$, where $n\geq d+2$. Then
$\sigma_D:H^{1,\eps}(\bR^d\times \bT) \to H^{1,\eps}(\bR^d\times \bT)$ and
\begin{equation*}%\label{n15}
|\sigma_D u|_{H^{1,\eps},\gamma} \leq C\, |u|_{H^{1,\eps},\gamma}.
\end{equation*}
\end{prop}

\textbf{Residual operators.} We sometimes denote by $r_{0,D}$ an operator that maps $L^2(\bR^d\times \bT)
\to L^2(\bR^d\times \bT)$ and satisfies a uniform operator bound
\begin{equation*}%\label{n16}
|r_{0,D}u|_0 \leq C\, |u|_0,
\end{equation*}
even when $r_{0,D}$ is not necessarily defined by a symbol in some class $S^0_n$. Similarly, we sometimes
let $r_{-1,D}$ denote an operator not necessarily associated to a symbol in $S^{-1}_n$ such that
\begin{equation}\label{n17}
|r_{-1,D} u|_{H^{1,\eps},\gamma} \leq C\, |u|_{0}.
\end{equation}
For example, the composition $\sigma_{-1,D} \, \tau_{0,D} =r_{-1,D}$ of an operator of order $-1$ (case {\it (c)}
in Proposition \ref{n11}) with an operator of order $0$ (case {\it (a)} when $m=0$) is of this latter type.

\begin{rem}\label{n16a}
\textup{Observe that a composition of the form $r_{0,D}r_{-1,D}$ is not necessarily an operator of type $r_{-1,D}$,
a fact that is a source of difficulty in the proof of the main linear estimate, Proposition \ref{i5z}. This is the case, for
example, if $r_{0,D}$ is the operator of multiplication by $V(x',\theta_0) \in \cC^1_b(\bR^d\times\bT)$. On the other
hand we have
\begin{equation*}%\label{n16b}
\eps \, V(x',\theta_0) \, r_{-1,D} =r_{-1,D} \, ,
\end{equation*}
and more generally, Proposition \ref{n11}(d) implies that if $\sigma\in S^0_n$, $n\geq d+2$, we have
\begin{align}\label{n16c}
\sigma_D \, r_{-1,D} =r_{-1,D}.
\end{align}}
\end{rem}

\subsection{Adjoints and products}

\emph{\quad} In spite of the fact that singular symbols and their derivatives fail to decay in the classical way in
$\langle \xi',k,\gamma\rangle$, it is possible to construct a crude calculus of singular pseudodifferential operators
with useful formulas for adjoints and products which, in particular, permit G{\aa}rding inequalities to be proved.
This calculus was used repeatedly in the proof of the main linear estimate, Proposition \ref{i5z}.

In the next proposition $\sigma^*$ denotes the conjugate transpose of the $M\times M$ matrix valued symbol
$\sigma$, while $(\sigma_D)^*$ denotes the adjoint operator for the $L^2$ scalar product.

\begin{prop}[Adjoints]\label{n18}
(a) Let $\sigma\in S^0_n$, where $n\geq 2d+3$. Then $(\sigma_D)^*-(\sigma^*)_D=r_{-1,D}$.

(b) Let $\sigma\in S^1_n$, where $n\geq 3d+4$. Then $(\sigma_D)^*-(\sigma^*)_D=r_{0,D}$.
\end{prop}

\begin{prop}[Products]\label{n19}
(a) Suppose $\sigma$ and $\tau$ lie in $S^0_n$, where $n\geq 2d+3$. Then
\begin{equation*}%\label{n20}
\sigma_D \, \tau_D -(\sigma \, \tau)_D =r_{-1,D}.
\end{equation*}

(b) Suppose $\sigma\in S^1_n$, $\tau\in S^0_n$ or $\sigma\in S^0_n$, $\tau\in S^1_n$, where $n\geq 3d+4$.
Then
\begin{equation*}%\label{n21}
\sigma_D \, \tau_D -(\sigma \, \tau)_D =r_{0,D}.
\end{equation*}

(c) Suppose $\sigma\in S^{-1}_n$, $\tau\in S^1_n$, where $n\geq 3d+4$. Then
\begin{align}\label{n22}
\sigma_D \, \tau_D -(\sigma \, \tau)_D =r_{-1,D}.
\end{align}
\end{prop}

\begin{rem}\label{n26}
\textup{Observe that when $\tau=\tau(X,\gamma)$ is independent of $\eps V(x,\theta_0)$ and thus gives rise to
a Fourier multiplier, the composition $\sigma_D \, \tau_D=(\sigma \, \tau)_D$ is exact, a fact that has been used
several times in the proof of Proposition \ref{i5z}.}
\end{rem}

The equality \eqref{n22} does not hold in general when  $\sigma\in S^{1}_n$ and $\tau\in S^{-1}_n$, and this
is one of the main difficulties in the proof of Proposition \eqref{hard}. There is however one frequency regime in
which it becomes possible to prove the analogue  of equality \eqref{n22} for $(1,-1)$ compositions, see Proposition 
\ref{n31a} below.

\subsection{Extended calculus}\label{extended}

\emph{\quad} In the proof of Proposition \ref{i5z} and of intermediate results, see in particular the decomposition
\eqref{ia2}, we use a slight extension of the singular calculus. For given parameters $0<\delta_1<\delta_2<1$, we
choose a cutoff $\chi^e (\xi',\frac{k\, \beta}{\eps},\gamma)$ such that
\begin{align}\label{n31}
\begin{split}
&0\leq \chi^e \leq 1\, ,\\
&\chi^e \left( \xi',\dfrac{k\, \beta}{\eps},\gamma \right) =1 \text{ on } \left\{
(\gamma^2 +|\xi'|^2)^{1/2} \leq \delta_1 \, \left| \dfrac{k\, \beta}{\eps} \right| \right\} \, ,\\
&\mathrm{supp }\chi^e \subset \left\{ (\gamma^2 +|\xi'|^2)^{1/2} \leq \delta_2 \, \left| \dfrac{k\, \beta}{\eps} \right|
\right\} \, ,
\end{split}
\end{align}
and define a corresponding Fourier multiplier $\chi_D$ in the extended calculus by the formula \eqref{n5} with
$\chi^e (\xi',\frac{k\, \beta}{\eps},\gamma)$ in place of $\sigma(\eps V,X,\gamma)$. Composition laws involving
such operators are proved in \cite{CGW2}, but here we need only the fact that part {\it (a)} of Proposition \ref{n19}
holds when either $\sigma$ or $\tau$ is replaced by an extended cutoff $\chi^e$.

The following Proposition is essential for treating some of the remainder terms that arise in the proof of Proposition
\ref{ia3}. Part {\it (b)} of Proposition \ref{n31a} is used to handle the remainder in the $(1,-1)$ composition in
\eqref{ii7}(a), and part {\it (c)} is used for that in \eqref{ib2}(a).

\begin{prop}\label{n31a}
(a)\;Suppose $V \in \cC^1_b(\bR^d\times \bT)$. Let $f(V)$ be any smooth function, suppose $a_{-1}(X,\gamma)$ is
of order $-1$, and let $\chi^e$ be a cutoff in the extended calculus satisfying \eqref{n31}. Then we have
\begin{equation*}%\label{n31b}
f(V)\, a_{-1,D} \, \chi^e_D =r_{-1,D}.
\end{equation*}

(b)\;Suppose $V\in \cC^n_b(\bR^d\times \bT)$ where $n\geq 2d+3$.  Let $c(\eps V)$ be any smooth function and
suppose $\alpha_1(X,\gamma)\in S^1_n$ and $a_{-1}(X,\gamma)\in S^{-1}_n$. Let also $\chi^e$ be a cutoff in the
extended calculus satisfying \eqref{n31}. Suppose also that $\alpha_1$ is smooth and homogeneous degree one
in $(X,\gamma)$. Then we have
\begin{equation*}%\label{n31d}
[\alpha_{1,D},c(\eps V)] \, a_{-1,D} \, \chi^e_D =r_{-1,D}.
\end{equation*}

(c)\;Suppose $V\in \cC^n_b(\bR^d\times \bT)$ where $n\geq 3d+5$. Let $c(\eps V)$ be any smooth function and
suppose $\alpha_1(X,\gamma)\in S^1_n$. Then we have
\begin{equation*}%\label{n31e}
[\nabla_{x'},[\alpha_{1,D},c(\eps V)]]=r_{0,D}.
\end{equation*}
\end{prop}

\begin{proof} {\it (a)} There is first an obvious $L^2$ bound:
\begin{equation*}
\gamma \, |f(V)\, a_{-1,D} \, \chi^e_D u|_0 \le \gamma \, |f(V)|_{L^\infty} \, |a_{-1,D} \, \chi^e_D u|_0
\le C \, |u|_0 \, .
\end{equation*}
Then we need to estimate the singular derivatives $(\partial_{x_j}+\beta_j\, \partial_{\theta_0}/\eps)$, $j=0,\dots,d-1$,
of the product $f(V)\, a_{-1,D} \, \chi^e_D u$. When the singular derivatives is applied to the term $a_{-1,D} \, \chi^e_D
u$, we use the bound $|X| \, |a_{-1}(X,\gamma)| \le C$ to obtain a uniform $L^2$ bound, and it thus only remains to
look at the term
\begin{equation*}
\left[ \left( \partial_{x_j} +\dfrac{\beta_j}{\eps}\, \partial_{\theta_0} \right) \, f(V) \right] \, a_{-1,D} \, \chi^e_D u \, .
\end{equation*}
Since $V \in \cC^1_b$, the only difficult term is $\partial_{\theta_0} \, f(V)/\eps$. For this term, we use the fact
that on the support of $\chi^e (\xi',\frac{k\, \beta}{\eps},\gamma)$ we have
\begin{equation*}%\label{n32}
|X,\gamma|^{-1} \leq C\, \eps \, ,
\end{equation*}
and therefore $|a_{-1,D} \, \chi^e_D u|_0 \le C\, \eps \, |u|_0$.

{\it (b)} It follows easily from part {\it (a)} that
\begin{align}\label{n35}
\left[ \partial_{x_j}+\beta_j\, \dfrac{\partial_{\theta_0}}{\eps},c(\eps U) \right] \, a_{-1,D}\, \chi^e_D
=r_{-1,D}\, , \quad j=0,\dots,d-1.
\end{align}
We now reduce to this case by writing
\begin{equation*}%\label{n36}
\alpha_1(X,\gamma)=\sum_{j=0}^{d-1}(\partial_{X_j}\alpha_)\, X_j +(\partial_\gamma\alpha_1)\, \gamma
=\sum_{j=0}^{d-1} b_{0,j}(X,\gamma)\, (iX_j) +b_{0,\gamma}(X,\gamma) \, \gamma.
\end{equation*}
With $b_0=b_{0,j}$, for example, we have
\begin{multline*}%\label{n37}
\left[b_{0,D} \, \left( \partial_{x_j}+\beta_j \, \dfrac{\partial_{\theta_0}}{\eps} \right),c(\eps U) \right] \,
a_{-1,D} \, \chi^e_D \\
=[b_{0,D},c(\eps U)] \, \left( \partial_{x_j}+\beta_j \, \dfrac{\partial_{\theta_0}}{\eps} \right) \, a_{-1,D} \, \chi^e_D
+b_{0,D} \, \left[ \partial_{x_j} +\beta_j \, \dfrac{\partial_{\theta_0}}{\eps},c(\eps U) \right] \, a_{-1,D} \, \chi^e_D.
\end{multline*}
The first term on the right is clearly of type $r_{-1,D}$ as a product of an $r_{-1,D}$ operator on the left by an
$r_{0,D}$ operator on the right. By \eqref{n35} and  \eqref{n16c} the same is true of the second term. We leave
to the reader the verification that the commutator with the last term $b_{0,\gamma}(X,\gamma) \, \gamma$ gives
also rise to an $r_{-1,D}$ remainder (this is even easier).

{\it (c)} The claim follows from the relation
\begin{equation*}
[\nabla_{x'},[\alpha_{1,D},c(\eps U)]] =[\alpha_{1,D},(\nabla_{x'} c(\eps U))] \, ,
\end{equation*}
and from Proposition \ref{n19} {\it (b)}.
%obtain a formula for $[\alpha_{1,D},c(\eps U)]$ we write $c(\eps U)=c(0)+\eps U d(\eps U)$ and express the
%$(1,0)$ composition $\alpha_{1,D} \circ (\eps Ud(\eps U))$ as
%\begin{align}\label{n33}
%\begin{split}
%&\alpha_{1,D}\circ (\eps U d(\eps U))=\widetilde {\mathrm{Op}}(\eps U d(\eps U(y,\omega))\alpha_1(X,\gamma))=\\
%&\qquad \eps U d(\eps U)\alpha_{1,D}+\widetilde {\mathrm{Op}}(r_1)+\widetilde {\mathrm{Op}}(r_2),
%\end{split}
%\end{align}
%where $r_i(x,\theta,y,\omega,\xi,k)$, $i=1,2$ \;are amplitudes given by the formulas in the proof of Theorem 3 of
%\cite{CGW2} and, for example,
%$\widetilde {\mathrm{Op}}(r_1)$ is the singular operator whose action on $u(y,\omega)$ is given by
%\begin{align}\label{n33z}
%\widetilde {\mathrm{Op}}(r_1) u(x,\theta)=\frac{1}{(2\pi)^{d+1}}\sum_k\int e^{i(x-y)\xi+i(\theta-\omega)k}
%r_1(x,\theta,y,\omega,\xi,k)u(y,\omega)dy d\omega d\xi.
%\end{align}
%Thus,
%\begin{align}\label{n34}
%[\alpha_{1,D},c(\eps U)]=\widetilde {\mathrm{Op}}(r_1)+\widetilde {\mathrm{Op}}(r_2).
%\end{align}
%An explicit integral formula for $[\nabla_{x'},[\alpha_{1,D},c(\eps U)]]$ similar to \eqref{n33z} is readily derived
%from the integral formula for the right side of \eqref{n34} by differentiating under the integral sign. An application
%of the $L^2$ boundedness criterion in Theorem  2 of \cite{CGW2} now yields the result.  Here the factor of
%$\eps$ in $\eps U d(\eps U)$ is needed to show
%\begin{align}\label{n40}
%[\nabla_{x'},\widetilde {\mathrm{Op}}(r_2)]=r_{0,D}.
%\end{align}
\end{proof}

\begin{rem}\label{n31c}
\textup{Since $|\eps \, \partial_d U|_{C^{0,M_0-1}}\leq K$ by the assumption in Proposition \ref{ia3}, the result
of Proposition \ref{n31a} part {\it (a)} yields $\partial_d Q_{-1,D}=r_{-1,D}$ in \eqref{ii7}(a). Similarly, Proposition
\ref{n31a} part {\it (b)} enables us to obtain \eqref{ii7}(b).}
\end{rem}

In the proof of Proposition \ref{i5z} we use the following localized G{\aa}rding inequality for zero order operators.
As before we write $\zeta=(\xi',\gamma)$.

\begin{prop}[G{\aa}rding's inequality] \label{n27}
Let $\sigma(v,\zeta)\in \bfS^0(\cO)$ and $\chi(v,\zeta)\in\bfS^0(\cO)$ be such that
\begin{equation*}%\label{n28}
\re \sigma(v,\zeta) \geq c>0
\end{equation*}
on a conic neighborhood of $\mathrm{supp }\chi$. Provided the corresponding singular symbols lie in $S^0_n$,
$n\geq 2d+2$, we have
\begin{equation*}%\label{n29}
\re (\sigma_D\, \chi_D u,\chi_Du) \geq \dfrac{c}{2} \, |\chi_D u|_0^2 -\dfrac{C}{\gamma} \, |u|_0^2.
\end{equation*}
\end{prop}

\newpage
\section{An example derived from the Euler equations}
\label{exeuler}

In this appendix we explain in a particular example how one can derive a single nonlocal nonlinear equation
that governs the evolution of the amplitude function $a$, which itself determines the leading profile ${\mathcal V}^0$,
see Proposition \ref{7a}. In the process we provide explicit constructions of a number of the objects that appeared
in our earlier discussion of approximate solutions.

As in \cite{CG}, we consider the linearized Euler equations in two space dimensions to which we add a
nonlinear zero order term (we slightly change notation compared with the introduction). More precisely,
we consider the following system
\begin{equation}
\label{eulerlin}
\begin{cases}
\dt V^\eps +A_1 \, \dxun V^\eps +A_2 \, \dxde V^\eps +{\bf D}(V^\eps,V^\eps) = 0 \, ,&
(t,x_1,x_2) \in \, ]-\infty,T] \times \R^2_+ \, ,\\
B \, V^\eps|_{x_2=0} +\Psi (V^\eps,V^\eps)|_{x_2=0} = \eps^2 \, G(t,x_1,\phi_0(t,x_1)/\eps) \, ,&
(t,x_1) \in \, ]-\infty,T] \times \R \, ,\\
V^\eps|_{t<0} = 0 \, ,&
\end{cases}
\end{equation}
where the $3\times 3$ matrices $A_1,A_2$ are given by
\begin{equation*}
A_1 := \begin{pmatrix}
0 & -v & 0 \\
-c^2/v & 0 & 0 \\
0 & 0 & 0 \end{pmatrix} \, ,\quad
A_2 := \begin{pmatrix}
u & 0 & -v \\
0 & u & 0 \\
-c^2/v & 0 & u \end{pmatrix} \, ,
\end{equation*}
and the parameters $v,u,c$ are chosen so that
\begin{equation*}
v>0 \, ,\quad 0<u<c \, .
\end{equation*}
The latter assumption corresponds to the linearization of the Euler equations at a given specific volume $v$ with
corresponding sound speed $c$, and a {\em subsonic incoming} velocity $(0,u)$ (observe the difference with
\cite{CG}). We also assume that ${\bf D}$ in \eqref{eulerlin} is a symmetric bilinear operator from $\R^3 \times
\R^3$ into $\R^3$, and that $\Psi$ is a bilinear operator from $\R^3 \times \R^3$ into $\R^2$ (why we choose
$\R^2$ is explained below).

For such parameters, the operator $\dt +A_1\, \dxun +A_2 \, \dxde$ in \eqref{eulerlin} is strictly hyperbolic with
three characteristic speeds
\begin{equation*}
\lambda_1(\xi_1,\xi_2) := u \, \xi_2 -c \, \sqrt{\xi_1^2+\xi_2^2} \, ,\quad
\lambda_2(\xi_1,\xi_2) := u \, \xi_2 \, ,\quad
\lambda_3(\xi_1,\xi_2) := u \, \xi_2 +c \, \sqrt{\xi_1^2+\xi_2^2} \, .
\end{equation*}
There are two incoming characteristics and one outgoing characteristic, so $B$ should be a $2 \times 3$ matrix
of maximal rank. The choice of $B$ is made precise below. Of course, the source term $G$ in \eqref{eulerlin} is
valued in $\R^2$. We assume moreover that $G$ is $1$-periodic and has mean zero with respect to its third variable
$\theta_0$. We choose a planar phase $\phi_0$ for the oscillations of the boundary source term in \eqref{eulerlin}:
\begin{equation*}
\phi_0 (t,x_1) := \underline{\tau} \, t +\underline{\eta} \, x_1 \, ,\quad \tauetabar \neq (0,0) \, .
\end{equation*}
The hyperbolic region ${\mathcal H}$ can be explicitly computed and is given by
\begin{equation*}
{\mathcal H} = \left\{ (\tau,\eta) \in \R \times \R \, / \, |\tau| > \sqrt{c^2-u^2} \, |\eta| \right\} \, .
\end{equation*}
For concreteness, we fix from now on the parameters $\tauetabar$ such that $\underline{\eta}>0$ and
$\underline{\tau} = c\, \underline{\eta}$. In this way, we have $\tauetabar \in {\mathcal H}$.

We determine the planar characteristic phases whose trace on $\{ x_2=0\}$ equals $\phi_0$. This amounts to finding
the roots $\omega$ of the dispersion relation
\begin{equation*}
\det \Big[ \underline{\tau} \, I +\underline{\eta}\, A_1 +\omega \, A_2 \Big] = 0 \, .
\end{equation*}
We obtain three real roots that are given by
\begin{equation*}
\underline{\omega}_1 := \dfrac{2 \, M}{1-M^2} \, \underline{\eta} \, ,\quad \underline{\omega}_2 := 0 \, ,\quad
\underline{\omega}_3 := -\dfrac{1}{M} \, \underline{\eta} \, ,\quad M:=\dfrac{u}{c} \in \, ]0,1[ \, .
\end{equation*}
The associated (real) phases are $\phi_i (t,x) := \phi_0(t,x_1)+\underline{\omega}_i \, x_2$, $i=1,2,3$.
The relations
\begin{equation*}
\underline{\tau} + \lambda_1 (\underline{\eta},\underline{\omega}_1) =
\underline{\tau} + \lambda_1 (\underline{\eta},\underline{\omega}_2) =
\underline{\tau} + \lambda_2 (\underline{\eta},\underline{\omega}_3) = 0 \, ,
\end{equation*}
yield the group velocity ${\bf v}_i$ associated with each phase $\phi_i$:
\begin{equation*}
{\bf v}_1 :=\dfrac{1-M^2}{1+M^2} \, \begin{pmatrix}
-c \\
-u \end{pmatrix} \, ,\quad {\bf v}_2 :=\begin{pmatrix}
-c \\
u \end{pmatrix} \, ,\quad {\bf v}_3 :=\begin{pmatrix}
0 \\
u \end{pmatrix} \, .
\end{equation*}
Hence the phase $\phi_1$ is outgoing while $\phi_2,\phi_3$ are incoming. With the notation of the introduction,
we can also compute
\begin{align*}
r_1 &:=\begin{pmatrix}
\dfrac{1+M^2}{1-M^2} \, v\\
c \\
\dfrac{2 \, M \, c}{1-M^2} \end{pmatrix} \, ,\quad & r_2 :=\begin{pmatrix}
v \\
c \\
0 \end{pmatrix} \, ,\quad r_3 &:=\begin{pmatrix}
0 \\
c \\
u \end{pmatrix} \, ,\\
\ell_1 &:=\dfrac{1-M^2}{2\, (1+M^2)} \, \begin{pmatrix}
1/v\\
-1/c \\
1/u \end{pmatrix} \, ,\quad & \ell_2 :=\dfrac{1}{2} \, \begin{pmatrix}
1/v \\
1/c \\
-1/u \end{pmatrix} \, ,\quad \ell_3 &:=\dfrac{1}{1+M^2} \, \begin{pmatrix}
-1/v \\
1/c \\
M/c \end{pmatrix} \, ,
\end{align*}
from which one can obtain the expression of the projectors $P_1,P_2,P_3$ as well as the expression of the partial
inverses $R_1,R_2,R_3$. The stable subspace at the frequency $\tauetabar$ is spanned by the vectors $r_2,r_3$.
The matrix $B$ in \eqref{eulerlin} is chosen as
\begin{equation*}
B := \begin{pmatrix}
0 & v & 0 \\
u & 0 & v \end{pmatrix} \, ,
\end{equation*}
so that we can choose $e := r_2-r_3$ as the vector that spans $\ker B \, \cap \, \E^s \tauetabar$. The reader
can check that all our weak stability assumptions are satisfied with this particular choice of boundary conditions.
(We skip the details that are just slightly more complicated than those in \cite{CG}.) The one-dimensional space
$B \, \E^s \tauetabar$ can be written as the orthogonal of the vector $b:=(u,-c)^T$.

The leading profile ${\mathcal V}^0$ and the corrector ${\mathcal V}^1$ satisfy
\begin{equation*}
{\mathcal V}^0 ={\mathcal V}^0_{inc} = \sigma_2(t,x,\theta_2) \, r_2 +\sigma_3 (t,x,\theta_3) \, r_3 \, ,\quad
{\mathcal V}^1_{out} =\tau_1 (t,x,\theta_1) \, r_1 \, .
\end{equation*}
Moreover, there holds
\begin{equation*}
{\mathcal V}^0 (t,x_1,0,\theta_0,\theta_0,\theta_0) =a(t,x_1,\theta_0) \, e =a(t,x_1,\theta_0) \, (r_2-r_3) \, ,
\end{equation*}
where the scalar function $a$ is $1$-periodic with respect to $\theta_0$ and has mean $0$. The Fourier coefficients of
$a$ are denoted $a_k$, $k \in \Z$, where $a_0$ equals $0$ for all time $t$. Since the functions $\sigma_2,\sigma_3$
satisfy the transport equations\footnote{Observe that there is no zero order term in the transport equations because the
zero order term in \eqref{eulerlin} has only a quadratic part. This choice has been made for the sake of simplicity.}
\begin{equation*}
\dt \sigma_2 +{\bf v}_2 \cdot \nabla_x \sigma_2 =\dt \sigma_3 +{\bf v}_3 \cdot \nabla_x \sigma_3 = 0 \, ,
\end{equation*}
and vanish for $t<0$, we obtain the expressions
\begin{equation}
\label{sigma23}
\sigma_2 (t,x,\theta_2) =a \left( t-\dfrac{x_2}{u}, x_1+\dfrac{x_2}{M},\theta_2 \right) \, ,\quad
\sigma_3 (t,x,\theta_3) =-a \left( t-\dfrac{x_2}{u}, x_1,\theta_3 \right) \, .
\end{equation}
To compute ${\mathcal V}^1_{out}$, we must solve
\begin{equation}
\label{V11}
E_{out} \, \left( L(\partial) \, {\mathcal V}^1_{out} +A_2^{-1} \, {\bf D} ({\mathcal V}^0_{inc},{\mathcal V}^0_{inc}) \right)
=0 \, ,\quad (\text{\rm here } E_{out} =E_1) \, ,
\end{equation}
and we thus need to determine the resonances between the phases. A simple calculation shows that there is a
nontrivial $n \in \Z^3$ satisfying $n_1 \, \phi_1 =n_2 \, \phi_2 +n_3 \, \phi_3$ if and only if $M^2$ is a rational number.
We thus assume this to be the case from now on. The resonance between the phases reads
\begin{equation*}
n_1 := q \, ,\quad n_2 := p+q \, ,\quad n_3 :=-p \, ,\quad \text{\rm with } \dfrac{2 \, M^2}{1-M^2} =\dfrac{p}{q} \, ,
\end{equation*}
and it is understood that $p,q$ are both positive and have no common divisor (for instance $p=q=1$ when $M$ equals
$1/\sqrt{3}$). Expanding the quadratic term ${\bf D} ({\mathcal V}^0_{inc},{\mathcal V}^0_{inc})$ in Fourier series, and
using the relation
\begin{equation*}
{\mathcal C}_1 = \Z \, \begin{pmatrix}
1 \\ 0 \\ 0 \end{pmatrix} \cup \Z \, \begin{pmatrix}
0 \\ n_2 \\ n_3 \end{pmatrix} \, ,
\end{equation*}
we obtain (use the expressions \eqref{sigma23})
\begin{multline*}
E_1 \, \left( A_2^{-1} \, {\bf D} ({\mathcal V}^0_{inc},{\mathcal V}^0_{inc}) \right) =-2 \, \sum_{k \in \Z}
a_{k(p+q)} \left( t-\dfrac{x_2}{u}, x_1+\dfrac{x_2}{M} \right) \, a_{-kp} \left( t-\dfrac{x_2}{u}, x_1 \right) \,
{\rm e}^{2i\pi k q \theta_1} \\
P_1 \, A_2^{-1} \, {\bf D} (r_2,r_3) \, .
\end{multline*}
In terms of the interaction integral, we obtain the expression
\begin{multline*}
E_1 \, \left( A_2^{-1} \, {\bf D} ({\mathcal V}^0_{inc},{\mathcal V}^0_{inc}) \right) \\
= -2\, \int_0^1
(a)_{n_2} \left( t-\dfrac{x_2}{u}, x_1+\dfrac{x_2}{M}, \dfrac{n_1}{n_2} \, \theta_1 -\dfrac{n_3}{n_2} \, \theta_3 \right)
\, a \left( t-\dfrac{x_2}{u}, x_1,\theta_3 \right) \, {\rm d}\theta_3 \, P_1 \, A_2^{-1} \, {\bf D} (r_2,r_3) \, ,
\end{multline*}
where $(a)_{n_2}$ still denotes the action of $a$ under the preparation map that retains only Fourier coefficients
that are multiples of $n_2$. Consequently \eqref{V11} reads
\begin{multline}
\label{V12}
\left( \dt -\dfrac{1-M^2}{1+M^2} \, c \, \dxun -\dfrac{1-M^2}{1+M^2} \, u \, \dxde \right) \, \tau_1 \\
= {\bf d} \, \int_0^1
(a)_{n_2} \left( t-\dfrac{x_2}{u}, x_1+\dfrac{x_2}{M}, \dfrac{n_1}{n_2} \, \theta_1 -\dfrac{n_3}{n_2} \, \theta_3 \right)
\, a \left( t-\dfrac{x_2}{u}, x_1,\theta_3 \right) \, {\rm d}\theta_3 \, ,
\end{multline}
with
\begin{equation*}
{\bf d} := -2 \, u \, \dfrac{1-M^2}{1+M^2} \, \ell_1 \cdot A_2^{-1} \, {\bf D} (r_2,r_3) \, .
\end{equation*}
The transport equation \eqref{V12} is solved by integrating along the characteristics and we obtain the expression
\begin{multline}
\label{tau1}
\! \tau_1(t,x_1,0,\theta_1) ={\bf d} \int_0^t \! \! \int_0^1 (a)_{n_2} \left( \dfrac{2\, s -(1-M^2)\, t}{1+M^2},
x_1+2\, c\, \dfrac{1-M^2}{1+M^2} \, (t-s), \dfrac{n_1}{n_2} \, \theta_1 -\dfrac{n_3}{n_2} \, \theta_3 \right) \\
a \left( \dfrac{2\, s -(1-M^2)\, t}{1+M^2}, x_1+c\, \dfrac{1-M^2}{1+M^2} \, (t-s), \theta_3 \right)
\, {\rm d}\theta_3 \, {\rm d}s \, .
\end{multline}
The Fourier series expansion of $\tau_1$ reads
\begin{multline*}
\tau_1(t,x_1,0,\theta_1) ={\bf d} \, \sum_{k \in \Z} \int_0^t
a_{k(p+q)} \left( \dfrac{2\, s -(1-M^2)\, t}{1+M^2}, x_1+2\, c\, \dfrac{1-M^2}{1+M^2} \, (t-s) \right) \\
a_{-kp} \left( \dfrac{2\, s -(1-M^2)\, t}{1+M^2}, x_1+c\, \dfrac{1-M^2}{1+M^2} \, (t-s) \right) \, {\rm d}s \,
{\rm e}^{2i\pi k q \theta_1} \, .
\end{multline*}

The equation governing the amplitude $a$ reads
\begin{equation*}
b \cdot \left( (a^2)^* \, \Psi(e,e) +\tau_1|_{x_2=0} \, B\, r_1 -B \, R \, (L(\partial) \, {\mathcal V}^0_{inc})|_{x_2=0}
\right) =b \cdot G \, ,
\end{equation*}
where functions are evaluated at $x_2=0$ and $\theta_1=\theta_2=\theta_3=\theta_0$. Since we already have the
expression of $\tau_1$ in terms of $a$, the only task left is to compute the trace of the term $B \, R \, (L(\partial) \,
{\mathcal V}^0_{inc})$. Recalling that $R_2 \, r_2 =R_3 \, r_3 =0$, we have
\begin{equation*}
B \, R \, (L(\partial) \, {\mathcal V}^0_{inc})|_{x_2=0} =(B\, R_2 \, A_2^{-1} \, r_2 +B\, R_3 \, A_2^{-1} \, r_3) \, \dt \mathfrak{a}
+(B\, R_2 \, A_2^{-1} \, A_1 \, r_2 +B\, R_3 \, A_2^{-1} \, A_1 \, r_3) \, \dxun \mathfrak{a} \, ,
\end{equation*}
with $\mathfrak{a}$ the unique primitive function of $a$ with zero mean. Using the expressions of $R_2,R_3$ in terms of
the projectors $P_1,P_2,P_3$, which themselves can be obtained from the vectors $r_i,\ell_i$, we get
\begin{align*}
& b \cdot (B\, R_2 \, A_2^{-1} \, r_2 +B\, R_3 \, A_2^{-1} \, r_3) =-\dfrac{u\, v \, (1+M^2)}{M^2 \, \underline{\eta}} \, ,\\
& b \cdot (B\, R_2 \, A_2^{-1} \, A_1 \, r_2 +B\, R_3 \, A_2^{-1} \, A_1 \, r_3)
=\dfrac{u\, c\, v \, (1+M^2)}{M^2 \, \underline{\eta}} \, .
\end{align*}
The fact that both quantities are proportional one to the other with a factor $-c$ comes from a general fact, see
\cite[Lemma 5.1]{CG}.

The function $a$ should therefore satisfy the amplitude equation
\begin{equation*}
\dfrac{u\, v \, (1+M^2)}{M^2 \, \underline{\eta}} \, \big( \dt \mathfrak{a} -c \, \dxun \mathfrak{a} \big)
+b \cdot \Psi(e,e) \, (a^2)^* +b \cdot B\, r_1 \, \tau_1|_{x_2=0} =b \cdot G \, ,
\end{equation*}
or equivalently
\begin{equation}
\label{equationa}
\dfrac{u\, v \, (1+M^2)}{M^2 \, \underline{\eta}} \, \big( \dt a -c \, \dxun a \big)
+b \cdot \Psi(e,e) \, \partial_{\theta_0} (a^2) +b \cdot B\, r_1 \, \partial_{\theta_0}  \tau_1|_{x_d=0}
=b \cdot \partial_{\theta_0} G \, .
\end{equation}
Let us define the two constants
\begin{equation*}
\alpha_1 :=\dfrac{M^2 \, \underline{\eta}}{u\, v \, (1+M^2)} \, b \cdot \Psi(e,e) \, ,\quad
\alpha_2 :=\dfrac{4\, u\, c \, M^2 \, \underline{\eta}}{1+M^2} \, \ell_1 \cdot A_2^{-1} \, {\bf D} (r_2,r_3) \, .
\end{equation*}
Then equation \eqref{equationa} reads
\begin{equation*}
\dt a -c \, \dxun a +\alpha_1 \, \partial_{\theta_0} (a^2) +\alpha_2 \, \partial_{\theta_0}  \dfrac{\tau_1}{{\bf d}}|_{x_2=0}
=\dfrac{M^2 \, \underline{\eta}}{u\, v \, (1+M^2)} \, b \cdot \partial_{\theta_0} G \, ,
\end{equation*}
where the derivative $\partial_{\theta_0} \tau_1/{\bf d}|_{x_2=0}$ is computed from the relation \eqref{tau1}:
\begin{multline*}
\partial_{\theta_0} \dfrac{\tau_1}{{\bf d}}|_{x_2=0} =\dfrac{n_1}{n_2} \, \int_0^t \! \! \int_0^1 (\partial_{\theta_0} a)_{n_2}
\left( \dfrac{2\, s -(1-M^2)\, t}{1+M^2}, x_1+2\, c\, \dfrac{1-M^2}{1+M^2} \, (t-s),
\dfrac{n_1}{n_2} \, \theta_0 -\dfrac{n_3}{n_2} \, \Theta \right) \\
a \left( \dfrac{2\, s -(1-M^2)\, t}{1+M^2}, x_1+c\, \dfrac{1-M^2}{1+M^2} \, (t-s), \Theta \right) \, {\rm d}\Theta \, {\rm d}s \, .
\end{multline*}

In terms of the Fourier coefficients $a_k$, the latter equation is seen to be equivalent to the infinite system of transport
equations
\begin{equation*}
\dt a_k -c \, \dxun a_k +2\, i\, \pi \, k \, \alpha_1 \, \sum_{k' \in \Z} a_{k'} \, a_{k-k'}
=2\, i\, \pi \, k \, \dfrac{M^2 \, \underline{\eta}}{u\, v \, (1+M^2)} \, b \cdot G_k \, ,\quad k \not \in q \, \Z \, ,
\end{equation*}
and
\begin{multline*}
\dt a_{kq} -c \, \dxun a_{kq} +2\, i\, \pi \, k \, q \, \alpha_1 \, \sum_{k' \in \Z} a_{k'} \, a_{kq-k'} \\
+2\, i\, \pi \, k \, q \, \alpha_2 \, \int_0^t
a_{k(p+q)} \left( \dfrac{2\, s -(1-M^2)\, t}{1+M^2}, x_1+2\, c\, \dfrac{1-M^2}{1+M^2} \, (t-s) \right) \\
a_{-kp} \left( \dfrac{2\, s -(1-M^2)\, t}{1+M^2}, x_1+c\, \dfrac{1-M^2}{1+M^2} \, (t-s) \right) \, {\rm d}s
=2\, i\, \pi \, k \, q \, \dfrac{M^2 \, \underline{\eta}}{u\, v \, (1+M^2)} \, b \cdot G_{kq} \, .
\end{multline*}
We recall that the coefficient $a_0$ vanishes.

In the special case $M=1/\sqrt{3}$, the above system reduces to
\begin{multline*}
\dt a_k -c \, \dxun a_k +2\, i\, \pi \, k \, \alpha_1 \, \sum_{k' \in \Z} a_{k'} \, a_{k-k'} \\
+2\, i\, \pi \, k \, \alpha_2 \, \int_0^t
a_{2k} \left( \dfrac{3\, s-t}{2}, x_1+c \, (t-s) \right) \, a_{-k} \left( \dfrac{3\, s-t}{2}, x_1+\dfrac{c}{2} \, (t-s) \right) \, {\rm d}s \\
=2\, i\, \pi \, k \, \dfrac{\underline{\eta}}{4\, u\, v} \, b \cdot G_k \, ,\quad k \in \Z \, ,
\end{multline*}
with parameters $\alpha_1,\alpha_2$ computed from the nonlinearities ${\bf D},\Psi$ in \eqref{eulerlin}:
\begin{equation*}
\alpha_1 :=\dfrac{\underline{\eta}}{4\, u\, v} \, b \cdot \Psi(e,e) \, ,\quad
\alpha_2 :=u\, c \, \underline{\eta} \, \ell_1 \cdot A_2^{-1} \, {\bf D} (r_2,r_3) \, .
\end{equation*}

\bibliographystyle{alpha}
\bibliography{CGWII}
\end{document}